\documentclass[11pt,a4paper]{article}

\usepackage[utf8]{inputenc}  
\usepackage[T1]{fontenc}  
\usepackage[english]{babel}
\usepackage{tikz}
\usetikzlibrary{positioning,chains,fit,shapes,calc}

\usepackage{geometry}
\usepackage{amsmath}
\usepackage{amsfonts}
\usepackage{amssymb}
\usepackage{amsthm}
\usepackage{stmaryrd}
\usepackage{dsfont}
\usepackage{bbm}
\usepackage{yhmath}
\usepackage{mathtools}

\newtheorem{thm}{Theorem}[section]
\newtheorem{prop}[thm]{Proposition}

\newtheorem{lem}[thm]{Lemma}
\newtheorem*{rk}{Remark}
\newtheorem*{ex}{Example}
\newtheorem*{exs}{Examples}

\newtheorem{defi}[thm]{Definition}
\newtheorem*{thm*}{Theorem}
\newtheorem*{prop*}{Proposition}
\newtheorem*{lem*}{Lemma}

\usepackage{graphicx}
\usepackage{float}
\restylefloat{figure}

\newcommand{\E}{\mathbb{E}}
\renewcommand{\P}{\mathbb{P}}
\newcommand{\bL}{\mathbb{L}}
\newcommand{\be}{\mathbbm{e}}
\newcommand{\R}{\mathbb{R}}
\newcommand{\N}{\mathbb{N}}
\newcommand{\Z}{\mathbb{Z}}
\newcommand{\C}{\mathbb{C}}
\newcommand{\D}{\mathbb{D}}
\newcommand{\oD}{\overline{\mathbb{D}}}
\newcommand{\bCL}{\mathbb{CL}}
\newcommand{\bS}{\mathbb{S}}
\newcommand{\bT}{\mathbb{T}}
\newcommand{\K}{\mathbb{K}}
\newcommand{\cP}{\mathcal{P}}

\newcommand{\cT}{\mathcal{T}}

\newcommand{\cF}{\mathcal{F}}

\newcommand{\cEG}{\mathcal{EG}}
\newcommand{\kS}{\mathfrak{S}}
\newcommand{\kM}{\mathfrak{M}}
\newcommand{\kT}{\mathfrak{T}}
\newcommand{\kC}{\mathfrak{C}}
\newcommand{\kU}{\mathfrak{U}}
\newcommand{\kBT}{\mathfrak{BT}}
\newcommand{\kK}{\mathfrak{K}}

\newcommand{\mub}{\mu^\circ}
\newcommand{\mun}{\mu^\bullet}

\newcommand{\wn}{w^\bullet}
\newcommand{\wb}{w^\circ}
\newcommand{\twb}{\tilde{w}^\circ}
\newcommand{\twn}{\tilde{w}^\bullet}

\usepackage[colorlinks=false]{hyperref}\tikzset{cross/.style={cross out, draw=black, minimum size=2*(#1-\pgflinewidth), inner sep=0pt, outer sep=0pt},
cross/.default={1pt}}
\usepackage{animate}
\usepackage{todonotes}

\usetikzlibrary{patterns}

\author{Paul Thevenin\footnote{CMAP \& \'Ecole polytechnique, paul.thevenin@polytechnique.edu \newline The author acknowledges partial support from Agence Nationale de la Recherche,
Grant Number ANR-14-CE25-0014 (ANR GRAAL).}}
\title{Random stable type minimal factorizations of the $n$-cycle}
\date{}
\begin{document}

\maketitle

\begin{abstract}
\small{We investigate random minimal factorizations of the $n$-cycle, that is, factorizations of the permutation $(1 \, 2 \cdots n)$ into a product of cycles $\tau_1, \ldots, \tau_k$ whose lengths $\ell(\tau_1), \ldots, \ell(\tau_k)$ verify the minimality condition $\sum_{i=1}^k(\ell(\tau_i)-1)=n-1$. By associating to a cycle of the factorization a black polygon inscribed in the unit disk, and reading the cycles one after an other, we code a minimal factorization by a process of colored laminations of the disk, which are compact subsets made of red noncrossing chords delimiting faces that are either black or white. Our main result is the convergence of this process as $n \rightarrow \infty$, when the factorization is randomly chosen according to Boltzmann weights in the domain of attraction of an $\alpha$-stable law, for some $\alpha \in (1,2]$. The new limiting process interpolates between the unit circle and a colored version of Kortchemski's $\alpha$-stable lamination. Our principal tool in the study of this process is a bijection between minimal factorizations and a model of size-conditioned labelled random trees whose vertices are colored black or white.}
\end{abstract}

\section{Introduction}
\label{sec:intro}

\subsection{Model and motivation}

The purpose of this work is to introduce and investigate a geometric representation, as compact subsets of the unit disk, of certain random minimal factorizations of the $n$-cycle. For an integer $n \geq 1$, let $\kS_n$ be the group of permutations of $\llbracket 1,n \rrbracket$, and $\kC_n$ the set of cycles of $\kS_n$. We denote by $\ell(c)$ the length of a cycle $c \in \kC_n$. A particular object of interest is the $n$-cycle $c_n \coloneqq (1 2 \ldots n)$, which maps $i$ to $i+1$ for $1 \leq i \leq n-1$, and $n$ to $1$. For any $n \geq k \geq 1$, the elements of the set
\begin{align*}
\kM_n^{(k)} := \left\{ (\tau_1, \ldots, \tau_k) \in \kC_n^k, \, \tau_1 \cdots \tau_k = c_n, \, \forall i \, \, \ell(\tau_i) \geq 2, \sum\limits_{i=1}^k \left( \ell(\tau_i)-1 \right) = n-1 \right\}
\end{align*}
are called minimal factorizations of $c_n$ of order $k$, while an element of
\begin{align*}
\kM_n := \underset{k=1}{\overset{n-1}{\bigcup}} \kM_n^{(k)}.
\end{align*}
is simply called a minimal factorization of $c_n$ (one can check that $\kM_n^{(k)}$ is empty as soon as $k \geq n$).
By convention, we read cycles from the left to the right, so that $\tau_1 \tau_2$ corresponds to $\tau_2 \circ \tau_1$. Notice that the condition $\sum_{i=1}^k \left( \ell(\tau_i)-1 \right) = n-1$ in the definition of $\kM_n^{(k)}$ is a condition of minimality, in the sense that any $k$-tuple of cycles $(\tau_1, \ldots, \tau_k)$ such that $\tau_1 \cdots \tau_k = c_n$ necessarily verifies 
\begin{align*}
\sum\limits_{i=1}^k \left( \ell(\tau_i)-1 \right) \geq n-1.
\end{align*}

Minimal factorizations of the $n$-cycle are a topic of interest, mostly in the restrictive case of factorizations into transpositions (that is, all cycles in the factorization have length $2$). The number of minimal factorizations of $c_n$ into transpositions is known to be $n^{n-2}$ since Dénes \cite{Den59}, and bijective proofs of this result have been given, notably by Moszkowski \cite{Mos89} or Goulden and Pepper \cite{GP93}. These proofs use bijections between the set of minimal factorizations into transpositions and sets of trees, whose cardinality is computed by other ways.

More recently, factorizations into transpositions have been studied from a probabilistic approach by Féray and Kortchemski, who investigate the asymptotic behaviour of such a factorization taken uniformly at random, as $n$ grows. On one hand from a 'local' point of view \cite{FK18}, by studying the joint trajectories of finitely many integers through the factorization. On the other hand from a 'global' point of view \cite{FK17}, by coding a factorization in the unit disk as was initially suggested by Goulden and Yong \cite{GY02}: associating to each transposition a chord in the disk and drawing these chords in the order in which the transpositions appear in the factorization, they code a uniform factorization by a random process of sets of chords, and prove the convergence of the $1$-dimensional marginals of this process as $n$ grows, after time renormalization. The author \cite{The19} extends this result by proving the functional convergence of the whole process, highlighting in addition interesting connections between this model and a fragmentation process of the so-called Brownian Continuum Random Tree (in short, CRT), due to Aldous and Pitman \cite{AP98}. This fragmentation process codes a way of cutting the CRT at random points into smaller components, as time passes. 

Let us also mention Angel, Holroyd, Romik and Vir{\'a}g \cite{AHRV07}, and later Dauvergne \cite{Dau19}, who investigate from a geometric point of view the closely related model of uniform sorting networks, that is, factorizations of the reverse permutation (which exchanges $1$ with $n$, $2$ with $n-1$, etc.) into adjacent transpositions, that exchange only consecutive integers. 

In an other direction, more general minimal factorizations have been studied as combinatorial structures. Specifically, Biane \cite{Bia96} investigates the case of minimal factorizations of $c_n$ of class $\overline{a} \coloneqq (a_1, \ldots, a_k)$, where $a_1, \ldots, a_k$ are all integers $\geq 2$ such that $\sum_{i=1}^k (a_i-1) = n-1$, which are $k$-tuples of cycles $(\tau_1, \ldots, \tau_k) \in \kM_n^{(k)}$ such that, for $1 \leq i \leq k$, $\ell(\tau_i)=a_i$. Biane notably proves that, at $\overline{a}$ fixed, the number of factorizations of class $\overline{a}$ is surprisingly always equal to $n^{k-1}$, and therefore only depends on the cardinality of the class. A proof based on a bijection with a new model of trees is given by Du and Liu \cite{DL15}, which inspired our own bijection, exposed in Section \ref{sec:minfac} and very close to theirs. In particular, the class $(2,2,\ldots,2)$, where $2$ is repeated $n-1$ times, corresponds to minimal factorizations of the $n$-cycle into transpositions, and one recovers Dénes' result.

Our goal in this paper is to extend the geometric approach of uniform minimal factorizations into transpositions initiated by Féray \& Kortchemski to minimal factorizations into cycles of random lengths, when the probability of choosing a given factorization only depends on its class.

\paragraph*{Weighted minimal factorizations}

Let us immediately introduce the object of interest of this paper, which is a new model of random factorizations. The idea is to generalize minimal factorizations of the cycle into transpositions, by giving to each element of $\kM_n$ a weight which depends on its class and then choosing a factorization at random proportionally to its weight. 

We fix a sequence $w \coloneqq (w_i)_{i \geq 1}$ of nonnegative real numbers, which we call weights. We will always assume that there exists $i \geq 1$ such that $w_i>0$. For any positive integers $n,k$, and any factorization $f \coloneqq (\tau_1, \ldots, \tau_k) \in \kM_n^{(k)}$, define the weight of $f$ as 
\begin{align*}
W_w(f) \coloneqq \prod_{i=1}^k w_{\ell(\tau_i)-1}.
\end{align*}
Then, we define the $w$-minimal factorization of the $n$-cycle, denoted by $f_n^w$, as the random variable on the set $\kM_n$ such that, for all $f \in \kM_n$, the probability that $f_n^w$ is equal to $f$ is proportional to the weight of $f$:
\begin{align*}
\P \left( f_n^w=f \right) = \frac{1}{Y_{n,w}} W_w(f),
\end{align*}
where $Y_{n,w} \coloneqq \sum_{f \in \kM_n} W_w(f)$ is a renormalization constant. We shall implicitly restrict our study to the values of $n$ such that $Y_{n,w}>0$.

Remark that some particular weight sequences give birth to specific models of random factorizations:
\begin{itemize}
\item fix an integer $r \geq 2$, and define $\delta^r$ as follows: $\delta^r_{r-1}=1$, and for all $k \neq r-1$, $\delta^r_k=0$. Then $f_n^{\delta^r}$ is a uniform minimal factorization of the $n$-cycle into $r$-cycles. In particular, when $r=2$, one recovers the model of minimal factorizations into transpositions, studied in depth in \cite{FK17, FK18, The19}.
\item define $v$ as the weight sequence such that, for all $k \geq 1$, $v_k=1$. Then $f_n^v$ is a uniform element of $\kM_n$.
\end{itemize}

\paragraph*{Minimal factorizations of stable type}

We specifically focus in this paper on a particular case of weighted factorizations, which we call factorizations of stable type. These random factorizations of the $n$-cycle are of great interest, as we can code them by a process of compact subsets of the unit disk which converges in distribution (see Theorem \ref{thm:mainresult}).

Let us start with some definitions. A function $L: \R_+^* \rightarrow \R_+^*$ is said to be slowly varying if, for any $c>0$, $L(cx)/L(x) \rightarrow 1$ as $x \rightarrow \infty$.
For $\alpha \in (1,2]$, we say that a probability distribution $\mu$ which is critical - that is, $\mu$ has mean $1$ - is in the domain of attraction of an $\alpha$-stable law if there exists a slowly varying function $L$ such that, as $x \rightarrow \infty$,
\begin{equation}
\label{eq:stablelaw}
Var\left[ X \, \mathds{1}_{X \leq x} \right] \sim x^{2-\alpha} L(x),
\end{equation}
where $X$ is a random variable distributed according to $\mu$. We refer to \cite{Jan11} for an in-depth study of these well-known distributions. In particular, any law with finite variance is in the domain of attraction of a $2$-stable law.

Throughout the paper, for such a distribution $\mu$, $(B_n)_{n \geq 1}$ denotes a sequence of positive real numbers satisfying
\begin{equation}
\label{eq:Bn}
\frac{n L(B_n)}{B_n^\alpha} \underset{n \rightarrow \infty}{\rightarrow} \frac{\alpha (\alpha - 1)}{\Gamma \left( 3 - \alpha \right)},
\end{equation}
where $L$ verifies \eqref{eq:stablelaw}. Notice in particular that, if $\mu$ has finite variance $\sigma^2$, then $L(x) \rightarrow \sigma^2$ as $x \rightarrow \infty$, and \eqref{eq:Bn} can be rewritten $B_n \underset{n \rightarrow \infty}{\sim} \frac{\sigma}{\sqrt{2}} \sqrt{n}$. For such a sequence $(B_n)_{n \geq 1}$, we denote by $(\tilde{B}_n)_{n \geq 1}$ the sequence defined as
\begin{equation}
\label{eq:tbn}
\tilde{B}_n = \left\{
      \begin{aligned}
        B_n \qquad & \text{ if } \quad \alpha<2, \text{ or } \alpha=2 \text{ and } Var(\mu)=\infty,\\
      \sqrt{\frac{\sigma^2+1}{2}} \sqrt{n} \qquad & \text{ if } \quad Var(\mu) = \sigma^2<\infty. \\
      \end{aligned}
    \right.
\end{equation}

Finally, for $\alpha \in (1,2]$, we say that a weight sequence $w$ is of $\alpha$-stable type if there exists a critical distribution $\nu$ in the domain of attraction of an $\alpha$-stable law and a real number $s>0$ such that, for all $i \geq 1$,
\begin{align*}
w_i = \nu_i s^i.
\end{align*}
In this case, $\nu$ is said to be the critical equivalent of $w$. One can check that, if $w$ admits a critical equivalent - which is not always the case - then it is unique. Furthermore, it appears that, whenever different weight sequences may have the same critical equivalent, the distribution of the minimal factorization $f_n^w$ only depends on this critical law. If $w$ is a weight sequence of $\alpha$-stable type, then we also say that $f_n^w$ is a minimal factorization of $\alpha$-stable type. 

Throughout the paper, we investigate several combinatorial quantities of these factorizations. Here are two examples. As a first result, we can control the number of cycles in such a factorization. For any $n \geq 1$ and any minimal factorization $F$ of the $n$-cycle, denote by $N(F)$ the number of cycles in $F$.

\begin{lem}
\label{lem:intronbcycles}
Let $w$ be a weight sequence of $\alpha$-stable type for some $\alpha \in (1,2]$, and $\nu$ be its critical equivalent. Then, as $n \rightarrow \infty$,
\begin{equation*}
\frac{1}{n} N \left( f_n^w \right) \overset{\P}{\rightarrow} 1-\nu_0,
\end{equation*}
where $\overset{\P}{\rightarrow}$ denotes the convergence in probability.
\end{lem}
In words, the number of cycles in $f_n^w$ behaves linearly in $n$. The proof of this lemma can be found in Section \ref{ssec:image}.

\begin{ex}
If one looks at the weight sequence $v$ defined as $v_i=1$ for all $i \geq 1$ (so that $f_n^v$ is a uniform element of $\kM_n$), then one can check that $v$ has a critical equivalent $\nu$ satisfying $\nu_0 = (3-\sqrt{5})/2$ and $\nu_i = ((3-\sqrt{5})/2)^i$ for $i \geq 1$. Thus, the average number of cycles in a uniform minimal factorization of the $n$-cycle is of order $(1-\nu_0) n = (\sqrt{5}-1) \, n/2$.
\end{ex}

As an other side result, we are able to control the length of the largest cycle in such factorizations. For any $n \geq 1$ and any minimal factorization $F$ of the $n$-cycle, denote by $\ell_{max}(F)$ the length of the largest cycle in $F$.

\begin{prop}
\label{prop:largestcycle}
Let $w$ be a weight sequence of $\alpha$-stable type for some $\alpha \in (1,2]$, and $\nu$ be its critical equivalent. Let $(B_n)_{n \geq 1}$ be a sequence satisfying \eqref{eq:Bn}. Then:
\begin{itemize}
\item[(i)] if $\alpha = 2$, then with probability going to $1$ as $n \rightarrow \infty$, $\ell_{max}(f_n^w) = o(B_n)$;
\item[(ii)] if $\alpha<2$ then for any $\epsilon>0$ there exists $\eta>0$ such that, for $n$ large enough, with probability larger than $1-\epsilon$, $\eta B_n \leq \ell_{max}(f_n^w) \leq \eta^{-1} B_n$.
\end{itemize}
\end{prop}
In other terms, for $\alpha<2$, the largest cycle in $f_n^w$ is of order $B_n$ (thus of order $n^{1/\alpha}$, up to a slowly varying function). If $\alpha=2$ then one can only say that the largest cycle is of length $o(B_n)$ (which means $o(\sqrt{n})$ if $\nu$ has finite variance). This is proved in Section \ref{ssec:randomtrees}.

\subsection{Coding a minimal factorization by a colored lamination-valued process}
\label{ssec:codingminfac}

The first aim of this paper is to code random minimal factorizations in the unit disk. In what follows, $\oD \coloneqq \{ z \in \C, |z| \leq 1 \}$ denotes the closed unit disk and $\bS^1 \coloneqq \{ z \in \oD, |z|=1 \}$ the unit circle.

The idea of coding random structures by compact subsets of $\oD$ goes back to Aldous \cite{ald} who investigates triangulations of large polygons: let $n \in \Z_+$ and define $P_n$, the regular $n$-gon inscribed in $\oD$, whose vertices are $\left\{ e^{2i k\pi/n}, 1 \leq k \leq n \right\}$. Aldous proves that a random uniform triangulation of $P_n$ (that is, a set of non-crossing diagonals of $P_n$ whose complement in $P_n$ is a union of triangles) converges in distribution towards a random compact subset of the disk which he calls the \textit{Brownian triangulation}. This Brownian triangulation is notably a lamination - that is, a compact subset of $\oD$ made of the union of the circle and a set of chords that do not cross, except maybe at their endpoints. In particular, the faces of this lamination, which are the connected components of its complement in $\oD$, are all triangles. See Fig. \ref{fig:colorlam}, middle, for an approximation of this lamination. Since then, the Brownian triangulation has been appearing as the limit of various random discrete structures \cite{Bet17, CK14}, and has also been connected to random maps \cite{LGP08}.

In a extension of Aldous' work, Kortchemski \cite{Kor14} constructed a family $(\bL_\infty^{(\alpha)})_{1<\alpha \leq 2}$ of random laminations, called $\alpha$-stable laminations (see Fig. \ref{fig:colorlam}, left, for an approximation of the stable lamination $\bL^{(1.3)}$). These laminations appear as limits of large general Boltzmann dissections of the regular $n$-gon \cite{Kor14}. This extends Aldous' result about triangulations, since the $2$-stable lamination $\bL_\infty^{(2)}$ is distributed as the Brownian triangulation. Stable laminations are also limits of large non-crossing partitions \cite{KM17}.

Let us now introduce colored laminations, which generalize the notion of lamination. A colored lamination is a subset of $\oD$, in which each point is colored either black, red or left white, so that the subset of red points is a lamination whose faces are each either completely black or completely white. See an example on Fig. \ref{fig:firstlamination}. A particular example of colored laminations is the colored analogue of the $\alpha$-stable laminations. These objects are colored laminations whose red chords form the $\alpha$-stable lamination, and whose faces are colored black in an i.i.d. way. Specifically, for $p \in [0,1]$, the $p$-colored $\alpha$-stable lamination $\bL_\infty^{(\alpha),p}$ is a random colored lamination such that: (i) the red part of $\bL_\infty^{(\alpha),p}$ has the law of $\bL_\infty^{(\alpha)}$; (ii) independently of the red component, the faces of $\bL_\infty^{(\alpha),p}$ are colored black independently of each other with probability $p$ (see Fig. \ref{fig:colorlam}, right for a simulation of $\bL_\infty^{(2), 0.5}$).

\begin{figure}[!h]
\center
\caption{Left: a simulation of the $1.3$-stable lamination $\bL_\infty^{(1.3)}$. Middle: a simulation of the Brownian triangulation $\bL_\infty^{(2)}$. Right: a simulation of $\bL_\infty^{(2),0.5}$.}
\label{fig:colorlam}
\begin{tabular}{c c c}
\includegraphics[scale=.34]{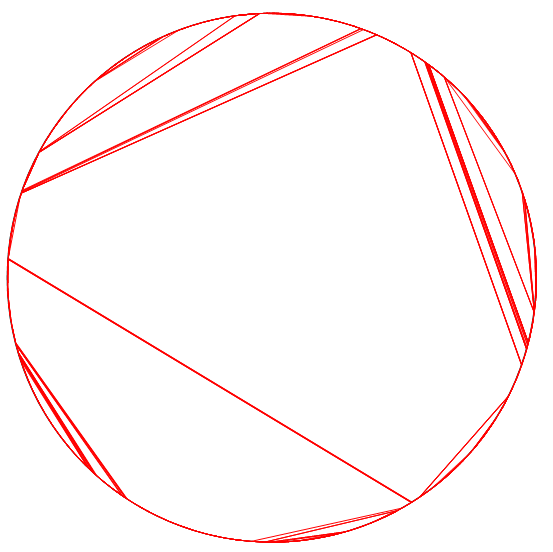}
&
\includegraphics[scale=.25]{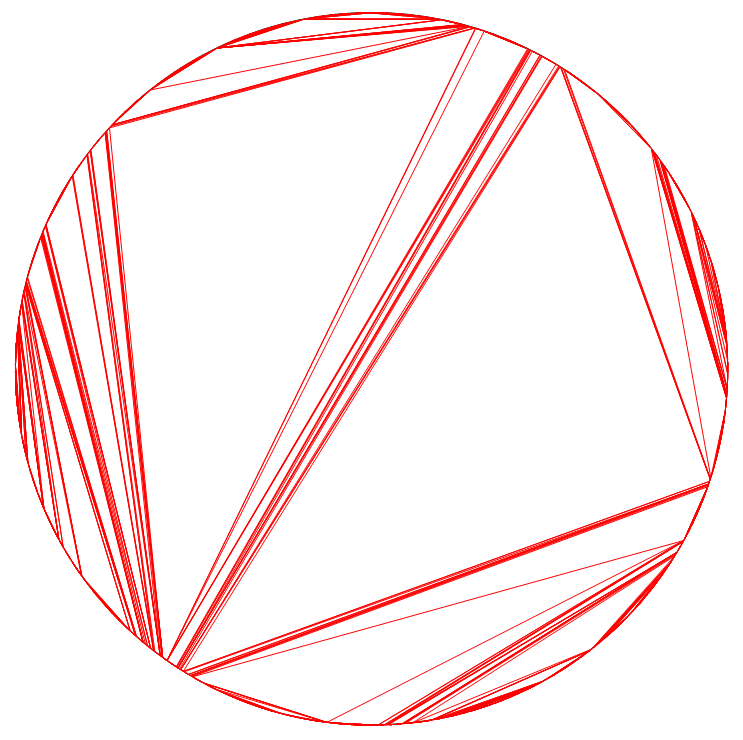} 
&
\includegraphics[scale=.25]{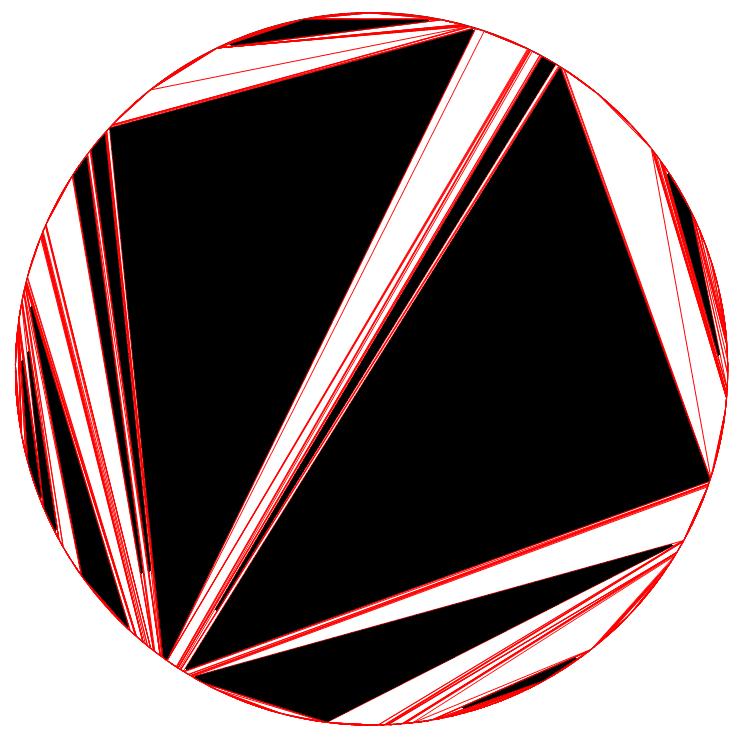}

\end{tabular}
\end{figure}

We now provide a way of representing a minimal factorization by a process of colored laminations of the unit disk. This representation is a generalization of the representation of a minimal factorization into transpositions, introduced by Goulden and Yong \cite{GY02}. It consists in drawing, for each cycle $\tau$ of the factorization, a number $\ell(\tau)$ of red chords in $\oD$, and coloring the face that it creates in black. More precisely, for $n \geq 1$, let $\tau \in \kC_n$ be a cycle and assume that it can be written as  $(e_1, \ldots, e_{\ell(\tau)})$, where $e_1 < \cdots < e_{\ell(\tau)}$. We will prove later that indeed, if $\tau$ appears in a minimal factorization of the $n$-cycle, then it satisfies this condition (see Section \ref{sec:minfac} for details and proofs). Then draw in red, for each $1 \leq j \leq \ell(\tau)-1$, the chord $[e^{-2i \pi e_j/n}, e^{-2i \pi e_{j+1}/n}]$, and also draw the chord $[e^{-2i \pi e_{\ell(\tau)}/n}, e^{-2i \pi e_1/n}]$ (here, $[x,y]$ denotes the segment connecting $x$ and $y$ in $\R^2$). This creates a red cycle. Color now the interior of this cycle in black, and finally color the unit circle in red. We denote the colored lamination that we obtain by $S(\tau)$. 

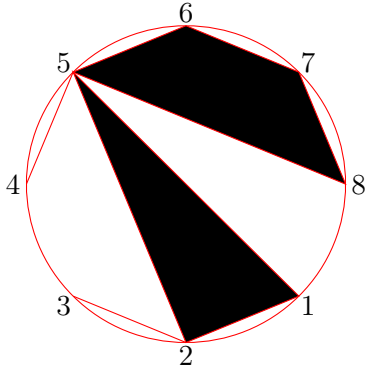
\begin{figure}[!h]
\center
\caption{The subset $S(f)$ where $f \coloneqq (5678)(23)(125)(45)$ is an element of $\kM_8$.}
\label{fig:firstlamination}
\begin{tikzpicture}[scale=.7,rotate=-45]
\draw[red] (0,0) circle (3);
\foreach \k in {1,...,8}
{\draw ({3.25*cos(360*(\k-1)/8)},{-3.25*sin(360*(\k-1)/8)}) node{\k};}
\draw[red,fill=black] ({3*cos(360*0/8)},{-3*sin(360*0/8)}) -- ({3*cos(360*4/8)},{-3*sin(360*4/8)}) -- ({3*cos(360/8)},{-3*sin(360/8)}) -- cycle;
\draw[red] ({3*cos(360*4/8)},{-3*sin(360*4/8)}) --  ({3*cos(360*3/8)},{-3*sin(360*3/8)});
\draw[red] ({3*cos(360*1/8)},{-3*sin(360*1/8)}) -- ({3*cos(360*2/8)},{-3*sin(360*2/8)});
\draw[red,fill=black] ({3*cos(360*7/8)},{-3*sin(360*7/8)}) -- ({3*cos(360*4/8)},{-3*sin(360*4/8)}) -- ({3*cos(360*5/8)},{-3*sin(360*5/8)}) -- ({3*cos(360*6/8)},{-3*sin(360*6/8)}) -- cycle ;
\end{tikzpicture}
\end{figure}

Now, let $n, k \geq 1$ and $f \coloneqq (\tau_1,\ldots,\tau_k) \in \kM_n^{(k)}$. For $c \in [0,\infty]$, define $S_c(f)$ as the colored lamination
$$S_c(f) = \bS^1 \cup \bigcup_{r=1}^{\lfloor c \rfloor \wedge k} S(\tau_r),$$
and finally define $S(f)$ as $S(f) \coloneqq S_k(f) = \bigcup_{r=1}^k S(\tau_r)$. See Fig. \ref{fig:firstlamination} for an example.

In their study of a uniform minimal factorization of the $n$-cycle into transpositions (which we now denote by $f_n^{(2)}$ instead of $f_n^{\delta^2}$, for convenience), Féray and Kortchemski \cite{FK17} show that a phase transition appears after having read roughly $\sqrt{n}$ transpositions. Specifically, at $c \geq 0$ fixed, the lamination $S_{c \sqrt{n}}(f_n^{(2)})$ converges in distribution for the Hausdorff distance, as $n$ grows, to a random lamination $\bL_c^{(2)}$. The author \cite{The19} later obtains the functional analogue of this convergence, thus providing a coupling between the laminations $(\bL^{(2)}_c)_{c \geq 0}$. Let us explain in which sense we understand the convergence of colored lamination-valued processes: for $E,F$ two metric spaces, following Annex $A2$ in \cite{Kal02}, let $\D(E,F)$ be the space of càdlàg functions from $E$ to $F$ (that is, right-continuous functions with left limits), endowed with the $J_1$ Skorokhod topology. In our case, we see a colored lamination of $\oD$ as an element of $\K^2$, where $\K$ denotes the space of compact subsets of $\D$. The first coordinate corresponds to the set of red points which is a lamination by definition, and the second one to the colored component (the set of points that are black or red). Finally, the set $\bCL(\oD)$ of colored laminations of $\oD$ is endowed with the distance which is the sum of the usual Hausdorff distances $d_H$ on the two coordinates. This means that, if $A $ and $B$ are two colored laminations of $\oD$, $d_H(A,B) = d_H(A_r,B_r)+d_H(A_c,B_c)$ where $L_r$ denotes the set of red points of a colored lamination $L$, and $L_c$ the set of colored points of $L$ (that is, either black or red). For convenience, we denote this new distance on $\bCL(\oD)$ by $d_H$ as well. Finally, the set of laminations of the disk is seen as a subset of $\bCL(\oD)$, on which the red part and the colored part of an element are equal.

\begin{thm*}{\cite[Theorem $1.2$]{The19}}
There exists a lamination-valued process $(\bL^{(2)}_c)_{c \in [0,\infty]}$ such that the following convergence holds in distribution for the Skorokhod distance in $\D([0,\infty], \bCL(\oD))$, as $n \rightarrow \infty$:
\begin{equation}
\label{eq:cvthe19}
\left( S_{c \sqrt{n}}\left(f_n^{(2)} \right) \right)_{c \in [0,\infty]} \overset{(d)}{\rightarrow} \left( \bL^{(2)}_c \right)_{c \in [0,\infty]}.
\end{equation}
\end{thm*}

In this case, it is to note that no face with $3$ or more chords in its boundary appears in $S(f_n^{(2)})$, as $S(\tau)$ is in fact just the union of $\bS^1$ and a chord when $\ell(\tau)=2$. Therefore, no point of $S(f_n^{(2)})$ is black and, for any $c \in [0,\infty]$, $S_{c\sqrt{n}}(f_n^{(2)})$ is just a lamination. Moreover, it appears that the process $(\bL^{(2)}_c)_{c \in [0,\infty]}$ is a nondecreasing interpolation between the unit circle and the Brownian triangulation. 

Our main result, which generalizes \eqref{eq:cvthe19}, states the convergence of the geometric representation of a random minimal factorization of stable type. In the stable case, the phase transition does not appear at the scale $\sqrt{n}$ anymore, but at the scale $\tilde{B_n}$.

\begin{thm}
\label{thm:mainresult}
Let $\alpha \in (1,2]$ and $w$ a weight sequence of $\alpha$-stable type. Let $\nu$ be its critical equivalent and $(\tilde{B}_n)_{n \geq 0}$ satisfying \eqref{eq:tbn}. Then, there exists a lamination-valued process $(\bL_c^{(\alpha)})_{c \in [0,\infty]}$, depending only on $\alpha$, such that:
\begin{itemize}
\item[(I)] If $\alpha<2$, then the following convergence holds:
\begin{align*}
\left( \left( S_{c (1-\nu_0) \tilde{B}_n}(f_n^w) \right)_{0 \leq c < \infty}, S_\infty(f_n^w) \right) \underset{n \rightarrow \infty}{\overset{(d)}{\rightarrow}} \left( \left(\bL_c^{(\alpha)} \right)_{0 \leq c < \infty}, \bL_\infty^{(\alpha), 1}\right).
\end{align*}

\item[(II)] If $\nu$ has finite variance, there exists a parameter $p_\nu \in [0,1]$ such that the following convergence holds:
\begin{align*}
\left( \left( S_{c (1-\nu_0) \tilde{B}_n}(f_n^w) \right)_{0 \leq c < \infty}, S_\infty(f_n^w) \right) \underset{n \rightarrow \infty}{\overset{(d)}{\rightarrow}} \left( \left(\bL_c^{(2)} \right)_{0 \leq c < \infty}, \bL_\infty^{(2), p_\nu}\right).
\end{align*}
Furthermore,
\begin{align*}
p_\nu = \frac{\sigma_\nu^2}{\sigma_\nu^2+1},
\end{align*}
where $\sigma_\nu^2$ denotes the variance of $\nu$.
\end{itemize}
Both convergences hold in distribution in the space $\D(\R_+, \bCL(\oD)) \times \bCL(\oD)$.
\end{thm}

This process $(\bL_c^{(\alpha)})_{c \in [0,\infty]}$ is a nondecreasing interpolation between the circle and $\bL_\infty^{(\alpha)}$. It is in addition lamination-valued, in the sense that, for all $c \geq 0$, almost surely $\bL_c^{(\alpha)}$ contains no black point.

\begin{rk}
We conjecture that the result of Theorem \ref{thm:mainresult} (I) still holds when $\alpha=2$ and $\nu$ has infinite variance. However our proofs do not directly apply to this case.
\end{rk}

\begin{exs}
In the case $w = \delta^j$ for some $j \geq 1$, the critical equivalent of $\delta^j$ is $\nu \coloneqq \frac{j-2}{j-1} \delta^0 + \frac{1}{j-1} \delta^j$, and $p_\nu=\frac{j-2}{j-1}$. When $j=2$, $p_\nu=0$ and we recover the Brownian triangulation without coloration, as the limit of the lamination obtained from a uniform minimal factorization of the $n$-cycle into transpositions. In the case $j=3$ of factorizations into $3$-cycles, each face of the limiting colored Brownian triangulation is black with probability $1/2$.

In the case of a minimal factorization of the $n$-cycle taken uniformly at random, one obtains the surprising limit value $p_\nu=1-1/\sqrt{5}$.
\end{exs}

\subsection{Construction of the processes $(\bL_c^{(\alpha)})_{c \in [0,\infty]}$}

Let us immediately explain how to construct the limiting processes $(\bL_c^{(\alpha)})_{c \in [0,\infty]}$ which appear in the statement of Theorem \ref{thm:mainresult} and in \eqref{eq:cvthe19}. In order to understand this construction, we define from a (deterministic) càdlàg function $F: [0,1] \rightarrow \R_+$ such that $F(0)=F(1)=0$ a random lamination-valued process $(\bL_c(F))_{c \in [0,\infty]}$. Define the epigraph of $F$ as $\cEG(F) \coloneqq \{ (s,t) \in \R^2, 0 \leq s \leq 1, \, 0 \leq t < F(s) \}$, the set of all points that are under the graph of $F$. Denote by $\cP(F)$ an inhomogeneous Poisson point process on $\cEG(F) \times \R_+$, of intensity 
\begin{align*}
\frac{2 \, ds dt}{d(F,s,t)-g(F,s,t)} \mathds{1}_{(s,t) \, \in \, \cEG(F)} \, dr,
\end{align*}
where $g(F,s,t) \coloneqq \sup \{ s' \leq s, F(s') < t \}$ and $d(F,s,t) \coloneqq \inf \{ s' \geq s, F(s') < t \}$, and $r$ shall be understood as a 'time' coordinate.
For any $c \geq 0$, its restriction to $\R^2 \times [0,c]$ is denoted by $\cP_c(F)$. Now, for $c \geq 0$, define the lamination $\bL_c(F)$ as
\begin{align*}
\bS^{1} \cup \overline{\bigcup\limits_{(s,t) \in \cP_c(F)} \left[e^{-2i\pi g(F,s,t)}, e^{-2i\pi d(F,s,t)}\right]},
\end{align*}
so that each point of $\cP_c(F)$ codes a chord in $\oD$ (see Fig. \ref{fig:Lcconstruction}), and let $\bL_\infty(F)$ be the lamination $\overline{\bigcup_{c \geq 0} \bL_c(F)}$.

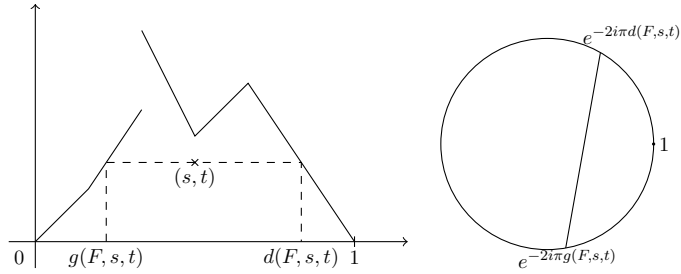
\begin{figure}[!h]
\caption{A càdlàg function $F$, a point $(s,t)$ of its epigraph $\cEG(F)$ and the corresponding chord in $\oD$.}
\label{fig:Lcconstruction}
\center
\begin{tabular}{c c c}
\begin{tikzpicture}[scale=.7, every node/.style={scale=0.7}]
\draw[->] (-.5,0) -- (7,0);
\draw[->] (0,-.5) -- (0,4.5);
\draw (0,0) -- (1,1) -- (2,2.5) (2,4) -- (3,2) -- (4,3) --(6,0);
\draw (6,-.1) -- (6,.1);
\draw (6,-.3) node{$1$};
\draw (-.3, -.3) node{$0$};

\draw (3,1.5) node[scale=3,cross]{};
\draw (3,1.2) node{$(s,t)$};
\draw[dashed] (1.3333,0) -- (1.3333,1.5) -- (5,1.5) -- (5,0);
\draw (1.333,-.3) node{$g(F,s,t)$};
\draw (5,-.3) node{$d(F,s,t)$}; 
\end{tikzpicture}
&
\begin{tikzpicture}[scale=1.4, every node/.style={scale=.7}]
\draw ({1.1*cos(-80)},{1.1*sin(-80)}) node{$e^{-2i\pi g(F,s,t)}$};
\draw ({1.6*cos(-300)},{1.2*sin(-300)}) node{$e^{-2i\pi d(F,s,t)}$};
\draw[fill] (1,0) circle (.01);
\draw (1.1,0) node{$1$};
\draw (0,0) circle (1);
\draw ({cos(-80)},{sin(-80)}) -- ({cos(-300)},{sin(-300)});
\end{tikzpicture}
\end{tabular}
\end{figure}

For $\alpha \in (1,2]$, the process $(\bL_c^{(\alpha)})_{c \in [0,\infty]}$ is constructed this way from the so-called stable height process $H^{(\alpha)}$:

\begin{align*}
\left(\bL_c^{(\alpha)}\right)_{c \in [0,\infty]} = \left(\bL_c\left(H^{(\alpha)}\right)\right)_{c \in [0,\infty]}.
\end{align*}
These stable height processes are random continuous processes on $[0,1]$, and can be defined starting from stable Lévy processes (see Fig. \ref{fig:excursions}, right, for a simulation of $H^{(1.7)}$, and Section \ref{sec:trees} for more background and details). The animated Fig. \ref{fig:animate} is an approximation of the process $(\bL^{(1.8)}_c)_{c \geq 0}$.

In the case $\alpha=2$, the $2$-stable height process happens to be distributed as the normalized Brownian excursion $(\be_t)_{0 \leq t \leq 1}$, which is roughly speaking a Brownian motion between $0$ and $1$, conditioned to reach $0$ at time $1$ and to be nonnegative between $0$ and $1$ (see Fig. \ref{fig:excursions}, left for a simulation of $\be$). It appears in addition that the lamination $\bL_\infty^{(2)}$ coded by $\be$ has the law of Aldous' Brownian triangulation.

\begin{figure}[!h]
\caption{A simulation of the normalized Brownian excursion $\be$ (left) and the $1.7$-stable height process $H^{(1.7)}$ (right).}
\label{fig:excursions}
\center
\begin{tabular}{c c c}
\includegraphics[scale=.22]{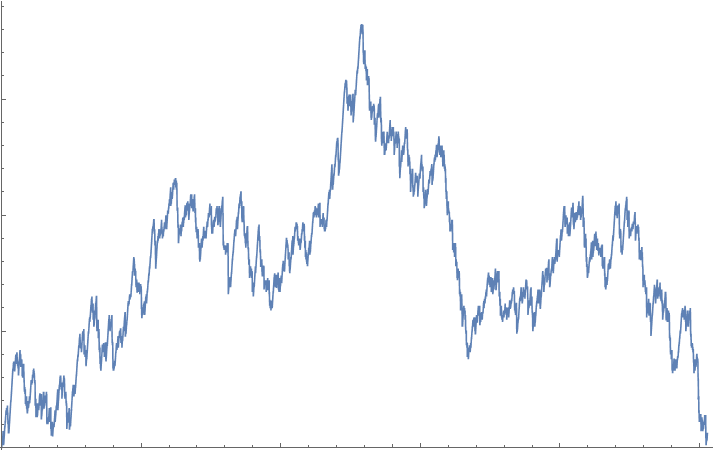} 
&
\begin{tikzpicture}
\draw[white] (0,0) -- (1.5,0);
\end{tikzpicture}
&
\includegraphics[scale=.4]{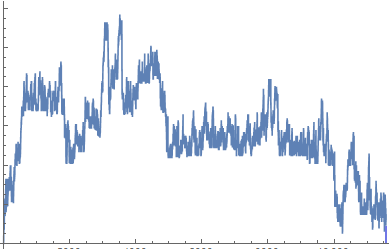}
\end{tabular}
\end{figure}
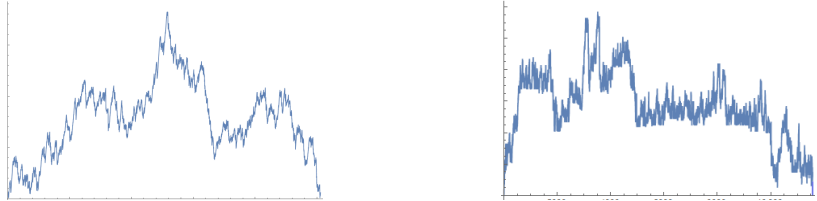

\begin{figure}[h!]
\center
\caption{The image represents an approximation of the lamination $\bL^{(1.8)}_{15}$. By using Adobe Acrobat and by clicking on the ``play'' button, one can view an approximation of the process $\left(\bL^{(1.8)}_c\right)_{c \geq 0}$.}
\label{fig:animate}
\animategraphics[scale=.7,width=.5\textwidth,controls]{5}{film/moiii}{0}{50}
\end{figure}

\subsection{A bijection with labelled bi-type trees}

The main idea in the proof of Theorem \ref{thm:mainresult} is to use a bijection between the set of minimal factorizations of the $n$-cycle and a certain set of discrete trees with labels on their vertices. Specifically, for $n \geq 1$, denote by $\kU_{n}$ the set of trees $T$ that satisfy the following conditions:

\begin{itemize}
\item $T$ is a bi-type tree, that is, its vertices at even height are colored white and its vertices at odd height are colored black;
\item the root and the leaves of $T$ are white. In other terms, each black vertex necessarily has at least one child;
\item $T$ has $n$ white vertices;
\item black vertices of $T$ are labelled from $1$ to $N^\bullet(T)$, where $N^\bullet(T)$ is the total number of black vertices in the tree. In addition, the labels of the neighbours (parent and children) of each white vertex are sorted in decreasing clockwise order, starting from one of these neighbours, and the labels of the children of the root are sorted in decreasing order.
\end{itemize}
See Fig. \ref{fig:bttree} for an example.

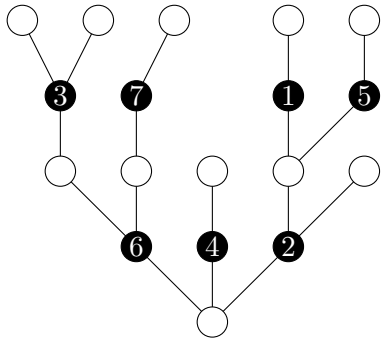
\begin{figure}[!h]
\center
\caption{An element of the set $\kU_{11}$. Its six black vertices are labelled from $1$ to $6$, in clockwise decrasing order around each white vertex. The children of the root are in decreasing order.}
\label{fig:bttree}
\begin{tikzpicture}
\draw (-2.5,4) -- (-2,3) -- (-2,2) -- (-1,1) -- (-1,2) (-1,1) -- (0,0) -- (0,1) (0,0) -- (1,1) -- (2,2) (1,1) -- (1,2) -- (1,3) -- (1,4) (-1.5,4) -- (-2,3) (0,1) -- (0,2) (1,2) -- (2,3) -- (2,4) (-1,2) -- (-1,3) -- (-.5,4);

\draw[fill=white] (0,0) circle (.2);
\draw[fill=white] (-2,2) circle (.2);
\draw[fill=white] (-1,2) circle (.2);
\draw[fill=white] (0,2) circle (.2);
\draw[fill=white] (1,2) circle (.2);
\draw[fill=white] (2,2) circle (.2);
\draw[fill=white] (1,4) circle (.2);
\draw[fill=white] (-2.5,4) circle (.2);
\draw[fill=white] (-1.5,4) circle (.2);
\draw[fill=white] (-.5,4) circle (.2);
\draw[fill=white] (2,4) circle (.2);

\draw[fill=black] (-2,3) circle (.2) node[white]{$3$};
\draw[fill=black] (-1,1) circle (.2) node[white]{$6$};
\draw[fill=black] (-1,3) circle (.2) node[white]{$7$};
\draw[fill=black] (1,1) circle (.2) node[white]{$2$};
\draw[fill=black] (1,3) circle (.2) node[white]{$1$};
\draw[fill=black] (0,1) circle (.2) node[white]{$4$};
\draw[fill=black] (2,3) circle (.2) node[white]{$5$};
\end{tikzpicture}
\end{figure}

\begin{thm}
\label{thm:nkbijection}
For any $n \geq 1$, the sets $\kM_n$ and $\kU_{n}$ are in bijection.
\end{thm}

In Section \ref{sec:minfac}, we provide an explicit bijection between these sets. Roughly speaking, from a minimal factorization $f \in \kM_n$, one constructs $T(f) \in \kU_n$ as the "dual tree" of the colored lamination $S(f)$: its black vertices are in bijection with the cycles of $f$, while its white vertices correspond to white faces of $S(f)$. Theorem \ref{thm:mainresult} is therefore obtained as a corollary of a result on trees, that states the convergence of colored lamination-valued processes coding random bi-type trees (Theorem \ref{thm:cvbitypelam}). It appears indeed that the distribution of the random bi-type tree $T(f_n^w)$ is particularly well understood when $w$ is a weight sequence of stable type.

\paragraph*{Notations}
In the whole paper, $\overset{\P}{\rightarrow}$ denotes the convergence in probability, $\overset{(d)}{\rightarrow}$ the convergence in distribution and $\overset{(d)}{=}$ the equality in distribution. Moreover, a sequence of events $(E_n)_{n \geq 0}$ being given, we say that $E_n$ occurs with high probability if $\P(E_n) \rightarrow 1$ as $n \rightarrow \infty$.

\paragraph*{Outline of the paper}
In Section \ref{sec:trees}, we first define and investigate plane trees and more particularly bi-type trees, which are a cornerstone of the paper as they code minimal factorizations. We suggest two different ways of coding trees by colored laminations-valued processes, and state in particular the convergence of one of these processes coding a particular model of random bi-type trees (Theorem \ref{thm:cvbitypelam}). The proof of this theorem is the main result of Section \ref{sec:particulartrees}, which is devoted to the study of these specific random trees. Finally, in Section \ref{sec:minfac}, we investigate in depth the model of minimal factorizations and explain a natural bijection between the sets $\kM_n$ and $\kU_n$ for all $n$. In addition, we provide the proof of Theorem \ref{thm:mainresult} by showing that the process of colored laminations coded by a random minimal factorization $f_n^w$ of stable type is in some sense also coded by the random bi-type tree $T(f_n^w)$, which is the image of the factorization by the abovementioned bijection.

\paragraph*{Acknowledgements}

I would like to thank Igor Kortchemski for the numerous fruitful discussions that have led to this paper.

\section{Plane trees, bi-type trees and different ways of coding them by laminations}
\label{sec:trees}

In this section, we rigorously define our notion of trees. Then, we describe a certain family of trees, which we call bi-type trees, whose vertices are given a color, either black or white. We finally introduce random models of trees, monotype or bi-type - which we call simply generated trees - and study some of their main properties. We state in particular Theorem \ref{thm:cvbitypelam}, which provides the convergence of a process of colored laminations coding random bi-type trees conditioned by their number of white vertices.

\subsection{Plane trees and their coding by laminations}
\label{ssec:deftree}

\paragraph*{Plane trees.}

We first define \textit{plane trees}, following Neveu's formalism \cite{Nev86}. First, let $\N^* = \left\{ 1, 2, \ldots \right\}$ be the set of all positive integers, and $\mathcal{U} = \cup_{n \geq 0} (\N^*)^n$ be the set of finite sequences of positive integers, with the convention that $(\N^*)^0 = \{ \emptyset \}$.

By a slight abuse of notation, for $k \in \Z_+$, we write an element $u$ of $(\N^*)^k$ as $u=u_1 \cdots u_k$, with $u_1, \ldots, u_k \in \N^*$. For $k \in \Z_+$, $u=u_1\cdots u_k \in (\N^*)^k$ and $i \in \N^*$, we denote by $ui$ the element $u_1 \cdots u_k i \in (\N^*)^{k+1}$ and by $iu$ the element $i u_1 \cdots u_k$. A plane tree $T$ is formally a subset of $\mathcal{U}$ satisfying the following three conditions:

(i) $\emptyset \in T$ ($\emptyset$ is called the root of $T$);

(ii) if $u=u_1\cdots u_n \in T$, then, for all $k \leq n$, $u_1\cdots u_k \in T$ (these elements are called ancestors of $u$, and the set of all ancestors of $u$ is called its ancestral line; $u_1 \cdots u_{n-1}$ is called the \textit{parent} of $u$). The set of ancestors of a vertex $u$ is denoted by $A_u(T)$;

(iii) for any $u \in T$, there exists a nonnegative integer $k_u(T)$ such that, for every $i \in \N^*$, $ui \in T$ if and only if $1 \leq i \leq k_u(T)$ ($k_u(T)$ is called the number of children of $u$, or the outdegree of $u$).

The elements of $T$ are called \textit{vertices}, and we denote by $|T|$ the total number of vertices in $T$. A vertex $u$ such that $k_u(T)=0$ is called a leaf of $T$. The height $h(u)$ of a vertex $u$ is its distance to the root, that is, the unique integer $k$ such that $u \in (\N^*)^k$. We define the height of a tree $T$ as $H(T) = {\sup}_{{u \in T}} \, h(u)$. In the sequel, by tree we always mean plane tree unless specifically mentioned.

The \textit{lexicographical order} $\prec$ on $\mathcal{U}$ is defined as follows:  $\emptyset \prec u$ for all $u \in \mathcal{U} \backslash \{\emptyset\}$, and for $u,w \neq \emptyset$, if $u=u_1u'$ and $w=w_1w'$ with $u_1, w_1 \in \N^*$, then we write $u \prec w$ if and only if $u_1 < w_1$, or $u_1=w_1$ and $u' \prec w'$.  The lexicographical order on the vertices of a tree $T$ is the restriction of the lexicographical order on $\mathcal{U}$.

We do not distinguish between a finite tree $T$, and the corresponding planar graph where each vertex is connected to its parent by an edge of length $1$, in such a way that the vertices with same height are sorted from left to right in lexicographical order.

\paragraph*{Subtrees and nodes}

Let $T$ be a plane tree and $u \in T$ be one of its vertices. We define the subtree of $T$ rooted in $u$ as the set of vertices that have $u$ as an ancestor. This subtree is denoted by $\theta_u(T)$.

One will often consider large nodes in a tree, i.e. vertices whose removal splits the tree into at least two components of macroscopic size (that is, of the same order as the size of $T$) that do not contain the root. Specifically, $a \geq 0$ being fixed, we say that $u \in T$ is an $a$-node of $T$ if there exists an integer $r \leq k_u(T)$ satisfying:
\begin{align*}
\sum_{w \in A_r(u,T)} |\theta_w(T)| \geq a, \sum_{w \in A_{-r}(u,T)} |\theta_w(T)| \geq a,
\end{align*}
where $A_r(u,T)$ denotes the set of the first $r$ children of $u$ in lexicographical order, and $A_{-r}(u,T)$ the set of its other children. In other terms, $u$ is an $a$-node of $T$ if one can split the set of its children into two disjoint subsets made of consecutive children, in such a way that the sum of the sizes of the subtrees rooted in the elements of each of these two subsets is larger that $a$. In what follows, $a$ will be of order $|T|$ (by this, we mean larger than $\epsilon |T|$, for some $\epsilon>0$). We denote by $E_a(T)$ the set of $a$-nodes of the tree $T$.

A particular case of $a$-nodes, for $a>0$, is the case of $a$-branching points. We say that $u \in T$ is an $a$-branching point if there exist two children of $u$, say, $v_1(u)$ and $v_2(u)$, such that $|\theta_{v_1(u)}(T)|\geq a, |\theta_{v_2(u)}(T)|\geq a$. One easily sees that any $a$-branching point is an $a$-node.

\paragraph*{Contour function of a tree, associated lamination-valued process.}

We introduce here some important objects derived from a plane tree. Specifically, a finite plane tree $T$ being given, we define its contour function, which is a walk on the nonnegative integers coding $T$ in a bijective way. In a second time, we construct from this contour function a lamination $\bL(T)$ and a random lamination-valued process $(\bL_u(T))_{u \in [0,\infty]}$ which interpolates between $\bS^1$ and $\bL(T)$. In what follows, $T$ is a plane tree and $n$ denotes its number of vertices.

\emph{The contour function $C(T)$:}
First, it is useful to define the \textit{contour function} $(C_t(T),0\leq t \leq 2n)$ of $T$, which completely encodes the tree. To construct $C(T)$, imagine a particle exploring the tree from left to right at unit speed, starting from the root. For $t \in [0,2n-2]$, denote by $C_t(T)$ the distance of the particle to the root at time $t$. We set in addition $C_t(T)=0$ for $2n-2 \leq t \leq 2n$. See Fig. \ref{fig:arbconlam} for an example. By construction, $C(T)$ is continuous, nonnegative and satisfies $C_0(T)=C_{2n}(T)=0$.

\begin{figure}[!h]
\caption{A tree $T \in \kU_{11}^{(7)}$, its contour function $C(T)$, and the associated lamination $\bL(T)$.}
\label{fig:arbconlam}
\begin{tabular}{c c c}
\begin{tikzpicture}
\draw (1,2) -- (1,1) -- (0,0) -- (-1,1)--(0,2)--(0,3) (-1,1) -- (-2,2);

\draw[fill=black] (0,0) circle (.1);

\draw[fill=black] (1,1) circle (.1);

\draw[fill=black] (-1,1) circle (.1);

\draw[fill=black] (0,2) circle (.1);

\draw[fill=black] (0,3) circle (.1);

\draw[fill=black] (-2,2) circle (.1);

\draw[fill=black] (1,2) circle (.1);
\end{tikzpicture}
&
\begin{tikzpicture}[scale=.5, every node/.style={scale=0.7}]
\draw (0,0) -- (1,1) -- (2,2) -- (3,1) -- (4,2) -- (5,3) -- (6,2) -- (7,1) -- (8,0) -- (9,1) -- (10,2) -- (11,1) -- (12,0);
\draw[->] (0,-.5) -- (0,3.5);
\draw[->] (-.5,0) -- (15,0);
\draw (1,.1) -- (1,-.1);
\draw (1,-.4) node{1};
\draw (2,.1) -- (2,-.1);
\draw (2,-.4) node{2};
\draw (3,.1) -- (3,-.1);
\draw (3,-.4) node{3};
\draw (4,.1) -- (4,-.1);
\draw (4,-.4) node{4};
\draw (5,.1) -- (5,-.1);
\draw (5,-.4) node{5};
\draw (6,.1) -- (6,-.1);
\draw (6,-.4) node{6};
\draw (7,.1) -- (7,-.1);
\draw (7,-.4) node{7};
\draw (8,.1) -- (8,-.1);
\draw (8,-.4) node{8};
\draw (9,.1) -- (9,-.1);
\draw (9,-.4) node{9};
\draw (10,.1) -- (10,-.1);
\draw (10,-.4) node{10};
\draw (11,.1) -- (11,-.1);
\draw (11,-.4) node{11};
\draw (12,.1) -- (12,-.1);
\draw (12,-.4) node{12};
\draw (13,.1) -- (13,-.1);
\draw (13,-.4) node{13};
\draw (14,.1) -- (14,-.1);
\draw (14,-.4) node{14};
\draw (.1,1) -- (-.1,1);
\draw (-.4,1) node{1};
\draw (.1,2) -- (-.1,2);
\draw (-.4,2) node{2};
\draw (.1,3) -- (-.1,3);
\draw (-.4,3) node{3};
\draw (-.4,-.4) node{0};
\draw [dashed](1,1) -- (7,1);
\draw [dashed](4,2) -- (6,2);
\draw [dashed](9,1) -- (11,1);
\end{tikzpicture}
&
\begin{tikzpicture}[scale=1.5, every node/.style={scale=.9}]
\foreach \i in {0,...,13}
{
\draw[auto=right] ({1.1*cos(-(\i)*360/14)},{1.1*sin(-(\i)*360/14)}) node{\i};
\draw[red,fill=red] ({cos(-(\i-1)*360/14)},{sin(-(\i-1)*360/14)}) circle (.02);
}
\draw[red] (0,0) circle (1);
\draw[red] ({cos(-360*0/14)},{sin(-360*0/14)}) -- ({cos(-360*12/14)},{sin(-360*12/14)});
\draw[red] ({cos(-360*7/14)},{sin(-360*7/14)}) -- ({cos(-360*1/14)},{sin(-360*1/14)});
\draw[red] ({cos(-360*4/14)},{sin(-360*4/14)}) -- ({cos(-360*6/14)},{sin(-360*6/14)});
\draw[red] ({cos(-360*9/14)},{sin(-360*9/14)}) -- ({cos(-360*11/14)},{sin(-360*11/14)});
\end{tikzpicture}
\end{tabular}
\end{figure}
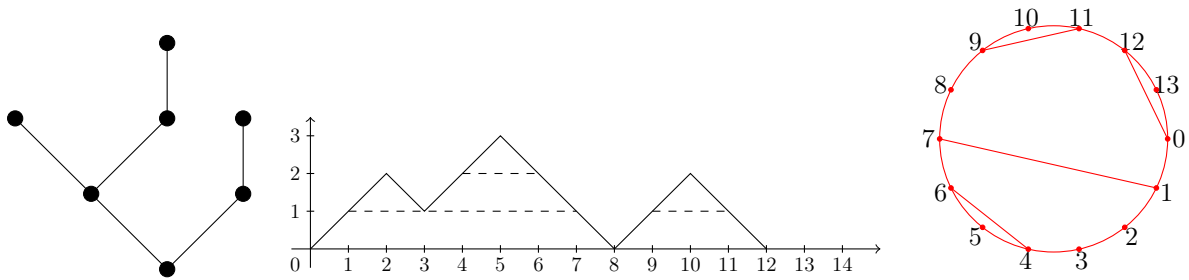

\emph{Chords and faces associated to vertices of $T$.}
We now propose a way of coding a vertex $x$ of $T$ by a chord in $\oD$: define $g(x)$ (resp. $d(x)$) the first (resp. last) time the particle performing this contour exploration is located at $x$, and denote by $c_x(T)$ the chord 
$$c_x(T) \coloneqq \left[e^{-2i\pi g(x)/2n}, e^{-2i\pi d(x)/2n}\right]$$
in $\oD$. We define the lamination associated to the tree $T$
$$\bL(T) \coloneqq \bS^1 \cup \bigcup_{x \in T} c_x(T).$$ One can indeed check that the chords $(c_x(T), x \in T)$ do not cross each other. See Fig. \ref{fig:arbconlam}, right for an example.

\emph{The lamination-valued process $(\bL_u(T))_{u \in [0,\infty]}$.} 
We derive here from $T$ a random nondecreasing lamination-valued process, which interpolates between $\bS^1$ and $\bL(T)$: at each integer time, we add a chord corresponding to a uniformly chosen vertex in the tree. More precisely, let $U_1$ be the root of $T$, and $U_2, \ldots, U_n$ be a uniform random permutation of the $n-1$ other vertices of $T$. Then, for $u \in [0,\infty]$, define
\begin{align*}
\bL_u(T) \coloneqq \bS^1 \cup \bigcup_{i =1}^{\lfloor u \rfloor \wedge n} c_{U_i}(T).
\end{align*} 
Remark notably that, for $u \geq n$, $\bL_u(T) = \bL(T)$.

\paragraph*{Lukasiewicz path of the tree}

We define here an other way to code the tree $T$, called its Lukasiewicz path and denoted by $(W_t(T))_{0 \leq t \leq n}$. It is constructed as follows: start from $W_0(T)=0$ and, for all $i \in \Z_+, \, i \leq n-1$, set $W_{i+1}(T)=W_i(T) + k_{v_i}(T) - 1$, where $v_r$ denotes the $(r+1)$-th vertex of $T$ in lexicographical order. Then, $W$ is the linear interpolation between these integer values. See an example on Fig. \ref{fig:luka}. 

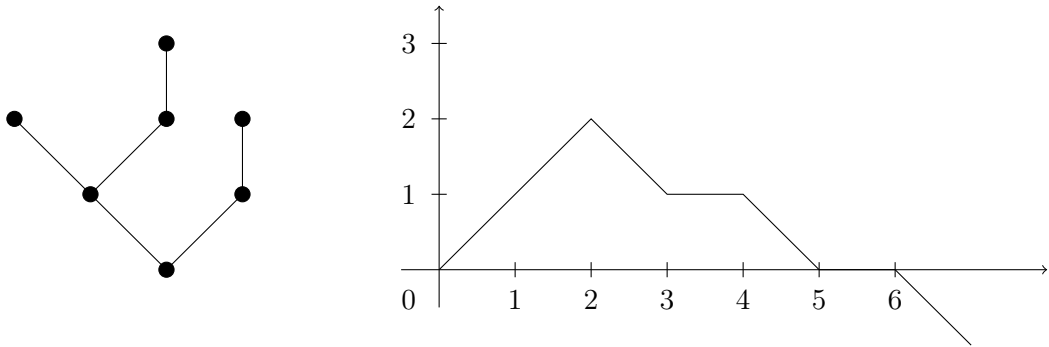
\begin{figure}[!h]
\center
\caption{A tree $T$ and its Lukasiewicz path $W(T)$.}
\label{fig:luka}
\begin{tabular}{c c c}
\begin{tikzpicture}
\draw (1,2) -- (1,1) -- (0,0) -- (-1,1)--(0,2)--(0,3) (-1,1) -- (-2,2);

\draw[white] (0,0)--(0,-1);

\draw[fill=black] (0,0) circle (.1);

\draw[fill=black] (1,1) circle (.1);

\draw[fill=black] (-1,1) circle (.1);

\draw[fill=black] (0,2) circle (.1);

\draw[fill=black] (0,3) circle (.1);

\draw[fill=black] (-2,2) circle (.1);

\draw[fill=black] (1,2) circle (.1);
\end{tikzpicture}
&
\begin{tikzpicture}
\draw[white] (0,0) -- (1,0);
\end{tikzpicture}
&
\begin{tikzpicture}
\draw (0,0) -- (1,1) -- (2,2) -- (3,1) -- (4,1) -- (5,0) -- (6,0) -- (7,-1); 
\draw[->] (0,-.5) -- (0,3.5);
\draw[->] (-.5,0) -- (8,0);
\draw (1,.1) -- (1,-.1);
\draw (1,-.4) node{1};
\draw (2,.1) -- (2,-.1);
\draw (2,-.4) node{2};
\draw (5,.1) -- (5,-.1);
\draw (5,-.4) node{5};
\draw (3,.1) -- (3,-.1);
\draw (3,-.4) node{3};
\draw (4,.1) -- (4,-.1);
\draw (4,-.4) node{4};
\draw (6,.1) -- (6,-.1);
\draw (6,-.4) node{6};
\draw (.1,1) -- (-.1,1);
\draw (-.4,1) node{1};
\draw (.1,2) -- (-.1,2);
\draw (-.4,2) node{2};
\draw (.1,3) -- (-.1,3);
\draw (-.4,3) node{3};
\draw (-.4,-.4) node{0};
\end{tikzpicture}
\end{tabular}
\end{figure}

One can check that $W_n(T)=-1$ and that, for all $t \leq n-1$, $W_t \geq 0$. This walk provides information on the degrees of the vertices of $T$, whereas the contour function rather allows to get information on the global shape of the tree.

\subsection{Bi-type trees}

We now give to a plane tree additional structure, by coloring each of its vertices either black or white - a tree whose vertices are not colored will be called monotype from now on.
We say that a finite plane tree $T$ is a \textit{bi-type tree} (in our context) if its vertices are colored white when their height is even and black when it is odd. In particular, white vertices only have black children and conversely. Notice that the root of a bi-type tree is white by definition. The number of white vertices in a bi-type tree $T$ is denoted by $N^\circ(T)$, and its number of black vertices by $N^\bullet(T)$. See Fig. \ref{fig:whitereduced}, left, for an example of bi-type tree. 

We say that $T$ is a \textit{labelled} bi-type tree if, in addition, its black vertices are labelled from $1$ to $N^\bullet(T)$. Such models of trees have already been studied in the past, notably by Bouttier, Di Francesco and Guitter \cite{BDG04} who establish a bijection between a class of planar maps and a class of labelled bi-type trees which they call mobiles.

Finally, for $n,k \geq 1$, we denote by $\kU_n^{(k)}$ the set of labelled bi-type trees with $n$ white vertices and $k$ black vertices, whose leaves are all white, in which the labels of the black neighbours (parent and children) of each white vertex are sorted in decreasing clockwise order (starting from one of these children), and in which the labels of the children of the root are sorted in decreasing order from left to right. Remark that, then, $\kU_n = \cup_{k \geq 1} \kU_n^{(k)}$.

\paragraph*{The white reduced tree.}

Let $T$ be a bi-type tree. We define the associated monotype \textit{white reduced tree} $T^\circ$, the following way:
\begin{itemize}
\item The vertices of $T^\circ$ are the white vertices of $T$.
\item A vertex $x$ is the child of a vertex $y$ in $T^\circ$ if and only if $x$ is a grandchild of $y$ in the original tree $T$.
\end{itemize}

This reduced tree encompasses the grandparent-grandchild relations between the white vertices in $T$. See Fig. \ref{fig:whitereduced} for a picture of a tree and of the associated white reduced tree. For convenience, we make no distinction between a white vertex of $T$ and the associated vertex of $T^\circ$.

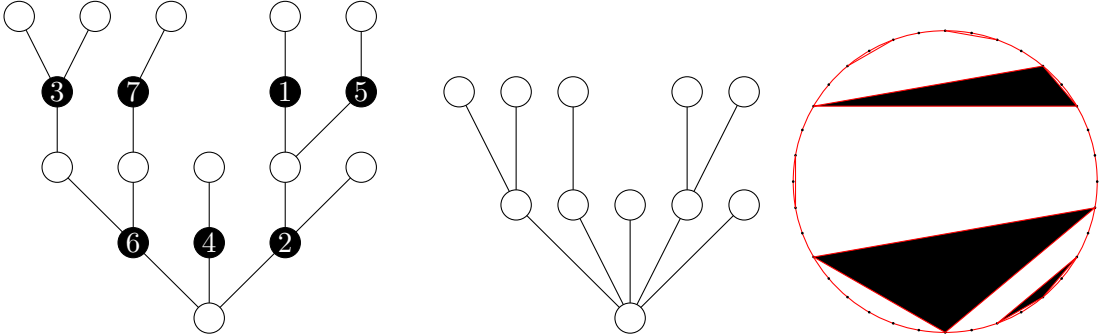
\begin{figure}[!h]
\center
\caption{A labelled bi-type tree $T$, its associated white reduced tree $T^\circ$ and the lamination $\bL^\bullet_6(T)$.}
\label{fig:whitereduced}
\begin{tabular}{c c c c}
\begin{tikzpicture}
\draw (-2.5,4) -- (-2,3) -- (-2,2) -- (-1,1) -- (-1,2) (-1,1) -- (0,0) -- (0,1) (0,0) -- (1,1) -- (2,2) (1,1) -- (1,2) -- (1,3) -- (1,4) (-1.5,4) -- (-2,3) (0,1) -- (0,2) (1,2) -- (2,3) -- (2,4) (-1,2) -- (-1,3) -- (-.5,4);

\draw[fill=white] (0,0) circle (.2);
\draw[fill=white] (-2,2) circle (.2);
\draw[fill=white] (-1,2) circle (.2);
\draw[fill=white] (0,2) circle (.2);
\draw[fill=white] (1,2) circle (.2);
\draw[fill=white] (2,2) circle (.2);
\draw[fill=white] (1,4) circle (.2);
\draw[fill=white] (-2.5,4) circle (.2);
\draw[fill=white] (-1.5,4) circle (.2);
\draw[fill=white] (-.5,4) circle (.2);
\draw[fill=white] (2,4) circle (.2);

\draw[fill=black] (-2,3) circle (.2) node[white]{$3$};
\draw[fill=black] (-1,1) circle (.2) node[white]{$6$};
\draw[fill=black] (-1,3) circle (.2) node[white]{$7$};
\draw[fill=black] (1,1) circle (.2) node[white]{$2$};
\draw[fill=black] (1,3) circle (.2) node[white]{$1$};
\draw[fill=black] (0,1) circle (.2) node[white]{$4$};
\draw[fill=black] (2,3) circle (.2) node[white]{$5$};
\end{tikzpicture}
&
\begin{tikzpicture}
\draw[white] (0,0) -- (0,1);
\end{tikzpicture}
&
\begin{tikzpicture}[scale=1.5]
\draw (-1,1) -- (0,0) -- (-.5,1) -- (-.5,2) (0,1) -- (0,0) -- (.5,1) -- (.5,2) (0,0) -- (1,1) (-1.5,2) -- (-1,1) -- (-1,2) (.5,1) -- (1,2);
\draw[fill=white] (-1,1) circle (.1333);
\draw[fill=white] (0,0) circle (.1333);
\draw[fill=white] (-.5,1) circle (.1333);
\draw[fill=white] (0,1) circle (.1333);
\draw[fill=white] (1,1) circle (.1333);
\draw[fill=white] (.5,1) circle (.1333);
\draw[fill=white] (.5,2) circle (.1333);
\draw[fill=white] (1,2) circle (.1333);
\draw[fill=white] (-1.5,2) circle (.1333);
\draw[fill=white] (-.5,2) circle (.1333);
\draw[fill=white] (-1,2) circle (.1333);
\end{tikzpicture}
&
\begin{tikzpicture}[scale=2]
\draw[red] (0,0) circle (1);
\foreach \k in {1,...,36}
{\draw[fill] ({cos(10*\k},{sin(10*\k)}) circle (.005);}
\draw[red,fill=black] ({cos(-1*10)},{sin(-1*10)}) -- ({cos(-9*10)},{sin(-9*10)}) -- ({cos(-15*10)},{sin(-15*10)}) -- cycle;
\draw[red,fill=black] ({cos(-21*10)},{sin(-21*10)}) -- ({cos(-31*10)},{sin(-31*10)}) -- ({cos(-33*10)},{sin(-33*10)}) -- cycle;
\draw[red,fill=black] ({cos(-23*10)},{sin(-23*10)}) -- ({cos(-25*10)},{sin(-25*10)});
\draw[red,fill=black] ({cos(-27*10)},{sin(-27*10)}) -- ({cos(-29*10)},{sin(-29*10)});
\draw[red,fill=black] ({cos(-5*10)},{sin(-5*10)}) -- ({cos(-3*10)},{sin(-3*10)}) -- ({cos(-7*10)},{sin(-7*10)}) -- cycle;
\draw[red,fill=black] ({cos(-17*10)},{sin(-17*10)}) -- ({cos(-19*10)},{sin(-19*10)});
\end{tikzpicture}
\end{tabular}

\end{figure}

\subsection{Colored laminations constructed from labelled bi-type trees}

We propose here two ways to code a finite labelled bi-type tree $T$ by discrete nondecreasing colored lamination-valued processes. The first one only takes into account white vertices, and is obtained from the contour function of the reduced tree $T^\circ$. The second one is obtained by considering the black vertices of $T$ and their labelling.

\paragraph*{The white process $\left(\bL^\circ_u(T)\right)_{u \in [0,\infty]}$}

The white process of a bi-type tree $T$ is the lamination-valued process presented in Section \ref{ssec:deftree}, applied to the white reduced tree $T^\circ$:
\begin{align*}
\bL^\circ_u(T) \coloneqq \bS^1 \cup \bigcup_{i=1}^{\lfloor u \rfloor \wedge N^\circ(T)} c_{U_i}(T^\circ).
\end{align*}
for any $u \in [0,\infty]$, where we recall that $U_1=\emptyset$ and $U_2, \ldots, U_n$ is a uniform permutation of the other vertices. Remark that this construction is not an injection, as it only depends on $C(T^\circ)$ while different bi-type trees may provide the same white reduced tree. Remark also that this process is only made of laminations, without any black point.

\paragraph*{The black process $\left(\bL^\bullet_u(T)\right)_{u \in [0,\infty]}$}

The black process of a bi-type tree $T$ is derived from the contour function $(C_t(T),0 \leq t \leq 2|T|)$ of the whole labelled bi-type tree $T$. Here, black vertices are coded by faces of a colored lamination, and not by chords as in the white process. More precisely, we define the face $F_x(T)$ associated to a vertex $x$ of $T$ the following way: let $0 \leq t_1 < t_2 < \ldots < t_{k_x(T)+1}$ be the times at which $x$ is visited by the contour function $C(T)$. Then define the associated face as:
$$F_x(T) \coloneqq Conv \left(\bigcup_{j=1}^{k_x(T)+1} \left[e^{-2i\pi t_j/2|T|}, e^{-2i\pi t_{j+1}/2|T|} \right] \right),$$
whose boundary is colored red and whose interior is colored black. In this definition, by convention, $t_{k_x(T)+2} = t_1$ and $Conv(A)$ denotes the convex hull of $A$.

Now, for $u \geq 0$, define
\begin{align*}
\bL^\bullet_u(T) \coloneqq \bS^1 \cup \bigcup_{i=1}^{\lfloor u \rfloor \wedge N^\bullet(T)} F_{V_i}(T).
\end{align*}
where $V_i$ is the black vertex labelled $i$ in $T$. The process $(\bL^\bullet_u(T))_{u \in [0,\infty]}$ is called the black process associated to $T$ (on Fig. \ref{fig:whitereduced} are represented a tree $T \in \kU_{11}^{(7)}$ (left) and the color lamination $\bL^\bullet_6(T)$).

\subsection{Random trees}
\label{ssec:randomtrees}

Let us now define random variables taking their values in the set of finite trees. We first define the so-called monotype simply generated trees, and then extend this framework to bi-type trees. We finally give some useful properties of both models. To avoid ambiguity, random monotype trees will be written with a straight double $\bT$, and random bi-type trees with a curved $\cT$.

\paragraph*{Monotype simply generated trees}

In the monotype case, we mostly rely on the deep survey of Janson \cite{Jan12} about simply generated trees, in which all proofs and further details can be found. Monotype simply generated trees (MTSG in short) were first introduced by Meir and Moon \cite{MM78}, and are random variables taking their values in the space of finite monotype trees. Specifically, for any $n \geq 1$, denote by $\kT_n$ the set of trees with $n$ vertices. Fix $w \coloneqq (w_i)_{i \geq 0} \in \R_+^{\Z_+}$ a weight sequence such that $w_0>0$ and, to a finite tree $T$, associate a weight $W_w(T) := \prod\limits_{x \in T} w_{k_x(T)}$. Then, for each $n \geq 1$, $n \geq 1$, we define the $w$-MTSG $\bT_n^w$ with $n$ vertices as the random variable satisfying, for any tree $T \in \kT_n$,
\begin{align*}
\P \left( \bT_n^w = T \right) = \frac{1}{Z_{n,w}} W_w(T)
\end{align*}
where
\begin{align*}
Z_{n,w} := \sum_{T \in \kT_n} W_w(T) .
\end{align*}
Since, for any $n \geq 1$, the set $\kT_n$ is finite, the tree $\bT_n^w$ is well-defined provided that $Z_{n,w}>0$, which we implicitly assume.

\begin{rk}
Here, weight sequences satisfy $w_0>0$, which was not the case for the weight sequences inducing factorizations of the $n$-cycle, defined in Section \ref{sec:intro}. Indeed, in the case of monotype trees the condition $w_0>0$ ensures that at least one finite tree has positive weight. We will still use the term 'weight sequence' for both.
\end{rk}

Remark that different weight sequences may provide the same simply generated tree:

\begin{lem}
\label{lem:monotypeequiv}
Let $w,w'$ be two weight sequences such that $w_0>0$, $w'_0>0$. Then, the two following assumptions are equivalent:

\begin{itemize}
\item[(i)] For any $n \geq 1$, $\bT_n^w \overset{(d)}{=} \bT_n^{w'}$
\item[(ii)] There exists $q,s>0$ such that, for all $i \geq 0$, $w_i = q s^i w'_i$. 
\end{itemize}
\end{lem}

\begin{proof}
The proof that (ii) $\Rightarrow$ (i) is essentially due to Kennedy \cite{Ken75} in a slightly different setting, and can be found in our form in \cite[$(4.3)$]{Jan12}.
In order to prove that (i) $\Rightarrow$ (ii), we proceed by induction on $n$. Assume without loss of generality that $w'_1 > 0$ (otherwise we just need to slightly adapt the proof). There exists a unique couple $(q,s) \in (\R_+^*)^2$ such that $w_0= q w'_0$ and $w_1=q s w'_1$. Now we consider the two different trees with $3$ vertices: one has weight $w_0 w_1^2$ and the other $w_0^2 w_2$, so that $Z_{2,w} = w_0 w_1^2 + w_0^2 w_2>0$, and as well $Z_{2,w'} = w_0' w_1'^2 + w_0'^2 w_2'>0$. Since $\bT_2^w$ and $\bT_2^{w'}$ have the same distribution, $Z_{2,w}^{-1} w_0 w_1^2 = Z_{2,w'}^{-1} w'_0 w_1'^2$, which implies that $Z_{2,w} = q^3 s^2 Z_{2,w'}$ and therefore that $w_2 = q s^2 w'_2$. By induction on $i \geq 2$, we get that $w_i=qs^iw'_i$ for all $i \geq 0$.
\end{proof}

\paragraph*{Galton-Watson trees}

When a weight sequence $\mu$ satisfies $\mu_0>0$ and $\sum_{i \geq 0} \mu_i = 1$, $\mu$ is a probability distribution and we can define a random variable $\bT^\mu$ such that, for any finite tree $T$, 
\begin{align*}
\P \left( \bT^\mu = T \right) = W_\mu(T) = \prod_{x \in T}\mu_{k_x(T)}.
\end{align*}
This variable is now defined on the whole set of finite trees, and not only on the subset of trees of fixed size. In this case, we say that $\bT^\mu$ is a $\mu$-Galton-Watson ($\mu$-GW in short) tree, and that $\mu$ is its offspring distribution. Thus, for any $n \geq 1$, the MTSG $\bT_n^\mu$ is a $\mu$-GW tree conditioned to have $n$ vertices. 

\paragraph*{Stable trees and stable processes}
Recall that the probability law $\mu$ is said to be critical if $\sum_{i \geq 0} i \mu_i= 1$. A particular case of GW trees is when the offspring distribution $\mu$ is critical and in the domain of attraction of an $\alpha$-stable law, for $\alpha \in (1,2]$.

If $\mu$ is in the domain of attraction of an $\alpha$-stable law, the sequence of trees $(\bT_n^\mu)_{n \geq 1}$, restricted to the values of $n$ such that $\P(|\bT^\mu|=n)>0)$, is known to have a scaling limit: these trees, seen as metric spaces for the graph distance and properly renormalized, converge in distribution, for the so-called Gromov-Hausdorff distance, to some random compact metric space, introduced by Duquesne and Le Gall \cite{DLG02} and called the $\alpha$-stable tree, which we denote by $\cT^{(\alpha)}$ (see Fig. \ref{fig:stablearbconlam}, left for a simulation of the $1.5$-stable tree $\cT^{(1.5)}$).

\begin{figure}[!h]
\center
\caption{A simulation of the $1.5$-stable tree $\cT^{(1.5)}$, the stable height process $H^{(1.5)}$ and the process $X^{(1.5)}$.}
\label{fig:stablearbconlam}
\begin{tabular}{c c}
\includegraphics[scale=.5]{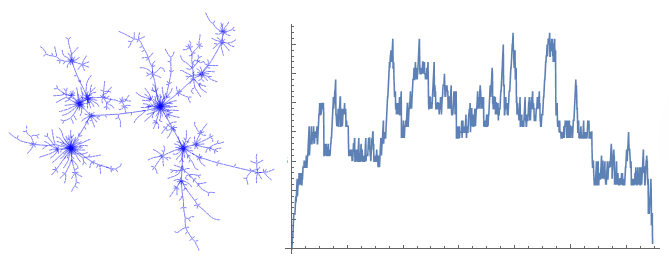}
&
\includegraphics[scale=.35]{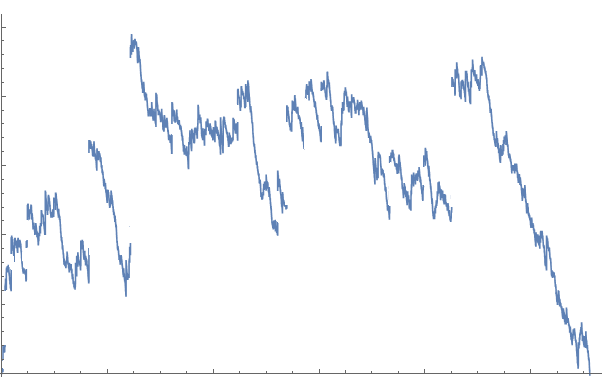}
\end{tabular}
\end{figure}

These trees have recently become a topic of interest for probabilists. In particular, a fundamental result states that, jointly with the convergence of the renormalized trees $(\bT_n^\mu)_{n \geq 1}$ towards $\cT^{(\alpha)}$, their contour functions and Lukasiewicz paths also converge, after renormalization, to some limiting càdlàg random processes $(X_t^{(\alpha)})_{0 \leq t \leq 1}$ and $(H_t^{(\alpha)})_{0 \leq t \leq 1}$ respectively, which can therefore be seen as the analogues of the Lukasiewicz path and the contour function of these stable trees. See Fig. \ref{fig:stablearbconlam} for a simulation of $H^{(1.5)}$ (middle) and $X^{(1.5)}$ (right).

\begin{thm}
\label{thm:cvcontour}
Let $\alpha \in (1,2]$ and let $\mu$ be a probability distribution in the domain of attraction of an $\alpha$-stable law. Let $(B_n)_{n \geq 0}$ be a sequence verifying \eqref{eq:Bn}. Then, in distribution in $\D([0,1], \R)^2$:
\begin{align*}
\left( \frac{1}{B_n} W_{nt}\left(\bT_n^\mu\right), \frac{B_n}{n} C_{2nt}\left(\bT_n^\mu\right) \right)_{t \in [0,1]} \overset{(d)}{\underset{n \rightarrow \infty}{\rightarrow}} \left( X_t^{(\alpha)}, H_t^{(\alpha)} \right)_{t \in [0,1]}
\end{align*}
\end{thm}

This result is due to Marckert and Mokkadem \cite{MaMo03} under the assumption that $\mu$ has a finite exponential moment (that is, $\sum_{i \geq 0} \mu_i e^{\beta i}>0$ for some $\beta>0$). The result in the general case can be deduced from the work of Duquesne \cite{Duq03}, although it is not clearly stated in this form. See \cite[Theorem $8.1$, (II)]{Kor12} (taking $\mathcal{A}=\Z_+$ in this theorem) for a precise statement.

In light of this convergence, investigating properties of these limiting objects allows to obtain information on the shape of a typical realization of the tree $\bT_n^\mu$ for $n$ large. Let us immediately see an example. For $\alpha \in (1,2)$, the limiting process $X^{(\alpha)}$ satisfies the following properties with probability $1$ where, for $s \in (0,1]$, we have set $\Delta_s \coloneqq X^{(\alpha)}_{s} - X^{(\alpha)}_{s-}$.
\begin{itemize}
\item[(H1)] The local minima of $X^{(\alpha)}$ are distinct (that is, for any $0 \leq s < t \leq 1$, there exists at most one $r\in (s,t)$ such that $X^{(\alpha)}_r = \inf_{[s,t]} X^{(\alpha)}$).
\item[(H2)] Let $t$ be a local minimum of $X^{(\alpha)}$ (i.e. $X^{(\alpha)}_t = \min \{ X^{(\alpha)}_u, t-\epsilon \leq u \leq t+\epsilon\}$ for some $\epsilon>0$), and define $s \coloneqq \sup \left\{ r \leq t, X^{(\alpha)}_{r} < X^{(\alpha)}_{t} \right\}$. Then $\Delta_s>0$, and $X^{(\alpha)}_{s-} < X^{(\alpha)}_{t} < X^{(\alpha)}_{s}$.
\item[(H3)] If $s \in (0,1)$ is such that $\Delta_s>0$, then for all $0 \leq \epsilon \leq s$, $\inf_{[s-\epsilon, s]} X^{(\alpha)} < X^{(\alpha)}_{s-}$.
\end{itemize}
These three properties are notably used in \cite{Kor14}, in order to construct the lamination $\bL_\infty^{(\alpha)}$, and are proved in \cite[Proposition $2.10$]{Kor14}. The following lemma, which is a consequence of these properties of $X^{(\alpha)}$, provides useful information about the structure of a large $\mu$-Galton-Watson tree:

\begin{lem}
\label{lem:grossommets}
Let $\alpha<2$, and let $\mu$ be a critical distribution in the domain of attraction of an $\alpha$-stable law. Then:
\begin{itemize}
\item[(i)] For any $\epsilon>0$, there exists $\eta>0$ such that, for all $n$ large enough:
\begin{align*}
\P \left( \exists \, u \in \bT_n^\mu, k_u\left(\bT_n^\mu\right) \geq \eta B_n \right) \geq 1-\epsilon.
\end{align*}

\item[(ii)] For any $\epsilon>0$, there exists $\eta>0$ such that, for $n$ large enough, with probability larger than $1-\epsilon$, any $\epsilon n$-node in $\bT_n^\mu$ has more than $\eta B_n$ children.
\end{itemize}
\end{lem}

In other terms, (i) means that, with high probability, there is at least one vertex in $\bT_n^\mu$ whose number of children is of order $B_n$. Furthermore, (ii) states that all $\epsilon n$-nodes of $\bT_n^\mu$ (which appear to correspond to large faces in the associated lamination $\bL(\bT_n^\mu)$) have a number of children of order $B_n$.

\begin{proof}[Proof of Lemma \ref{lem:grossommets}]
The proof of (i) is straightforward: it is known that the set of points $t \in [0,1]$ where $\Delta_t>0$ is dense in $[0,1]$ (for instance, the process satisfies Assumption $(H0)$ in \cite{Kor14}). In particular, almost surely, $M \coloneqq \max \, \{\Delta_t, t \in [0,1] \}>0$. Fix $\epsilon>0$, and take $\eta>0$ such that $M>2\eta$ with probability $\geq 1-\epsilon/2$. Then, by the convergence of Theorem \ref{thm:cvcontour}, for $n$ large enough, the maximum degree in $\bT_n^\mu$ is larger than $\eta B_n$ with probability $\geq 1-\epsilon$.

Let us now prove (ii). For $x$ a vertex of $\bT_n^\mu$, denote by $i(x)$ the position of $x$ in $\bT_n^\mu$ in lexicographical order.  Take $u$ an $\epsilon n$-node in $\bT_n^\mu$. In particular, $W_{i(u)}-W_{i(u)-1} = k_u(\bT_n^\mu)-1$. By definition of an $\epsilon n$-node, its children can be split into two subsets $A_r(u), A_{-r}(u)$ for some $r \geq 1$, such that $\sum_{w \in A_r(u)} |\theta_w(\bT_n^\mu)| \geq \epsilon n$ and $\sum_{w \in A_{-r}(u)} |\theta_w(\bT_n^\mu)| \geq \epsilon n$. Let $v(u) \in \bT_n^\mu$ be the first vertex of $A_{-r}(u)$ in lexicographical order. Then, it is clear by definition of $W$ that
\begin{equation} 
\label{eq:Wiu}
W_{i(u)-1}(\bT_n^\mu) \leq W_{i(v(u))-1}(\bT_n^\mu) \leq W_{i(u)}(\bT_n^\mu).
\end{equation}
Now assume by the Skorokhod representation theorem that the convergence of Theorem \ref{thm:cvcontour} holds almost surely. 
Assume that, for $n$ along a subsequence, there exists an $\epsilon n$-node $u_n$ in $\bT_n^\mu$ such that $W_{i(u_n)}(\bT_n^\mu) - W_{i(u_n)-1}(\bT_n^\mu) = o(B_n)$ as $n \rightarrow \infty$ (that is, $u_n$ has $o(B_n)$ children). Since $[0,1]$ is compact, up to extraction, one can assume in addition that there exists $0<s<t<1$ such that $i(u_n)/n \rightarrow s$ and $i(v(u_n))/n \rightarrow t$. Thus, the limiting process $X^{(\alpha)}$ should satisfy: 
\begin{itemize}
\item[(a')]
$X^{(\alpha)}_{s} = X^{(\alpha)}_{t-}$
\item[(b')] $X^{(\alpha)}_{t-} = \inf_{[t-\epsilon, t+\epsilon]} X^{(\alpha)}$.
\end{itemize}
Indeed, (a') can be deduced from \eqref{eq:Wiu} and the fact that $W_{i(u_n)}(\bT_n^\mu) - W_{i(u_n)-1}(\bT_n^\mu) = o(B_n)$, while (b') 
comes from (a') along with the fact that $W$ is larger than $W_{i(u_n)}(\bT_n^\mu)-1$ on $[i(u_n), i(v(u_n))+2\epsilon n]$ as $|\theta_u(\bT_n^\mu)| \geq 2 \epsilon n$. There are now two possible cases:
\begin{itemize}
\item first, if $\Delta_s=0$, then, as by (H1) the local minima of $X^{(\alpha)}$ are almost surely distinct, $X^{(\alpha)}_s$ is not a local minimum of $X^{(\alpha)}$ (since by (b') $X^{(\alpha)}_{t-}$ is a local minimum). Therefore, $s= \sup \{r \leq t, X^{(\alpha)}_r < X^{(\alpha)}_t \}$ and by (H2) $\Delta_s>0$, which contradicts our assumption;
\item second, if $\Delta_s>0$ then by (H3) $s= \sup \{r \leq t, X^{(\alpha)}_r < X^{(\alpha)}_t \}$; by (H2), it should happen that $X^{(\alpha)}_t<X^{(\alpha)}_s$, which is not the case by (a').
\end{itemize}  

In conclusion, almost surely, there exists $\eta>0$ such that, as $n \rightarrow \infty$, all $\epsilon n$-nodes in $\bT_n^\mu$ have at least $\eta B_n$ children.
\end{proof}

We can now present a first result of convergence concerning the lamination-valued process associated to the trees $\bT_n^\mu, n \geq 1$. As the renormalized contour functions of these trees converge as $n \rightarrow \infty$ to $H^{(\alpha)}$ by Theorem \ref{thm:cvcontour}, it appears that the associated processes of laminations converge towards the $\alpha$-stable lamination-valued process $(\bL_c^{(\alpha)})_{c \in [0,\infty]}$.

\begin{thm}
\label{thm:cvlamgw}
Let $\alpha \in (1,2]$, $\mu$ a distribution in the domain of attraction of an $\alpha$-stable law, and $(B_n)$ satisfying \eqref{eq:Bn}. Then, jointly with the convergence of Theorem \ref{thm:cvcontour}, the following holds in distribution in $\D([0,\infty],\bCL(\oD))$:
\begin{align*}
\left( \bL_{cB_n}\left(\bT_n^\mu\right) \right)_{c \in [0,\infty]} \overset{(d)}{\rightarrow}\left(\bL_c^{(\alpha)}\right)_{c \in [0,\infty]},
\end{align*}
where we recall that the process $(\bL_c^{(\alpha)})_{c \in [0,\infty]}$ is obtained from $H^{(\alpha)}$ by the construction of Section \ref{sec:intro}.
\end{thm} 
This result is an immediate consequence of \cite[Theorem $4.3$ and Proposition $4.3$]{The19}, and is a cornerstone of the proof of \eqref{eq:cvthe19}.

To end this section on random monotype trees, we provide a useful tool in the study of Galton-Watson trees called the local limit theorem. It can be seen for instance as a consequence of \cite[Theorem $8.1$, (I)]{Kor12}, taking $\mathcal{A}
=\Z_+$ in the statement:

\begin{thm}[Local limit theorem]
\label{thm:llt}
Let $\mu$ be a critical distribution in the domain of attraction of a stable law, and $(B_n)$ satisfying \eqref{eq:Bn}. Then, there exists a constant $C>0$ such that, for the values of $n$ for which $\P( |\bT^\mu|=n)>0$, 
\begin{align*}
\P \left( |\bT^\mu|=n \right) \sim \frac{C}{n B_n}
\end{align*} 
as $n \rightarrow \infty$.
\end{thm}

In particular, $\P(|\bT^\mu|=n)$ decreases more slowly than some polynomial in $n$.

\paragraph*{Bi-type simply generated trees}

We now define the bi-type analogue of MTSG trees, which we call bi-type simply generated trees (in short, BTSG). Such random bi-type trees notably appear in \cite{LGM10}.

Let $w^\circ, w^\bullet$ be two weight sequences, and impose that $w^\bullet_0=0$ and $w^\circ_0>0$.
For $T$ a bi-type tree, define the weight of $T$ as 
$$W_{\wb,\wn}(T) :=  \prod\limits_{x \in T, x \text{ white}} \wb_{k_x(T)} \times \prod\limits_{y \in T, y \text{ black}} \wn_{k_y(T)}.$$

An integer $n$ being fixed, a $(\wb, \wn)$-BTSG with $n$ white vertices is a random variable $\cT_n^{(\wb, \wn)}$, taking its values in the set $\kBT_n$ of bi-type rooted trees with $n$ white vertices, such that the probability that $\cT_n^{(\wb, \wn)}$ is equal to some bi-type tree $T \in \kBT_n$ is
\begin{align*}
\P\left( \cT_n^{(\wb, \wn)} = T \right) = \frac{1}{Z_{n,\wb,\wn}} W_{\wb,\wn}(T).
\end{align*}
Here, $Z_{n,\wb,\wn} = \sum_{T \in \kBT_n} W_{\wb,\wn}(T)$ is a renormalization constant (as usual, we shall restrict ourselves to the values of $n$ such that $Z_{n,\wb,\wn}>0$). Note that, since we impose the condition $w^\bullet_0=0$, the set $\{ T \in \kBT_n, W_{\wb,\wn}(T)>0 \}$ is finite at $n$ fixed and therefore $Z_{n,\wb,\wn}<\infty$. In addition, the leaves of $\cT_n^{(\wb, \wn)}$ are all white, and the number of black vertices in $\cT_n^{(\wb,\wn)}$ is at most $n-1$.

As in the monotype case, different couples $(\wb,\wn)$ may give the same BTSG.

\begin{lem}[Exponential tilting]
\label{lem:equivalence}
Take two sequences $\wb,\wn$ such that $\wn_0=0$ and $\wb_0>0$. Take $p,q,r,s \in \R_+^*$ such that $qr=1$, and define two new weight sequences $\twb,\twn$ as, for $i \in \Z_+$,
\begin{align*}
\twb_i=pq^i\wb_i, \twn_i=rs^i\wn_i.
\end{align*}
Then, for all $n\geq 1$, $\cT_n^{(\wb, \wn)}$ has the same distribution as $\cT_n^{(\twb, \twn)}$.
\end{lem}

In this case, we say that $(\wb,\wn)$ and $(\twb,\twn)$ are two \textit{equivalent} couples of weight sequences (one easily checks that this indeed defines an equivalence relation on the set of couples of weight sequences $(\wb,\wn)$ such that $\wn_0=0$ and $\wb_0>0$).

\begin{proof}
Fix $n \geq 1$. Take $T$ a bi-type tree with $n$ white vertices, and denote by $k$ the number of black vertices in $T$. Then, remark that 

\begin{align*}
W_{\twb,\twn}(T) &= \prod\limits_{x \in T, \, x \text{ white}} \twb_{k_x(T)} \times \prod\limits_{y \in T, \, y \text{ black}} \twn_{k_y(T)}\\
&= \prod\limits_{x \in T, \, x \text{ white}} p q^{k_x(T)}\wb_{k_x(T)} \times \prod\limits_{y \in T, \, y \text{ black}} r s^{k_y(T)} \wn_{k_y(T)}\\
&= p^n q^k r^k s^{n-1} \prod\limits_{x \in T, \, x \text{ white}} \wb_{k_x(T)} \times \prod\limits_{y \in T, \, y \text{ black}} \wn_{k_y(T)}\\
&= p^n s^{n-1} W_{\wb, \wn}(T) \text{ since } qr=1.
\end{align*}
This implies in particular that $Z_{n,\twb,\twn} = p^n s^{n-1} Z_{n,\wb,\wn}$. Thus, for any tree $T \in \kBT_n$,
\begin{align*}
\P \left( \cT_n^{(\twb, \twn)}=T \right) = \P \left( \cT_n^{(\wb, \wn)}=T \right),
\end{align*}
which provides the result.
\end{proof}

\textbf{From now on, unless explicitly mentioned, the tree $\cT_n^{(\wb,\wn)}$ will always be considered as a labelled bi-type tree, whose black vertices are labelled uniformly at random from $1$ to $N^\bullet(\cT_n^{(\wb,\wn)})$.}

Finally, we end this section by stating a bi-type analogue of Theorem \ref{thm:cvlamgw}, in the case of a size-conditioned $(\mu_*,\nu)$-BTSG. Here, $\mu_*$ denotes the Poisson distribution of parameter $1$ (that is, for all $i \geq 0$, $(\mu_*)_i=e^{-1} i!^{-1}$), and $\nu$ is a probability distribution satisfying either one of the following two conditions:
\begin{itemize}
\item[(I)] There exists $\alpha \in (1,2)$ such that $\nu$ is in the domain of attraction of an $\alpha$-stable law.
\item[(II)] $\nu$ has finite variance $\sigma_\nu^2$ (in which case we recall that $\nu$ is in the domain of attraction of a $2$-stable law).
\end{itemize}

\begin{thm}
\label{thm:cvbitypelam}
Let $\nu$ be a probability law satisfying (I) or (II), and let $\wn$ be the weight sequence defined as $\wn_0=0$ and $\wn_i=\nu_i$ for $i \geq 1$. Let $(\tilde{B}_n)$ be a sequence satisfying \eqref{eq:tbn}. Then the following convergence holds in the space $\D(\R_+, \bCL(\oD)) \times \bCL(\oD)$.
\begin{itemize}
\item[(i)] In case (I),
\begin{align*}
\left( \left( \bL^\bullet_{c (1-\nu_0) \tilde{B}_n}\left(\cT_n^{(\mu_*, \wn)}\right) \right)_{0 \leq c < \infty}, \bL_\infty^\bullet\left(\cT_n^{(\mu_*, w^\bullet)}\right) \right) \underset{n \rightarrow \infty}{\overset{(d)}{\rightarrow}} \left( \left(\bL_c^{(\alpha)} \right)_{0 \leq c < \infty}, \bL_\infty^{(\alpha),1}\right).
\end{align*}

\item[(ii)] In case (II),
\begin{align*}
\left( \left( \bL^\bullet_{c (1-\nu_0) \tilde{B}_n}\left(\cT_n^{(\mu_*, \wn)}\right) \right)_{0 \leq c < \infty}, \bL_\infty^\bullet\left(\cT_n^{(\mu_*, \wn)}\right) \right) \underset{n \rightarrow \infty}{\overset{(d)}{\rightarrow}} \left( \left(\bL_c^{(2)} \right)_{0 \leq c < \infty}, \bL_\infty^{(2),p_\nu}\right),
\end{align*}
where $$p_\nu \coloneqq \frac{\sigma_\nu^2}{\sigma_\nu^2+1} \in [0,1).$$
\end{itemize}
\end{thm}
The proof of this theorem, which is quite technical, is postponed to Section \ref{ssec:proof}. Studying this particular family of BTSG is of great interest in our case, since they code in some sense a $w$-minimal factorization of the $n$-cycle, as will be seen in Section \ref{sec:minfac}.

\begin{rk}
Although we state and prove Theorem \ref{thm:cvbitypelam} only in the specific case $w^\circ=\mu_*$ for its connection with minimal factorizations, this result holds in a more general framework. Specifically, let $\nu^\circ$ and $\nu^\bullet$ be two critical probability distributions, and $w^\bullet$ a weight sequence such that $w^\bullet_0=0$, whose critical equivalent is $\nu^\bullet$. Then, Theorem \ref{thm:cvbitypelam} still holds in the following more general cases (I') and (II'):
\begin{itemize}
\item[(I')] $\nu^\bullet$ is in the domain of attraction of an $\alpha$-stable law for $1<\alpha<2$, and $\nu^\circ$ has a finite moment of order $2+2\alpha$. This seemingly strange condition appears when one adapts the proof of Lemma \ref{lem:pnu<2}.
\item[(II')] both $\nu^\circ$ and $\nu^\bullet$ have finite variance. In this framework, $p_\nu$ shall be replaced by a parameter $p$ which depends on both $\nu^\circ$ and $\nu^\bullet$.
\end{itemize}
In these two cases, all further proofs can be easily adapted.
\end{rk}

\begin{rk}
Using the same tools as for the proof of Theorem \ref{thm:cvbitypelam}, one can in fact prove that the following slightly stronger convergence holds in the space $\D(\R_+, \bCL(\oD)) \times \D((0,1], \bCL(\oD))$:
\begin{itemize}
\item[(i)] In case (I),
\begin{align*}
\left( \left( \bL^\bullet_{c (1-\nu_0) \tilde{B}_n}\left(\cT_n^{(\mu_*, \wn)}\right) \right)_{0 \leq c < \infty}, \left(\bL_{dn}^\bullet\left(\cT_n^{(\mu_*, w^\bullet)}\right) \right)_{0 < d \leq 1} \right) \underset{n \rightarrow \infty}{\overset{(d)}{\rightarrow}} \left( \left(\bL_c^{(\alpha)} \right)_{0 \leq c < \infty}, \left( \bL_{\infty,d}^{(\alpha),1} \right)_{0 < d \leq 1} \right).
\end{align*}

\item[(ii)] In case (II),
\begin{align*}
\left( \left( \bL^\bullet_{c (1-\nu_0) \tilde{B}_n}\left(\cT_n^{(\mu_*, \wn)}\right) \right)_{0 \leq c < \infty}, \left( \bL_{dn}^\bullet\left(\cT_n^{(\mu_*, \wn)}\right) \right)_{0 < d \leq 1} \right) \underset{n \rightarrow \infty}{\overset{(d)}{\rightarrow}} \left( \left(\bL_c^{(2)} \right)_{0 \leq c < \infty}, \left( \bL_{\infty,d}^{(2),p_\nu} \right)_{0 < d \leq 1} \right),
\end{align*}
where $$p_\nu \coloneqq \frac{\sigma_\nu^2}{\sigma_\nu^2+1}.$$
\end{itemize}
Here, for $p \in [0,1]$ and $\alpha \in (1,2]$, the process $\left( \bL_{\infty,d}^{(\alpha),p} \right)_{0 < d \leq 1}$ that appears at the limit interpolates in a 'linear' way between the stable lamination $\bL_\infty^{(\alpha)}$ (for $d \downarrow 0$) and the colored stable lamination $\bL_\infty^{(\alpha),p}$ (for $d=1$). To construct this process, start from the stable lamination and sort its faces by decreasing area. Associate to the face $F_i$ labelled $i$ a couple of independent variables $(X_i,Y_i)$, all these couples $(X_i,Y_i)_{i \geq 1}$ being independent. For any $i$, the variable $X_i$ is binomial of parameter $p$, and determines whether the face $F_i$ is colored in $\bL_\infty^{(\alpha),p}$ or not. The variable $Y_i$ is uniform on $[0,1]$: if $X=1$ (that is, $F_i$ is colored at the limit) then $Y_i$ is the time at which it is colored in the process. More rigorously, for any $d \in (0,1]$, 
\begin{align*}
\bL_{\infty,d}^{(\alpha),1} \coloneqq \bL_\infty^{(\alpha)} \cup \bigcup_{\substack{i \geq 1 \\ X_i=1, Y_i \leq d}} F_i^\bullet
\end{align*}
where $F_i^\bullet$ denotes the face $F_i$ colored black. 
In words, we state that the faces that are colored in the limiting lamination $\bL_\infty^{(\alpha),p}$ in case (I) and (II) appear in the process at independent times $Y_i n$, where the $Y_i$'s are uniform on $(0,1]$. In some sense, this marks a second phase transition at scale $n$, in which large black faces begin to appear.
\end{rk}

\section{The particular case of $(\mu_*,w)$-bi-type trees.}
\label{sec:particulartrees}

This section is devoted to the study of particular properties of $(\mu_*,w)$-BTSG trees, where $\mu_*=Po(1)$ and $w$ is a weight sequence of stable type. Indeed, such trees appear as a natural coding of minimal factorizations of stable type, as we will see in Section \ref{ssec:image}.
In particular, we characterize the distribution of the associated white reduced tree, and use it to prove Theorem \ref{thm:cvbitypelam}. \textbf{From now on, as there is no ambiguity, we write $\cT_n$ instead of $\cT_n^{(\mu_*,w)}$, and $\cT_n^\circ$ instead of $\cT_n^{\circ, (\mu_*,w)}$.}

\subsection{Reachable distributions: a study of the white reduced tree.}

Recall that, a weight sequence $(w_i)_{i \geq 1}$ being given, there exists at most one critical probability distribution $\nu$ called the critical equivalent of $w$ such that, for some $s>0$, for all $i \geq 1$, $\nu_i = w_i s^i$. If it is the case, the white reduced tree $\cT_n^\circ$ is distributed as a Galton-Watson tree. The goal of this section is to investigate the possible behaviours of its offspring distribution, and notably for which sequences $w$ this white reduced tree converges to a stable tree. Indeed, in this case, the white process of $\cT_n$ converges by Theorem \ref{thm:cvlamgw}, which helps us prove the convergence of the black process of $\cT_n$.

In the rest of the paper, in the case of a $(\mu_*,w)$-BTSG tree, $\nu$ will always denote the critical equivalent of $w$, and $\mu$ the critical distribution such that $\cT_n^\circ \overset{(d)}{=} \bT_n^\mu$.

Let us state things formally. To a weight sequence $w$, we associate its generating function $F_w: x \rightarrow \sum\limits_{i = 0}^\infty w_i x^i$. In particular, $F_{\mu_*}(x) \coloneqq e^{x-1}$ is defined for any $x \in \R$. It is a simple matter of fact that the white reduced tree $\cT_n^{\circ}$ is distributed as the monotype simply generated tree $\bT_n^{\tilde{w}}$, where $\tilde{w}$ is the weight sequence whose generating function satisfies 
$$F_{\tilde{w}} = F_{\mu_*} \circ F_w := e^{F_w-1}.$$

We say that a critical probability distribution $\mu$ is \textit{reachable} if there exists a weight sequence $(w_i)_{i \geq 1}$ such that, for all $n \geq 1$, $\cT_n^{\circ}$ is a $\mu$-Galton-Watson tree conditioned to have $n$ vertices. In this case, we say that $w$ \textit{reaches} $\mu$. The following theorem states that a large range of distributions are reachable:

\begin{thm}
\label{thm:reachabledistributions}
Let $\alpha \in (1,2)$ and $L$ a slowly varying function. Then, there exists a reachable critical distribution $\mu$ such that
\begin{equation}
\label{eq:stable}
F_\mu(1-u)-(1-u) \underset{\substack{u \downarrow 0 \\ u \neq 0}}{\sim} u^\alpha L\left(u^{-1}\right).
\end{equation}

Let $\alpha=2$ and $\ell>0$. Then there exists a reachable critical distribution $\mu$  such that
\begin{equation}
\label{eq:vf}
F_\mu(1-u)-(1-u) \underset{\substack{u \downarrow 0 \\ u \neq 0}}{\rightarrow} \ell
\end{equation}
if and only if $\ell \geq 1/2$. This is equivalent to saying that $\mu$ has finite variance $2\ell$.

Furthermore, let $w$ be a weight sequence that reaches a critical probability distribution $\mu$. Then the two following points are equivalent:
\begin{itemize}
\item $\mu$ verifies \eqref{eq:stable} or \eqref{eq:vf}.
\item the critical equivalent $\nu$ of $w$ satisfies:
\begin{align*}
F_\nu(1-u)-(1-u) \underset{\substack{u \rightarrow 0 \\ u \neq 0}}{\sim}
\begin{cases}\displaystyle 
 \quad  u^\alpha L\left(u^{-1}\right) \text{ if } \alpha \in (1,2) \\
\displaystyle \quad   u^2 \left( \ell-1/2 \right) \text{ if } \mu \text{ has finite variance.}
\end{cases}  
\end{align*}
\end{itemize}
\end{thm}

In other terms, let $(B_n)_{n \geq 1}$ be a sequence of positive numbers satisfying \eqref{eq:Bn} for $\nu$, and $(\tilde{B}_n)_{n \geq 1}$ be the sequence constructed from $(B_n)_{n \geq 1}$ as in \eqref{eq:tbn}. Then, $(\tilde{B}_n)$ satisfies \eqref{eq:Bn} for $\mu$. This relation between both is the reason of the scaling in $\tilde{B}_n$ appearing in both Theorems \ref{thm:mainresult} and \ref{thm:cvbitypelam}.

Theorem \ref{thm:reachabledistributions} also means that the only way to reach a distribution in the domain of attraction of a stable law is to start from a weight sequence whose critical equivalent is already in the domain of attraction of a stable law.

Theorem \ref{thm:reachabledistributions} is the consequence of a general result about reachable distributions, which may be of independent interest: a reachable $\mu$ being given, all weight sequences that reach it are closely related.

\begin{prop}
\label{prop:equiv}
Let $\mu$ be reachable and $w$ a weight sequence reaching $\mu$. Then, for any weight sequence $w'$, $w'$ reaches $\mu$ if and only if there exists $s,t>0$ such that
\begin{align*}
F_{w'}(tx) - F_{w'}(t) = F_w(sx) - F_w(s).
\end{align*}
\end{prop}

In particular, the set of weight sequences reaching $\mu$ can be written $\{ w^{(s)}, s \in \R_+^* \}$ defined as: for any $s>0$, any $i \geq 1$, $w^{(s)}_i := w_i s^i$.

\begin{proof}
Let $w$ be a weight sequence reaching $\mu$, and let $\tilde{w}$ be the weight sequence such that $F_{\tilde{w}} = e^{F_w-1}$. Then $\tilde{w}$ shall satisfy for some $q,s>0$, by Lemma \ref{lem:monotypeequiv},

$$\forall x \in [-1,1], F_\mu(x) = q F_{\tilde{w}}(sx) = q e^{F_w(sx)-1}.$$
Applying this for $x=1$, one gets $1=q e^{F_w(s)-1}$, which finally gives:

\begin{align*}
\forall x \in [-1,1], F_\mu(x) = e^{F_w(sx)-F_w(s)}.
\end{align*}
The result directly follows.
\end{proof}

Let us see how it implies Theorem \ref{thm:reachabledistributions}:

\begin{proof}[Proof of Theorem \ref{thm:reachabledistributions}]
We start by proving the first part of this theorem. When $\alpha<2$, one can define $\nu$ a critical distribution satisfying $\nu_k = k^{-1-\alpha} L(k)$, for $k$ large enough. Thus, $F_\nu(1-u)-(1-u) \sim u^\alpha (L(u^{-1}))$ as $u \rightarrow 0$, by e.g. \cite[Theorem $8.1.6$]{BGT89}. Define now the weight sequence $w$ by $w_0=0$ and $w_i=\nu_i$ for $i \geq 1$. In particular, $F_w(1-u)-F_w(1)=F_\nu(1-u)-F_\nu(1)=-u+u^\alpha L(u^{-1}) (1+o(1))$, and $F_w'(1)=1$. Then, one can check that the probability law $\mu$ such that $F_\mu(x)=e^{F_w(x)-F_w(1)}$ is reached by $w$ and is critical. One gets in addition:
\begin{align*}
F_\mu(1-u) = e^{F_w(1-u)-F_w(1)} &= e^{F_\nu(1-u)-1}\\
&= F_\nu(1-u) + \frac{1}{2} \left( (F_\nu(1-u)-1)^2 \right)\\
&= 1-u + u^\alpha L\left( u^{-1} \right) + o(u^\alpha),
\end{align*}
which implies the first part of Theorem \ref{thm:reachabledistributions}.

When $\alpha=2$ and $L(x) \underset{x \rightarrow \infty}{\rightarrow} \ell \geq 1/2$, choose any critical distribution $\nu$ with variance $2\ell-1$ and construct $w$ from $\nu$ the same way. This leads to
\begin{align*}
F_\mu(1-u) =1-u+u^2 \left(\ell-\frac{1}{2} \right) + \frac{u^2}{2} + o(u^2).
\end{align*}

In a second time, we shall prove that any critical reachable distribution has variance greater than $1$. Let $\mu$ be reachable, and take a weight sequence $w$ reaching $\mu$. Then, there exists $s>0$ such that, for all $x \in (-1,1)$, $F_\mu(x)=e^{F_w(sx)-F_w(s)}$. After differentiating once and applying at $x=1$, one gets 
\begin{equation}
\label{eq:expectation}
1=F_\mu'(1)=sF'_w(s).
\end{equation} 
By differentiating twice, one gets, for any $x \in (-1,1)$,
\begin{equation}
\label{eq:variance}
F_\mu''(x) = s^2 \left(F_w''(sx) + \left( F_w'(sx)\right)^2 \right) e^{F_w(sx)-F_w(s)}.
\end{equation}
Assume now that $\mu$ has finite variance $\sigma_\mu^2$. Since $\mu$ is critical, $\sigma^2_\mu=F_\mu''(1)$. Letting $x$ go to $1$ in \eqref{eq:variance}, we get by \eqref{eq:expectation} that $\sigma_\mu^2=s^2 F''_w(s)+1 \geq 1$.

The second part is just a consequence of Lemma \ref{prop:equiv} and the construction of suitable weight sequences in the beginning of this proof.
\end{proof}

Finally, we give a simple criterion for a distribution to be reachable, which may be of independent interest.

\begin{prop}
\label{prop:logcondition}
Let $\mu$ be a critical distribution on $\Z_+$. Then, the following statements are equivalent :
\begin{itemize}
\item[(i)] $\mu$ is reachable
\item[(ii)] All successive derivatives of $\log F_\mu$ at $0$ are nonnegative.
\end{itemize}
\end{prop}
\begin{proof}
By Lemma \ref{prop:equiv}, $\mu$ is reachable if and only if there exists a weight sequence $w$ and $s >0$ such that $F_\mu(x) = e^{F_w(sx)-F_w(s)}$ on $(-1,1)$, i.e. $F_w(sx)=F_w(s)+\log F_\mu(x)$ on this interval. All $w_i$'s are nonnegative, which proves that (i) $\Rightarrow$ (ii). Now assume (ii) and denote by $v_i$ the $i$th derivative of $\log F_\mu$ at $0$. Then, the weight sequence $(w_i)_{i \geq 0}$ defined by $w_i := v_i (i!)^{-1}$ for $i \geq 1$ and $w_0=0$ satisfies $F_\mu=e^{F_w-F_w(1)}$ on $(-1,1)$. Therefore, $w$ reaches $\mu$.
\end{proof}

As a consequence of Proposition \ref{prop:logcondition}, for any $\alpha \in (1,2)$, any slowly varying function $L$, there exists a probability distribution $\mu$ verifying \eqref{eq:stable}, such that all successive derivatives of $F_\mu$ at $0$ are nonnegative. For any $\ell \geq 1$, there exists $\mu$ with variance $\ell$, verifying \eqref{eq:vf}, such that all successive derivatives of $F_\mu$ at $0$ are nonnegative.

\subsection{Compared counting of the vertices in the tree $\cT_n$ and the white reduced tree $\cT_n^{\circ}$}
\label{ssec:relationtreewhitereducedtree}

Before we prove Theorem \ref{thm:cvbitypelam} in the next subsection, we gather results concerning the number of black vertices in different connected components of the tree $\cT_n$, comparing them to the number of white vertices in these connected components. It appears that these quantities are asymptotically proportional, the proportionality constant being the average number of black children of a white vertex. Let us state things properly:

\begin{lem}[Number of black vertices in a BTSG]
\label{lem:nbofcycles}
Let $w$ be a weight sequence of $\alpha$-stable type for some $\alpha \in (1,2]$, and $\nu$ be its critical equivalent. Then, as $n \rightarrow \infty$,
\begin{align*}
\frac{1}{n} N^\bullet\left(\cT_n \right) \overset{\P}{\rightarrow} 1-\nu_0.
\end{align*}
\end{lem}
As we will see in the next section, this straightforwardly implies Lemma \ref{lem:intronbcycles}.

\begin{proof}
The idea of the proof is to split the set of black vertices in the tree according to the number of white grandchildren of their parents. Let $N_k^{n,\circ}$ be the number of white vertices in $\cT_n$ that have exactly $k$ white grandchildren. Then, observe two things:
(i) for any fixed $K \in \Z_+$, jointly for $k \in \llbracket 1,K \rrbracket$, we have with high probability
\begin{equation}
\label{eq:nkncirc}
\left| N_k^{n,\circ} - n \mu_k \right| \leq n^{3/4};
\end{equation} 
(ii) conditionally to the fact that a white vertex has $k$ white grandchildren, its number of black children is independent of the rest of the tree, and is distributed as a variable $X_k$ verifying $X_0=0$ almost surely and $1 \leq X_k \leq k$ for all $k \geq 1$.

Indeed, (i) is a consequence of the joint asymptotic normality of the quantities $N_k^{n,\circ}$ (see e.g. \cite[Theorem $6.2$ (iii)]{The18}), while (ii) is clear by definition of the BTSG. Let us see how it implies our result. Fix $\epsilon>0$, and $K\geq 1$ such that $\sum_{k=1}^K k \mu_k \geq 1-\epsilon$. Such a $K$ exists by criticality of $\mu$. By (i) and (ii), a central limit theorem on the variables $X_k,k \leq K$ gives that, with high probability, jointly for any $0 \leq k \leq K$, 
\begin{equation}
\label{eq:nkbullet}
|N_k^{n,\bullet}-n\mu_k \E[X_k]|\leq n^{4/5},
\end{equation}
where $N_k^{n,\bullet}$ denotes the number of black vertices in the tree whose parent has $k$ white grandchildren.
On the other hand, a white vertex $u$ being given, its number of black children is necessarily less than its number of white grandchildren. Thus we get that the total number of black vertices in the tree whose parent has at least $K+1$ white grandchildren satisfies

\begin{align*}
\sum_{k \geq K+1} N_k^{n,\bullet} \leq \sum_{k \geq K+1} k N_k^{n,\circ} = (n-1) - \sum_{k =0}^K k N_k^{n,\circ},
\end{align*}
as $\sum_{k \in \Z_+} k N_k^{n,\circ}$ is the number of white grandchildren in the tree, which is equal to $n-1$ (only the root is not a grandchild of any white vertex). Therefore, applying \eqref{eq:nkncirc} to each $k \leq K$, we get that $\sum_{k \geq K+1} N_k^{n,\bullet} \leq \epsilon n + (K+1) n^{3/4}$ with high probability. Finally, using \eqref{eq:nkbullet}, as $n \rightarrow \infty$,
\begin{align*}
\P \left( \left|\frac{N^\bullet\left(\cT_n\right)}{n} - \sum_{k \in \Z_+} \mu_k \E[X_k] \right| \geq 2 \epsilon \right) \rightarrow 0.
\end{align*}

The only thing left to prove is that
\begin{equation}
\label{eq:Anucomputation}
\sum_{k \geq 0} \, \mu_k \, \E\left[ X_k \right] = 1-\nu_0.
\end{equation}

To this end, we see the tree $\cT_n$ as a bi-type Galton-Watson tree. We define two probability measures $\mub, \mun$ as follows:
\begin{align}
\label{eq:defmubmun}
&\forall i \geq 0, \mub_i = \mu^*_i \, e^{\nu_0} \, (1-\nu_0)^i \nonumber \\
&\mun_0=0 \text{ and } \, \forall i \geq 1, \mun_i = (1-\nu_0)^{-1} \nu_i,
\end{align}
One easily checks that these measures have total mass $1$. A quantity of particular interest is the mean of $\mub$:
\begin{equation}
\label{eq:meanmub}
\sum_{j \geq 1} j \, \mub_j = \sum_{j \geq 1} j \, \mu^*_j \, e^{\nu_0} \, (1-\nu_0)^j = e^{\nu_0-1} \, \sum_{j \geq 1} j \, \frac{(1-\nu_0)^j}{j!} = 1-\nu_0.
\end{equation}

Furthermore, by Lemma \ref{lem:equivalence}, for any $n \geq 1$, $\cT_n^{(\mub, \, \mun)} \overset{(d)}{=} \cT_n^{(\mu^*,w)}$. We can therefore write:
\begin{align*}
\sum_{k \geq 0} \, \mu_k \, \E\left[ X_k \right]  &= \sum_{k \geq 0} \, \mu_k \, \sum_{j \geq 1} j \, \P \left( k_\emptyset\left(\cT_n^{(\mub, \, \mun)}\right)=j \big| k_\emptyset\left(\cT_n^{\circ,(\mub, \, \mun)}\right)=k \right).
\end{align*}
Indeed, by definition, the variable $X_k$ is distributed as the number of black children of $\emptyset$ (or any other white vertex) conditionally to the fact that $\emptyset$ has $k$ white grandchildren.
Now remark that, since $\mub$ and $\mun$ are probability measures, one can define the bi-type Galton-Watson tree $\cT^{(\mub, \, \mun)}$ as in the monotype case, as the random variable on the set of finite bi-type trees satisfying, for any bi-type tree $T$:
\begin{align*}
\P \left( \cT^{(\mub, \, \mun)} = T \right) = \prod\limits_{x \in T, x \text{ white}} \mub_{k_x(T)} \times \prod\limits_{y \in T, y \text{ black}} \mun_{k_y(T)}.
\end{align*}
In particular, the BTSG $\cT_n^{(\mub, \, \mun)}$ is distributed as the tree $\cT^{(\mub, \, \mun)}$ conditioned to have $n$ white vertices. Now recall that $\mu$ is the critical distribution such that $\cT_n^\circ\overset{(d)}{=} \bT_n^\mu$ for all $n \geq 1$. Notably, for all $k$, $\mu_k=\P(k_\emptyset(\cT^{\circ,(\mub, \, \mun)})=k)$. Thus, using the fact that, conditionally to the number of white grandchildren of a white vertex $u$ of $\cT_n$, the number of black children of $u$ is independent of the rest of the tree, we can prove \eqref{eq:Anucomputation}. Here, for convenience, we write $\cT$ for $\cT^{(\mub, \, \mun)}$ and $\cT^\circ$ for $\cT^{\circ,(\mub, \, \mun)}$.

\begin{align*}
\sum_{k \geq 0} \, \mu_k \, \E\left[ X_k \right]  &= \sum_{k \geq 0} \, \mu_k \, \sum_{j \geq 1} j \, \P \left( k_\emptyset\left(\cT\right)=j \big| k_\emptyset\left(\cT^\circ\right)=k \right)\\
&= \sum_{j\geq 1} j \sum_{k \geq 0} \, \P \left( k_\emptyset\left(\cT\right)=j \big| k_\emptyset\left(\cT^\circ\right)=k \right) \P \left(k_\emptyset\left(\cT^\circ\right)=k\right)\\
&= \sum_{j \geq 1} j \, \P \left(k_\emptyset\left(\cT\right)=j\right) = \sum_{j \geq 1} j \, \mub_j,
\end{align*}
which implies \eqref{eq:Anucomputation} by \eqref{eq:meanmub}.
\end{proof}

We now generalize this statement, by investigating the number of black vertices in different components of a tree. This refinement allows us to precisely control the location of large faces in the black process of the tree, and thus to prove Theorem \ref{thm:cvbitypelam}. Specifically, a tree $T$ being given, each vertex $u$ of $T$ induces a partition of the set of vertices of $T$ into three parts: the set $G_1(u,T)$ of vertices that are visited for the first time by the contour function $C(T)$ before $u$, the subtree $G_2(u,T)$ rooted in $u$ and the set $G_3(u,T)$ of the vertices visited for the first time by $C(T)$ after $u$ has been visited for the last time.

\begin{lem}
\label{lem:proportions}
With high probability, jointly for $u \in \cT_n$ a white vertex, as $n \rightarrow \infty$, we have, jointly for $i=1,2,3$:
\begin{align*}
\left|G_i\left(u,\cT_n^{\circ}\right) \right| = \left(1+(1-\nu_0)\right)^{-1} \, \left|G_i(u,\cT_n)\right| + o(n),
\end{align*}
where we recall that we also denote by $u$ the vertex in $\cT_n^{\circ}$ corresponding to $u$. 
\end{lem}

In other terms, the proportions of vertices in $\cT_n$ in lexicographical order respectively before $u$, in the subtree rooted at $u$ and after $u$ are, with high probability, close to the proportions of vertices in $\cT_n^{\circ}$ in lexicographical order before $u$, in the subtree rooted at $u$ and after $u$. This boils down to proving that, in each of these components, the number of black vertices is roughly $(1-\nu_0)$ times the number of white vertices.

\begin{proof}
Fix $\epsilon>0$, and take $K \in \Z_+$ such that $\sum_{k=0}^K k \mu_k \geq 1-\epsilon$.
For $0 \leq k \leq K$, denote by $N^k_{2nt}(\cT_n^{\circ})$, for $0 \leq t \leq 1$, the number of different vertices in $\cT_n^{\circ}$ with $k$ children visited by the contour function $C(\cT_n^\circ)$ before time $2nt$. Then, it is known (see \cite[Theorem $1.1$ (ii)]{The18} for the finite variance case and \cite[Theorem $6.1$ (ii)]{The18} for the infinite variance case) that, uniformly in $k \leq K$:

\begin{equation}
\label{eq:ancienarticle}
\left( \frac{N^k_{2nt}\left( \cT_n^{\circ} \right) - n \mu_k t}{\sqrt{n}} \right)_{0 \leq t \leq 1} \underset{n \rightarrow \infty}{\overset{(d)}{\rightarrow}} \left(C_1 \mathbbm{e}_t + C_2 B_t \right)_{0 \leq t \leq 1},
\end{equation}
where $C_1, C_2$ are constants that only depend on $\mu$, $\be$ is a normalized Brownian excursion and $B$ is a Brownian motion independent of $\be$.

Now, for $u$ a white vertex of $\cT_n$, denote by $N^{k,(1)}(u)$ (resp. $N^{k,(2)}(u)$, $N^{k,(3)}(u)$) the number of different white vertices with $k$ white granchildren in $\cT_n$ visited by $C(\cT_n)$ for the first time before the first visit of $u$ (resp. between the first and last visits of $u$, and after the last visit of $u$). For $1 \leq i \leq 3$, set in addition $N^{(i)}(u) = |G_i(u,\cT_n)| \coloneqq \sum_{k \geq 0} N^{k,(i)}(u)$, the total number of vertices visited by the contour function resp. before the first visit of $u$, between the first and last visits of $u$ and after the last visit of $u$. 
We obtain from \eqref{eq:ancienarticle} that, as $n \rightarrow \infty$:
\begin{align*}
\P \left( \exists u \in \cT_n, u \text{ white}, \exists k \in \llbracket 0,K \rrbracket, \exists i \in \llbracket 1,3 \rrbracket, \left|N^{k,(i)}(u)- \mu_k N^{(i)}(u) \right| \geq n^{3/4} \right) \rightarrow 0.
\end{align*}

Now, on the complement of this event, using the notation $X_k$ of Lemma \ref{lem:nbofcycles}, for any white vertex $u \in \cT_n$ white, a central limit theorem provides:
\begin{equation}
\label{eq:probankbullet}
\P \left( \left| N^{\bullet,k,(i)}(u) - N^{k,(i)}(u) \E[X_k]\right| \geq n^{3/4} \right) = o(1/n),
\end{equation}
where $N^{\bullet,k,(i)}(u)$ denotes the number of black vertices in $G_i(u,\cT_n)$ whose parent has $k$ white grandchildren.
On the other hand, the total number of black vertices in the tree whose parent has more than $K$ white grandchildren is again at most $\epsilon n + (K+1) n^{3/4}$ with high probability by \eqref{eq:nkncirc}. By summing \eqref{eq:probankbullet} over all $k \leq K$, all $1 \leq i \leq 3$ and all white vertices $u \in \cT_n$, and finally by letting $\epsilon \rightarrow 0$, we obtain the result.
\end{proof}

\subsection{Proof of the technical theorem \ref{thm:cvbitypelam}}
\label{ssec:proof}

This whole subsection is devoted to the proof of Theorem \ref{thm:cvbitypelam}. First of all, we explain the structure of this proof: let $\alpha \in (1,2]$, $(w_i)_{i \geq 1}$ be a weight sequence of $\alpha$-stable type, and $\nu$ be its critical equivalent. When $\alpha<2$ or when $\nu$ has finite variance, we prove that the black and white processes coded by the BTSG $\cT_n$ are asymptotically close to each other at the scale $\tilde{B}_n$ (where $(\tilde{B}_n)_{n \geq 1}$ satisfies \eqref{eq:tbn} for $\nu$). Then, we investigate the whole colored lamination $\bL_\infty^\bullet(\cT_n)$, showing that it converges to a random stable lamination whose faces are colored independently with the same probability. The following theorem gathers these different results. Again in this section, as there is no ambiguity, $\cT_n$ stands for $\cT_n^{(\mu_*,w)}$.

\begin{thm}
\label{thm:convergenceblackwhite}
Let $\alpha \in (1,2]$, $w$ be a weight sequence of $\alpha$-stable type, $\nu$ be its critical equivalent, and $(\tilde{B}_n)_{n \geq 1}$ verifying \eqref{eq:tbn} for $\nu$. Then, if $\alpha \in (1,2)$ or if $\nu$ has finite variance:
\begin{itemize}
\item[(i)]
There exists a coupling between the black process and the white process of $\cT_n$ such that:
\begin{align*}
d_{Sk} \left( \left( \bL_{c (1-\nu_0) \tilde{B}_n}^\bullet(\cT_n) \right)_{c \geq 0}, \left( \bL_{c \tilde{B}_n}^\circ(\cT_n) \right)_{c \geq 0} \right) \underset{n \rightarrow \infty}{\overset{\P}{\rightarrow}} 0.
\end{align*}
where $d_{Sk}$ denotes the Skorokhod distance on $\D(\R_+,\bCL(\oD))$.

\item[(ii)]
The white process of $\cT_n$ converges in distribution towards the $\alpha$-stable lamination process:
\begin{align*}
\left( \bL_{c \tilde{B}_n}^\circ(\cT_n) \right)_{c \in [0,\infty]} \underset{n \rightarrow \infty}{\overset{(d)}{\rightarrow}} \left( \bL_c^{(\alpha)} \right)_{c \in [0,\infty]}.
\end{align*}

\item[(iii)]
In distribution, under the coupling of (i) and jointly with convergence (ii),
\begin{align*}
\bL^\bullet_\infty \left( \cT_n\right) \underset{n \rightarrow \infty}{\overset{(d)}{\rightarrow}} \bL_\infty^{(\alpha),p_\nu},
\end{align*}
where $$p_\nu \coloneqq \frac{\sigma_\nu^2}{\sigma_\nu^2+1}.$$
\end{itemize}
\end{thm}

Before jumping into the proof of Theorem \ref{thm:convergenceblackwhite}, let us explain why this theorem is enough to get Theorem \ref{thm:cvbitypelam}.

\begin{proof}[Proof of Theorem \ref{thm:cvbitypelam}]
The proof of Theorem \ref{thm:cvbitypelam} is now straightforward. Indeed, Theorem \ref{thm:convergenceblackwhite} (i) and (ii) imply the convergence of the first marginal in Theorem \ref{thm:cvbitypelam}, that is, the convergence of the black process of $\cT_n$ on any compact of $\R_+$.
The joint convergence of $\bL_\infty^\bullet(\cT_n)$ is finally a consequence of Theorem \ref{thm:convergenceblackwhite} (iii).
\end{proof}

Let us therefore prove Theorem \ref{thm:convergenceblackwhite}.

\subsubsection{Proof of Theorem \ref{thm:convergenceblackwhite} (i)}

We first explain the way of coupling the white and black processes coded by $\cT_n$. To each black vertex $u$, associate its white child $k(u)$ whose subtree in $\cT^\circ_n$ has the largest size (if the largest size is reached by more than one white child, then choose one uniformly at random). Now, start from a uniform labelling of the white vertices. We label the black vertices the following way: give the label $1$ to the black vertex $u_1$ such that $k(u_1)$ has the smallest label among all white vertices of the form $k(u)$; give the label $2$ to $u_2$ such that $k(u_2)$ has the second smallest label, etc. This provides a way of labelling the black vertices of $\cT_n$ from $1$ to $N^\bullet(\cT_n)$, and this labelling is clearly uniform.  See Fig. \ref{fig:coupling} for an example of this coupling. This induces therefore a coupling between the black and white processes $(\bL_u^\circ(\cT_n))_{u \in [0,\infty]}$ and $(\bL_u^\bullet(\cT_n))_{u \in [0,\infty]}$.

\begin{figure}[!h]
\caption{The coupling between labels of black and white vertices in a tree: arrows go from a black vertex $u$ to the white vertex $k(u)$. Left: the coupling between vertices. Middle: a uniform labelling of the white vertices. Right: the induced labelling of the black vertices}
\label{fig:coupling}
\center
\begin{tabular}{c c c}
\begin{tikzpicture}
\draw (-2.5,4) -- (-2,3) -- (-2,2) -- (-1,1) -- (-1,2) (-1,1) -- (0,0) -- (0,1) (0,0) -- (1,1) -- (2,2) (1,1) -- (1,2) -- (1,3) -- (1,4) (-1.5,4) -- (-2,3) (0,1) -- (0,2) (1,2) -- (2,3) -- (2,4) (-1,2) -- (-1,3) -- (-.5,4);

\draw[fill=white] (0,0) circle (.2);
\draw[fill=white] (-2,2) circle (.2);
\draw[fill=white] (-1,2) circle (.2);
\draw[fill=white] (0,2) circle (.2);
\draw[fill=white] (1,2) circle (.2);
\draw[fill=white] (2,2) circle (.2);
\draw[fill=white] (1,4) circle (.2);
\draw[fill=white] (-2.5,4) circle (.2);
\draw[fill=white] (-1.5,4) circle (.2);
\draw[fill=white] (-.5,4) circle (.2);
\draw[fill=white] (2,4) circle (.2);

\draw[fill=black] (-2,3) circle (.2);
\draw[fill=black] (-1,1) circle (.2);
\draw[fill=black] (-1,3) circle (.2);
\draw[fill=black] (1,1) circle (.2);
\draw[fill=black] (1,3) circle (.2);
\draw[fill=black] (0,1) circle (.2);
\draw[fill=black] (2,3) circle (.2);

\draw[dashed,->] (-1,1) to[out=200,in=-90] (-2,1.8);
\draw[dashed,->] (-2,3.2) to[out=-90,in=180] (-1.7,4);
\draw[dashed,->] (-1,3.2) to[out=-90,in=180] (-.7,4);
\draw[dashed,->] (0,1) to[out=135,in=-135] (-.16,1.92);
\draw[dashed,->] (1,1) to[out=135,in=-135] (.84,1.92);
\draw[dashed,->] (1,3) to[out=135,in=-135] (.84,3.92);
\draw[dashed,->] (2,3) to[out=135,in=-135] (1.84,3.92);
\end{tikzpicture}
&
\begin{tikzpicture}[every node/.style={scale=0.8}]
\draw (-2.5,4) -- (-2,3) -- (-2,2) -- (-1,1) -- (-1,2) (-1,1) -- (0,0) -- (0,1) (0,0) -- (1,1) -- (2,2) (1,1) -- (1,2) -- (1,3) -- (1,4) (-1.5,4) -- (-2,3) (0,1) -- (0,2) (1,2) -- (2,3) -- (2,4) (-1,2) -- (-1,3) -- (-.5,4);

\draw[fill=white] (0,0) circle (.2) node{$3$};
\draw[fill=white] (-2,2) circle (.2) node{$8$};
\draw[fill=white] (-1,2) circle (.2) node{$4$};
\draw[fill=white] (0,2) circle (.2) node{$6$};
\draw[fill=white] (1,2) circle (.2) node{$2$};
\draw[fill=white] (2,2) circle (.2) node{$11$};
\draw[fill=white] (1,4) circle (.2) node{$1$};
\draw[fill=white] (-2.5,4) circle (.2) node{$9$};
\draw[fill=white] (-1.5,4) circle (.2) node{$5$};
\draw[fill=white] (-.5,4) circle (.2) node{$10$};
\draw[fill=white] (2,4) circle (.2) node{$7$};

\draw[dashed,->] (-1,1) to[out=200,in=-90] (-2,1.8);
\draw[dashed,->] (-2,3.2) to[out=-90,in=180] (-1.7,4);
\draw[dashed,->] (-1,3.2) to[out=-90,in=180] (-.7,4);
\draw[dashed,->] (0,1) to[out=135,in=-135] (-.16,1.92);
\draw[dashed,->] (1,1) to[out=135,in=-135] (.84,1.92);
\draw[dashed,->] (1,3) to[out=135,in=-135] (.84,3.92);
\draw[dashed,->] (2,3) to[out=135,in=-135] (1.84,3.92);

\draw[fill=black] (-2,3) circle (.2);
\draw[fill=black] (-1,1) circle (.2);
\draw[fill=black] (-1,3) circle (.2);
\draw[fill=black] (1,1) circle (.2);
\draw[fill=black] (1,3) circle (.2);
\draw[fill=black] (0,1) circle (.2);
\draw[fill=black] (2,3) circle (.2);
\end{tikzpicture}
&
\begin{tikzpicture}
\draw (-2.5,4) -- (-2,3) -- (-2,2) -- (-1,1) -- (-1,2) (-1,1) -- (0,0) -- (0,1) (0,0) -- (1,1) -- (2,2) (1,1) -- (1,2) -- (1,3) -- (1,4) (-1.5,4) -- (-2,3) (0,1) -- (0,2) (1,2) -- (2,3) -- (2,4) (-1,2) -- (-1,3) -- (-.5,4);

\draw[fill=white] (0,0) circle (.2);
\draw[fill=white] (-2,2) circle (.2);
\draw[fill=white] (-1,2) circle (.2);
\draw[fill=white] (0,2) circle (.2);
\draw[fill=white] (1,2) circle (.2);
\draw[fill=white] (2,2) circle (.2);
\draw[fill=white] (1,4) circle (.2);
\draw[fill=white] (-2.5,4) circle (.2);
\draw[fill=white] (-1.5,4) circle (.2);
\draw[fill=white] (-.5,4) circle (.2);
\draw[fill=white] (2,4) circle (.2);

\draw[fill=black] (-2,3) circle (.2) node[white]{$3$};
\draw[fill=black] (-1,1) circle (.2) node[white]{$6$};
\draw[fill=black] (-1,3) circle (.2) node[white]{$7$};
\draw[fill=black] (1,1) circle (.2) node[white]{$2$};
\draw[fill=black] (1,3) circle (.2) node[white]{$1$};
\draw[fill=black] (0,1) circle (.2) node[white]{$4$};
\draw[fill=black] (2,3) circle (.2) node[white]{$5$};
\end{tikzpicture}
\end{tabular}
\end{figure}
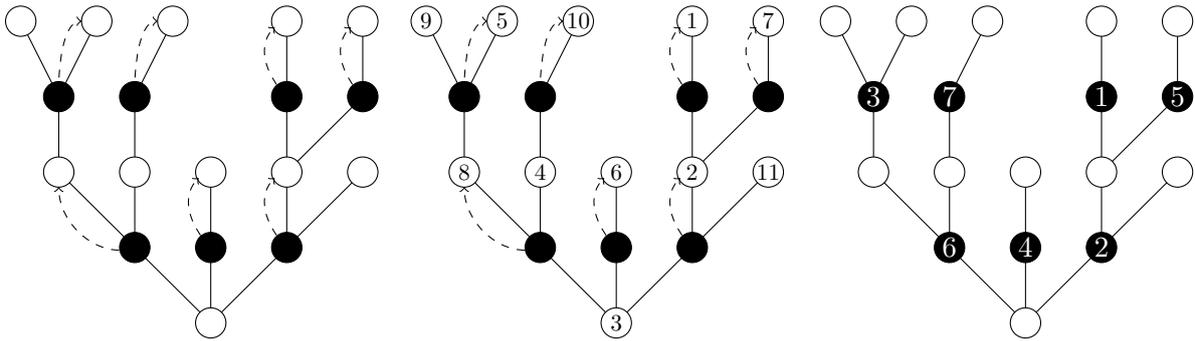

We claim that, under this coupling, Theorem \ref{thm:convergenceblackwhite} (i) holds. To this end, we prove that the following two events hold with high probability: 
\begin{itemize}
\item[(a)] first, uniformly for $u$ a black vertex in $\cT_n$ with label $\leq \tilde{B}_n \log n$, the distance between $\bS^1 \cup F_u(\cT_n)$ (in $\bL^\bullet(\cT_n)$) and $\bS^1 \cup c_{k(u)}(\cT_n^\circ)$ (in $\bL^\circ(\cT_n)$) goes to $0$; \item[(b)] uniformly for each black vertex $u$ with label $e(u) \leq \tilde{B}_n \log n$, 
\begin{equation}
\label{eq:obtilden}
\left| (1-\nu_0) \, e^\circ(k(u))-e(u) \right| = o(\tilde{B}_n),
\end{equation}
where $e^\circ(x)$ is the label of the white vertex $x$.
\end{itemize} 
Roughly speaking, (a) proves that faces of the black process are close (one by one) to some chords of the white process, and (b) that each face roughly appears at the same time as the associated chord, in the time-rescaled processes.

Under these two events, the Skorokhod distance between the black and the white processes up to time $\tilde{B}_n \log n$, rescaled in time by a factor $\tilde{B}_n$, goes to $0$ as $n \rightarrow \infty$. Indeed, by \eqref{eq:obtilden}, if one rescales by this factor $\tilde{B}_n$, asymptotically the face $F_u(\cT_n)$ and the chord $c_{k(u)}(\cT_n^\circ)$ appear at the same time up to $o(1)$, uniformly for $u$ a black vertex with label $\leq \tilde{B}_n \log n$. The only thing left to prove is that no other large white chord appears in the white process before time $\tilde{B_n} \log n$. To see this, remark that, at $\epsilon>0$ fixed, if a chord $c_v(\cT^\circ_n)$ has length larger than $\epsilon$, where $v$ is a white vertex that is not of the form $k(u)$ for some black vertex $u$, then necessarily the parent of $v$ in $\cT_n^\circ$ is an $\epsilon n$-branching point. The number of white vertices $v$ such that $|\theta_v(\cT_n^\circ)| \geq \epsilon$ and such that the parent of $v$ in $\cT_n^\circ$ is an $\epsilon n$-branching point is bounded by $\epsilon^{-1}$, independently of $n$. Hence, with high probability none of them has a label less than $\tilde{B}_n \log n$, and all large white chords in the white process that appear before time $\tilde{B_n} \log n$ are of the form $c_{k(u)}(\cT^\circ_n)$ for some black vertex $u$. This implies Theorem \ref{thm:convergenceblackwhite} (i).

W now prove (a) and (b). In what follows, we call \textit{marked vertices} the white vertices of the form $k(u)$ for some black vertex $u \in \cT_n$.

In order to prove (a), we mostly rely on Lemma \ref{lem:proportions}. Fix $\epsilon>0$ and take $u$ a black vertex in $\cT_n$ with label $\leq \tilde{B}_n \log n$. Then, with high probability, $u$ is not a black $\epsilon n$-node of $\cT_n$. Indeed, there are at most $2n$ vertices in total in $\cT_n$, and thus at most $2\epsilon^{-1}$ $\epsilon n$-nodes in this tree. Assume that it is not an $\epsilon n$-node. Then, if all chords of the boundary of $F_u$ have lengths $<\epsilon$, with high probability $c_{k(u)}$ has length less than $2\epsilon$ by Lemma \ref{lem:proportions}. Now assume that one of the chords in the boundary of $F_u$, which we denote by $c_*$, has length greater than $\epsilon$. As $u$ is not an $\epsilon n$-node of $\cT_n$, there are at most two such chords in the boundary of $F_u$ and therefore $d_H(c_*,F_u)<2 \pi \epsilon$. In addition, again by Lemma \ref{lem:proportions}, with high probabliity $d_H(c_*, c_{k(u)}) < 2 \pi \epsilon$. Furthermore, this holds jointly for all $u$ with label $\leq \tilde{B}_n \log n$.

In order to prove (b), the idea is to code the location of marked vertices (corresponding to the children of each black vertex having the largest subtree, which are fixed and do not depend on the labelling on the white vertices; they are white vertices that are targets of an arrow on Fig. \ref{fig:coupling} left and middle) in lexicographical order by a walk on $\R$, and then use well-known results about the behaviour of random walks.

First, remark that, by Lemma \ref{lem:nbofcycles}, with high probability there are $N^\bullet(\cT_n) \coloneqq (1-\nu_0) n (1+o(1))$ black vertices in the tree $\cT_n$. Therefore, among the $n$ white vertices in the tree, $(1-\nu_0) n (1+o(1))$ of them are marked, and the fact that a vertex is marked does not depend on the labelling. Moreover, the labels of these white vertices are uniformly chosen among all $N^\bullet(\cT_n)$-tuples of distinct integers between $1$ and $n$. 

Thus, the problem boils down to the following: there are $n$ white vertices, among which $(1-\nu_0) n (1+o(1))$ are marked. We want to prove that, with high probability, uniformly in $c \leq \log n$, among the first $c \tilde{B}_n$ white vertices (for the order of the labels), there are $c (1-\nu_0) \tilde{B}_n (1+o(1))$ marked ones.

To prove it, denote by $q_x$ the number of marked vertices among the first $x$ ones. It is clear that, uniformly for $k \leq \tilde{B}_n \log n$, uniformly for $N \geq (1-\nu_0) n/2$, conditionally to $N^\bullet(\cT_n) = N$: 
\begin{equation}
\label{eq:binasymp}
\P\left( q_{\tilde{B}_n \log n} = k \right) = \frac{\P \left( B_1=k \right) \P\left(B_2=N-k\right)}{\P \left(B_3=N\right)} \sim \P\left( B_1=k \right)
\end{equation}
as $n \rightarrow \infty$, where $B_1= Bin(\lfloor \tilde{B}_n \log n \rfloor, 1-\nu_0), B_2= Bin(n-\lfloor \tilde{B}_n \log n \rfloor, 1-\nu_0), B_3= Bin(n, 1-\nu_0)$. Remark that $N^\bullet(\cT_n) \geq (1-\nu_0) n/2$ with high probability, so that \eqref{eq:binasymp} holds with high probability. Furthermore, conditionally to the value $k$ of $q_{\tilde{B}_n \log n}$, the set of marked vertices is uniformly distributed among all possible subsets of $k$ of these $\tilde{B}_n \log n$ white vertices.

Finally, notice that the quantity $(1-\nu_0) \, e^\circ(k(u))-e(u)$, for $u$ the white vertex labelled $i$, can be seen as the value at time $i$ of a specific random walk, constructed from the labelling of the vertices in $\cT_n$. More precisely, denote by $(S_i)_{0 \leq i \leq \tilde{B}_n \log n}$ the walk defined as follows: it starts from the value $S_0=0$ and, for $1 \leq i \leq \tilde{B}_n \log n$, $S_i-S_{i-1}=-1$ if the white vertex labelled $i$ is of the form $k(u)$ for some black vertex $u$ (that is, the vertex is marked), and $S_i-S_{i-1} = (1-\nu_0)/\nu_0$ otherwise. Then, one can check that conditionally to the value $k$ of $q_{\tilde{B}_n \log n}$, this walk is distributed as a random walk $(S'_i, \, 0 \leq i \leq \tilde{B}_n \log n)$ starting from $0$ with i.i.d. jumps, the jumps being $-1$ with probability $1-\nu_0$ and $(1-\nu_0)/\nu_0$ with probability $\nu_0$, conditioned to have $k$ "$-1$" jumps. In particular, the expectation of each jump of $S'$ is $0$. By the so-called local limit theorem (see \cite[Theorem $4.2.1$]{IL71} for a statement and proof), the maximum of the absolute value of this walk is of order $\sqrt{\tilde{B}_n \log n} = o(\tilde{B}_n)$. Using \eqref{eq:binasymp}, the maximum of the absolute value of $(S_i)_{0 \leq i \leq \tilde{B}_n \log n}$ is also of order $\sqrt{\tilde{B}_n \log n}$ with high probability. Finally, remark that, for any white vertex $u$ labelled $i \leq \tilde{B}_n \log n$, the value $S_i$ of the walk at time $i$ is exactly $(1-\nu_0) \, e^\circ(k(u)) - e(u)$ by construction. This proves the result.

\bigskip

\subsubsection{Proof of Theorem \ref{thm:convergenceblackwhite} (ii)}

To prove this, we use the fact that the white reduced tree $\cT_n^\circ$ is a $\mu$-GW tree conditioned to have $n$ vertices, where - by Theorem \ref{thm:reachabledistributions} - $\mu$ is a critical probability distribution in the domain of attraction of an $\alpha$-stable law. Hence, Theorem \ref{thm:convergenceblackwhite} (ii) directly follows from \cite[Theorem $3.3$ and Proposition $4.3$]{The19}, and is used under this form in \cite{The19} to study the model of minimal factorizations of the $n$-cycle into transpositions.

\bigskip

We now prove the third part of Theorem \ref{thm:convergenceblackwhite}. We separately treat the two cases when $\alpha<2$ and when $\nu$ has finite variance. 

\subsubsection{Proof of Theorem \ref{thm:convergenceblackwhite} (iii), when $\alpha<2$}

In the whole paragraph, $(\tilde{B}_n)_{n \geq 1}$ is a sequence that satisfies \eqref{eq:tbn} for $\nu$. In particular, as $n \rightarrow \infty$,
\begin{equation}
\label{eq:bnslowlyvar}
\tilde{B}_n \sim n^{1/\alpha} \ell(n)
\end{equation}
for some slowly varying function $\ell$. 

We prove here that, jointly with the convergence of Theorem \ref{thm:convergenceblackwhite} (ii), the sequence $(\bL^\bullet_\infty(\cT_n))_{n \geq 1}$ converges towards the colored stable lamination $\bL_\infty^{(\alpha),1}$, whose red part is $\bL_\infty^{(\alpha)}$ (which denotes here the limit of the process $(\bL^\circ_\infty(\cT_n))_{n \geq 1}$ by Theorem \ref{thm:convergenceblackwhite} (ii)), and whose faces are all colored black. In order to see it, we prove that with high probability in the tree $\cT_n$, for any white $\epsilon n$-node $u$ of $\cT_n$, almost all grandchildren of $u$ have the same black parent. To this end, we rely on the following lemma, inspired from \cite[Section $5$, Lemma $5$]{LGM10}:

\begin{lem}
\label{lem:pnu<2}
There exists a small $\delta>0$ such that, for any $\eta>0$, with high probability, for any white vertex $u \in \cT_n$ having at least $\eta \tilde{B}_n$ white grandchildren, all of them but at most $\tilde{B}_n n^{-\delta}$ have the same black parent.
\end{lem}

Let us immediately see how this implies the convergence of Theorem \ref{thm:convergenceblackwhite} (iii) in this case. The key remark, which is straightforward by construction, is that all faces with a 'large' area in the colored lamination are coded by large nodes in the tree $\cT_n$ (either black or white). More precisely, for any $r>0$, there exists $\epsilon>0$ such that all faces of area larger than $r$ in $\bL_\infty^\bullet(\cT_n)$ are coded by $\epsilon n$-nodes of $\cT_n$. In addition, if a black vertex is a $\rho n$-node of $\cT_n$, then, by Lemma \ref{lem:proportions}, with high probability its white parent is an $(1-\nu_0) \rho n/2$-node of the reduced tree $\cT_n^\circ$. This allows us to focus only on white $\epsilon n$-nodes of $\cT_n^\circ$. 

\begin{proof}[Proof of Theorem \ref{thm:convergenceblackwhite} (iii)]
We use the fact that with high probability all large white nodes in original tree have a large number of white grandchildren.
Let us fix $\epsilon>0$, and take $\eta>0$ such that, with probability $\geq 1-\epsilon$, all white $\epsilon n$-nodes in $\cT^\circ_n$ have at least $\eta \tilde{B}_n$ white grandchildren in $\cT_n$ (such an $\eta$ exists by Lemma \ref{lem:grossommets} (ii)). Denote by $K_{\epsilon}(\cT^\circ_n)$ the (random) number of $\epsilon n$ nodes in $\cT_n^\circ$. Remark that there are at most $\epsilon^{-1}$ of them, and denote them by $a_1, \ldots, a_{K_{\epsilon}(\cT_n)}$ in lexicographical order.

Let us focus on $a_1$. Take $\delta>0$ such that, by Lemma \ref{lem:pnu<2}, with high probability all white grandchildren of $a_1$ except at most $\tilde{B}_n n^{-\delta}$ have the same black parent, which we denote by $b_1$. Set now $S_\epsilon(a_1) \coloneqq  \{ u \text{ granchild of } a_1, |\theta_u(\cT_n)|\geq \epsilon n \}$, the subset of grandchildren of $a_1$ whose subtree in $\cT_n$ has size more than $\epsilon n$. Then $|S_\epsilon(a_1)| \leq \lfloor 2\epsilon^{-1} \rfloor$, and with high probability all elements of $S_\epsilon(a_1)$ are children of $b_1$. Now define from these points the face $\tilde{F}_{a_1}(\cT_n)$, as
\begin{align*}
\tilde{F}_{a_1}(\cT_n) = \bS^1 \cup c_{a_1}(\cT_n) \cup \bigcup_{u \in S_\epsilon(a_1)} c_u (\cT_n),
\end{align*}
whose connected component having $c_{a_1}$ in its boundary and not containing $1$ is colored black. In other terms, this face does only take into account the subtrees of size larger than $\epsilon n$ rooted in grandchildren of $a_1$.

Then, using Lemma \ref{lem:proportions} jointly for each point of $S_\epsilon(a_1)$, it is clear that, with high probability,
\begin{align*}
d_H \left( F_{b_1}(\cT_n), \tilde{F}_{a_1}(\cT_n) \right) \leq 2 \pi \epsilon.
\end{align*}
On the other hand, by construction,
\begin{align*}
d_H \left( \tilde{F}_{a_1}(\cT_n), F'_{a_1}(\cT^\circ_n) \right) \leq 2 \pi \epsilon,
\end{align*}
where $F'_{a_1}(\cT^\circ_n)$ is the colored lamination defined as $$F'_{a_1}(\cT_n^\circ) \coloneqq \bS^1 \cup c_{a_1}(\cT_n^\circ) \cup \bigcup_{u \text{ granchild of } a_1} c_u(\cT_n^\circ)$$ 
in which the face of $\bL^\circ_\infty(\cT_n)$ whose boundary contains $c_{a_1}$ and all chords $c_u$ for $u$ a grandchild of $a_1$ is colored black. In other words, the large face of $\bL_\infty^\bullet(\cT_n)$ coded by $b_1$ is close to the large face of $\bL_\infty^\circ(\cT_n)$ bounded by the chords coded by $a_1$ and its grandchildren, and colored black. In addition, the same holds for $a_2, \ldots, a_{K_\epsilon(\cT_n)}$. Since $\bL_\infty^\circ(\cT_n)$ converges in distribution towards the $\alpha$-stable lamination $\bL_\infty^{(\alpha)}$, $\bL_\infty^\bullet(\cT_n)$ converges in distribution towards $\bL_\infty^{(\alpha),1}$.
\end{proof}

We now prove Lemma \ref{lem:pnu<2}.

\begin{proof}[Proof of Lemma \ref{lem:pnu<2}]
The proof is inspired from \cite[Section $5$, Lemma $5$]{LGM10}. Fix $\delta>0$ such that $2 \delta (\alpha+1/\alpha)<1$. Take $\eta>0$, and $n$ large enough so that $\eta \tilde{B}_n > 2 \tilde{B}_n n^{-\delta}$. For $u$ a white vertex of $\cT_n$, for any $k,M \geq 1$, define the following event $E(u,k,M)$: $u$ has $k$ black children, a number $M \geq \eta \tilde{B}_n$ of white grandchildren and simultaneously none of its black children has more than $M - \tilde{B}_n n^{-\delta}$ white children. This implies that at least two among its black children have more than $\tilde{B}_n n^{-\delta}/k$ white children. 

Therefore, for any white vertex $u$, uniformly in $M \geq \eta \tilde{B}_n$ and $k \geq 2$, one gets:
\begin{align*}
\P\left(E(u,k,M) \, \big| \, k_u(\cT_n^\circ)=M \right) \leq \mu_*(k) \binom{k}{2} \nu \left([\tilde{B}_n n^{-\delta}/k, \infty) \right)^2.
\end{align*}
On the other hand, by usual properties of the domain of attraction of stable laws (see e.g. \cite{Fel08}, Corollary $XVII.5.2$), there exists a constant $K>0$ such that, for all $R>0$, $\nu ([R, \infty)) \leq K R^{-\alpha+\delta}$. Hence, the probability that there exists a white vertex $u$ in $\cT_n$ with more than $\eta \tilde{B}_n$ white grandchildren and such that $E(u,k,M)$ holds for some $k \geq 2$, $M \geq \eta \tilde{B}_n$ is less than
\begin{align*}
n \sum_{k=2}^\infty \mu_*(k) \binom{k}{2} \nu \left([\tilde{B}_n n^{-\delta}/k, \infty) \right)^2 \leq n \left(\sum_{k=2}^\infty \mu_*(k) \binom{k}{2} k^{2\alpha-2\delta} \right) \tilde{B}_n^{-2\alpha+2\delta} n^{2\alpha \delta} = O \left( n^{1+2\alpha \delta} \tilde{B}_n^{2\delta-2\alpha} \right).
\end{align*}
Using \eqref{eq:bnslowlyvar} and the definition of $\delta$, $n^{1+2\alpha \delta} \tilde{B}_n^{2\delta-2\alpha} \leq n^{2\delta(\alpha+1/\alpha)-1} \ell(n)^{2\delta-2\alpha}$ for some slowly varying function $\ell$. It is finally well-known that, for any $\epsilon>0$, for $n$ large enough, $\ell(n) \in (n^{-\epsilon}, n^{\epsilon})$, by the so-called Potter bounds (see e.g. \cite[Theorem $1.5.6$]{BGT89} for a precise statement and a proof).
Thus, $n^{1+2\alpha \delta} \tilde{B}_n^{2\delta-2\alpha} = o(1)$ as $n \rightarrow \infty$, which proves our result.
\end{proof}

\subsubsection{Proof of Theorem \ref{thm:convergenceblackwhite} (iii), when $\nu$ has finite variance}

The case with finite variance is different. Indeed, in this case, it may happen that $0 < p_\nu < 1$, and the coloration of the limiting Brownian triangulation is not trivial.
We prove that, still, each face of the limiting object is colored black independently with the same probability $p_\nu$.

Let us first recall some notation. In what follows, for $\mu$ a critical distribution, $\bT^\mu$ denotes a $\mu$-GW tree, and, for any $i \geq 1$, $\bT_i^\mu$ denotes a $\mu$-GW tree conditioned to have exactly $i$ vertices. $\emptyset$ always denotes the root of the tree, and $K_u(T)$ denotes the set of children of $u$ in $T$.

Fix $\epsilon>0$. When $\mu$ has finite variance, for $n$ large, $\epsilon n$-nodes in $\bT^\mu_n$ are in fact $\epsilon n/2$-branching points, which we recall are vertices such that two of their children are the root of a subtree of size $\geq \epsilon n/2$:
\begin{lem}
\label{lem:bp}
With high probability as $n \rightarrow \infty$, jointly for all $\epsilon n$-nodes $u$ of $\bT_n^\mu$, there exist $v_1(u), v_2(u)$ two children of $u$ such that 
\begin{align*}
|\theta_{v_1(u)}(\bT_n^\mu)|\geq \epsilon n/2, \qquad |\theta_{v_2(u)}(\bT_n^\mu)|\geq \epsilon n/2, \\
\text{ and } \, \sum_{w \in K_u(\bT_n^\mu), w \neq v_1(u),v_2(u)}|\theta_{w}(\bT_n)| = o(n).
\end{align*}
\end{lem}

In other therms, if the tree splits at the level of $u$ into at least two macroscopic components, then with high probability it splits into exactly two of them. This is a well-known fact, direct consequence of the convergence of Theorem \ref{thm:cvcontour} and the fact that the local minima of the normalized Brownian excursion are almost surely distinct. Thus, exactly two children of each $\epsilon n$-node are the root of a 'large' subtree, while the sum of the sizes of all other subtrees rooted in a child of this node is $o(n)$. Therefore, investigating $\epsilon n$-nodes boils down to investigating $\epsilon n$-branching points.

In order to prove that faces are asymptotically colored in an i.i.d. way, remark that, a white $\epsilon n$-branching point of $\cT_n^\circ$ being given, there are two possible cases: either its two white grandchildren with a large subtree $v_1(u), v_2(u)$ have the same black parent (see Fig. \ref{fig:twocases}, top-left) which provides a large black face in the lamination; or they have two different black parents (see Fig. \ref{fig:twocases}, top-right) which provides a large white face. Finally, remark that the event that $v_1(u), v_2(u)$ have the same black parent, conditionally to the number of white grandchildren of $u$, is independent of the rest of the tree.

The proof therefore has two different steps. We first prove that the distribution of the colors of the faces asymptotically does not depend on the shape of the tree (this means that it is asymptotically independent of the colored lamination-valued process $(\bL_{c \tilde{B}_n}^\bullet(\cT_n))_{0 \leq c \leq M}$ stopped at any finite time $M$). This step is done by shuffling branching points in the tree, in such a way that the shape of the tree is not changed much. In a second time, we prove that the distribution of the colors of the largest faces in the final lamination indeed converges towards i.i.d. random variables, and compute the asymptotic probability that a large face is colored black.

Let us first define a transformation on bi-type trees, which allows to introduce additional randomness in the degree distribution of white branching points without changing the overall shape of this tree. The image $\tilde{\cT}_n$ of the random tree $\cT_n$ by this transformation shall be distributed as $\cT_n$, and their black processes shall in addition be close with high probability. Furthermore, $\bL_\infty^\bullet(\tilde{\cT}_n)$ shall be close to $\bL_\infty^{(\alpha),p}$ for some $p \in [0,1]$, which proves Theorem \ref{thm:convergenceblackwhite} (iii).

The idea of the transformation is to randomize a small part of the tree $\cT_n$, so that the whole black process $(\bL_c^\bullet(\cT_n))_{c \geq 0}$ does not change much. To this end, we associate to each 'large' face of $\bL_\infty^\bullet(\cT_n)$ a white branching point of $\cT_n^\circ$: the vertex coded by this face if the face is white, and the parent of this vertex if it is black. Then, $\epsilon>0$ being given, one shuffles some well-chosen branching points in the tree, so that white $\epsilon n$-branching points of $\cT_n^\circ$ are still $\epsilon n$-branching points after this shuffling, but the coloration of the face that they code is randomized. Indeed, although we are able to compute the limiting joint distribution of the degrees of the branching points in a conditioned GW-tree, it is not clear at first sight that this distribution is asymptotically independent of the shape of the tree. This transformation allows us to prove it, by shuffling a large number of $\rho n$-branching points (for $0 < \rho < \epsilon$) with the $\epsilon n$-branching points of the initial tree. 

Let us state it properly. For $\epsilon>\eta>0$ two constants, we define the set $\kBT_n^{\epsilon,\eta}$ as the set of bi-type trees $T_n$ with $n$ white vertices, such that there exists a white vertex $u \in T_n$ satisfying $|\theta_u(T_n^\circ)| \in (\eta n, \epsilon n)$. For any tree $T_n \in \kBT_n^{\epsilon,\eta}$, we define a shuffling operation.

\begin{defi}[The shuffling operation]
\label{def:epsilonshuffling}
Fix three constants $\epsilon>\eta>\rho>0$ and take $T_n \in \kBT_n^{\epsilon,\eta}$. We construct the shuffled tree $T_n^{\epsilon,\eta,\rho}$ as follows: take $u$ a white vertex of $T_n$ such that $|\theta_u(T^\circ_n)| \in (\eta n,\epsilon n)$. Let $E \coloneqq E_{\epsilon n}(T^\circ_n) \cup E_{\rho n}(\theta_u(T^\circ_n))$, the set made of all white $\epsilon n$-branching points of $T_n^\circ$ and all white $\rho n$-branching points of the white subtree rooted in $u$ (remark that there is no $\epsilon n$-branching point in this subtree, by definition of $u$). Since $\rho<\epsilon$, $\epsilon n$-branching points are also $\rho n$-branching points and thus $|E| \leq \rho^{-1}$ (notice that $|E|$ is random anyway). Let $U_1,\ldots, U_{|E|}$ be the elements of $E$, sorted in lexicographical order. For each $i \leq |E|$, denote by $v_1(U_i), v_2(U_i)$, in lexicographical order, the two grandchildren of $U_i$ whose subtrees are the largest (in case of equality, arbitrarily pick two that are larger than all others). 
Define the tree $T_n^{\epsilon,\eta,\rho}$ from $T_n$ as follows: denote by $S(U_i)$ the part of the subtree $\theta_{U_i}(T_n) \backslash U_i$, where one also "cuts" the edges between $v_1(U_i)$, $v_2(U_i)$ and its black parent(s). See Fig. \ref{fig:twocases} for an example. We now take $\sigma$, a permutation of $\llbracket 1, |E| \rrbracket$ chosen uniformly at random, and exchange the $S(U_i)$'s according to $\sigma$, reattaching the half-edges which lead to $v_1(U_i), v_2(U_i)$ to $S_{\sigma(i)}$. In addition, each black vertex keeps its original label. See Fig. \ref{fig:twocases} for an example of this shuffling of $S_i$'s.
\end{defi}

We claim that, for any $\epsilon>\eta>\rho>0$, any $c \leq \tilde{B}_n \log n$, with high probability the Hausdorff distance between $\bL_c^\bullet(T_n)$ and $\bL^\bullet_c(T^{\epsilon,\eta,\rho}_n)$ is bounded from above by the following quantity:
\begin{align*}
C_\rho(T_n) \coloneqq \frac{4\pi}{n} \sum_{u \in E_{\rho n}(T^\circ_n)} \sum_{v \in K^{(-2)}_u(T^\circ_n)} |\theta_v(T^\circ_n)|,
\end{align*}
where, for any $u \in T_n$, $K^{(-2)}_u(T^\circ_n)$ denotes the union of the set of children $v$ of $u$ in $T_n^\circ$ whose subtree $\theta_v(T_n^\circ)$ has size less than $\rho n$.

\begin{lem}
\label{lem:Crho}
Let $\epsilon > \eta > \rho>0$, and take a tree $T_n \in \kBT_n^{\epsilon,\eta}$. Then, with high probability, uniformly for $0 \leq c \leq \log n$:
\begin{align*}
d_H\left( \bL_{c\tilde{B}_n}^\bullet(T_n), \bL_{c\tilde{B}_n}^\bullet(T^{\epsilon, \eta, \rho}_n) \right) \leq C_\rho(T_n).
\end{align*}
\end{lem}
Notice that this is not true for all $c$, and in particular not for $c=\infty$, as colors of large faces may be changed by the transformation of Definition \ref{def:epsilonshuffling}.

\begin{proof}
By shuffling a certain subset of $E_{\rho n}(T^\circ_n)$ as stated in Definition \ref{def:epsilonshuffling}, one moves subtrees rooted in children and grandchildren in $T_n$ of a white $\rho n$-branching point of $T_n^\circ$. In particular, using the fact that the number of black vertices in a subtree of $T_n$ is less than the number of white vertices in this subtree, the total number of vertices moved by the shuffling operation is at most $\sum_{u \in E_{\rho n}(T^\circ_n)} \sum_{v \in K^{(-2)}_u(T_n)} 2 |\theta_v(T^\circ_n)|$. 
Furthermore, with high probability, up to time $\tilde{B}_n \log n$ there is no black face of area larger whose color is changed between both colored lamination-valued processes. Indeed, there are at most $2 |E|$ black vertices with a subtree of size larger than $2 \epsilon n$ in $T_n$ that are moved by these operations. Thus, with high probability none of them has a label $\leq \tilde{B}_n \log n$.
The result follows.
\end{proof}

\begin{figure}[!h]
\caption{Top: the two possible cases for a white branching point $u$ of the tree $T_n$: either the two larger subtrees of grandchildren of $u$ have the same black parent (left), or two different black parents (right). The part that is (possibly) shuffled by the transformation of Definition \ref{def:epsilonshuffling} is in green (resp. red).
Bottom: after having switched the green and red parts, in the tree $T_n^{{\epsilon,\eta,\rho}}$. Remark that the set of degrees of the vertices stays the same on top and bottom.}
\label{fig:twocases}
\center
\begin{tabular}{c c c}

\begin{tikzpicture}
\draw[green] (-1.2,1) node{$S(u)$};
\draw[green, fill=green!10] plot [smooth cycle] coordinates  {({-.2*sqrt(2)/2},{.2*sqrt(2)/2}) (-.7,.3) (-.7,1.7) (0,2.5) (.7,1.7) (.7,.3) ({.2*sqrt(2)/2},{.2*sqrt(2)/2})};
\draw (0,0) -- (0,1) -- (-1,2); 
\draw (0,1) -- (1,2);
\draw[dashed] (0,1) -- (.5,1.25) -- (.5,1.05) -- cycle;
\draw[dashed] (0,1) -- (-.2,2) -- (.2,2) -- cycle;
\draw[dashed] ({.2*sqrt(2)/2},{.2*sqrt(2)/2}) -- (.6,.6) -- (.3,.6) -- cycle;
\draw[dashed] ({-.2*sqrt(2)/2},{.2*sqrt(2)/2}) -- (-.6,.6) -- (-.3,.6) -- cycle;
\draw[dashed] ({-1-.2*sqrt(2)/2},{2+.2*sqrt(2)/2}) -- (-1.5,2.5);
\draw[dashed] ({-1+.2*sqrt(2)/2},{2+.2*sqrt(2)/2}) -- (-.5,2.5);
\draw[dashed] ({1-.2*sqrt(2)/2},{2+.2*sqrt(2)/2}) -- (.5,2.5);
\draw[dashed] ({1+.2*sqrt(2)/2},{2+.2*sqrt(2)/2}) -- (1.5,2.5);
\draw (0,-.5) node{$u$};
\draw (-1.5,1.6) node{$v_1(u)$};
\draw (1.5,1.6) node{$v_2(u)$};
\draw (-1,2) -- (-1.3,3) -- (-.7,3) -- cycle;
\draw (1,2) -- (.7,3) -- (1.3,3) -- cycle;
\draw[fill=white] (0,0) circle (.2);
\draw[fill=white] (1,2) circle (.2);
\draw[fill=white] (-1,2) circle (.2);
\draw[fill=black] (0,1) circle (.2);
\end{tikzpicture}
&
\begin{tikzpicture}
\draw[white] (0,0) -- (2,0);
\draw (1,0) node{Two vertices in the tree $T_n$};
\end{tikzpicture}
&
\begin{tikzpicture}
\draw[red] (0,2) node{$S(u')$};
\draw[red, fill=red!10] plot [smooth cycle] coordinates  {(-1.5,1.5) (1.5,1.5) (1.3,.3) (-1.3,.3)};
\draw[dashed] (0,0) -- (-1,.5);
\draw[dashed] (0,0) -- (0,1);
\draw[dashed] (-1,1) -- (-.5,1.3);
\draw[dashed] (1,2) -- (2,3);
\draw (-1,2) -- (-1,1) -- (0,0) -- (1,1) -- (1,2);
\draw (-1,2) -- (-1.3,3) -- (-.7,3) -- cycle;
\draw (1,2) -- (.7,3) -- (1.3,3) -- cycle;
\draw (0,-.5) node{$u'$};
\draw (-1.5,1.6) node{$v_1(u')$};
\draw (1.5,1.6) node{$v_2(u')$};
\draw[fill=white] (0,0) circle (.2);
\draw[fill=white] (1,2) circle (.2);
\draw[fill=white] (-1,2) circle (.2);
\draw[fill=black] (1,1) circle (.2);
\draw[fill=black] (-1,1) circle (.2);
\end{tikzpicture}
\\
\begin{tikzpicture}
\draw[red] (0,2) node{$S(u')$};
\draw[red, fill=red!10] plot [smooth cycle] coordinates  {(-1.5,1.5) (1.5,1.5) (1.3,.3) (-1.3,.3)};
\draw[dashed] ({-1-.2*sqrt(2)/2},{2+.2*sqrt(2)/2}) -- (-1.5,2.5);
\draw[dashed] ({-1+.2*sqrt(2)/2},{2+.2*sqrt(2)/2}) -- (-.5,2.5);
\draw[dashed] ({1-.2*sqrt(2)/2},{2+.2*sqrt(2)/2}) -- (.5,2.5);
\draw[dashed] ({1+.2*sqrt(2)/2},{2+.2*sqrt(2)/2}) -- (1.5,2.5);
\draw[dashed] (0,0) -- (-1,.5);
\draw[dashed] (0,0) -- (0,1);
\draw[dashed] (-1,1) -- (-.5,1.3);
\draw (-1,2) -- (-1,1) -- (0,0) -- (1,1) -- (1,2);
\draw (-1,2) -- (-1.3,3) -- (-.7,3) -- cycle;
\draw (1,2) -- (.7,3) -- (1.3,3) -- cycle;
\draw (0,-.5) node{$u$};
\draw (-1.5,1.6) node{$v_1(u)$};
\draw (1.5,1.6) node{$v_2(u)$};
\draw[fill=white] (0,0) circle (.2);
\draw[fill=white] (1,2) circle (.2);
\draw[fill=white] (-1,2) circle (.2);
\draw[fill=black] (1,1) circle (.2);
\draw[fill=black] (-1,1) circle (.2);
\end{tikzpicture}
&
\begin{tikzpicture}

\draw[white] (0,0) -- (2,0);
\draw (1,0) node{The same vertices, exchanged in $T_n^{\epsilon,\eta,\rho}$};
\end{tikzpicture}
&
\begin{tikzpicture}
\draw[green] (-1.2,1) node{$S(u)$};
\draw[green, fill=green!10] plot [smooth cycle] coordinates  {({-.2*sqrt(2)/2},{.2*sqrt(2)/2}) (-.7,.3) (-.7,1.7) (0,2.5) (.7,1.7) (.7,.3) ({.2*sqrt(2)/2},{.2*sqrt(2)/2})};
\draw (0,0) -- (0,1) -- (-1,2); 
\draw (0,1) -- (1,2);
\draw[dashed] (0,1) -- (.5,1.25) -- (.5,1.05) -- cycle;
\draw[dashed] (0,1) -- (-.2,2) -- (.2,2) -- cycle;
\draw[dashed] ({.2*sqrt(2)/2},{.2*sqrt(2)/2}) -- (.6,.6) -- (.3,.6) -- cycle;
\draw[dashed] ({-.2*sqrt(2)/2},{.2*sqrt(2)/2}) -- (-.6,.6) -- (-.3,.6) -- cycle;
\draw[dashed] (1,2) -- (2,3);
\draw (0,-.5) node{$u'$};
\draw (-1.5,1.6) node{$v_1(u')$};
\draw (1.5,1.6) node{$v_2(u')$};
\draw (-1,2) -- (-1.3,3) -- (-.7,3) -- cycle;
\draw (1,2) -- (.7,3) -- (1.3,3) -- cycle;
\draw[fill=white] (0,0) circle (.2);
\draw[fill=white] (1,2) circle (.2);
\draw[fill=white] (-1,2) circle (.2);
\draw[fill=black] (0,1) circle (.2);
\end{tikzpicture}
\end{tabular}
\end{figure}
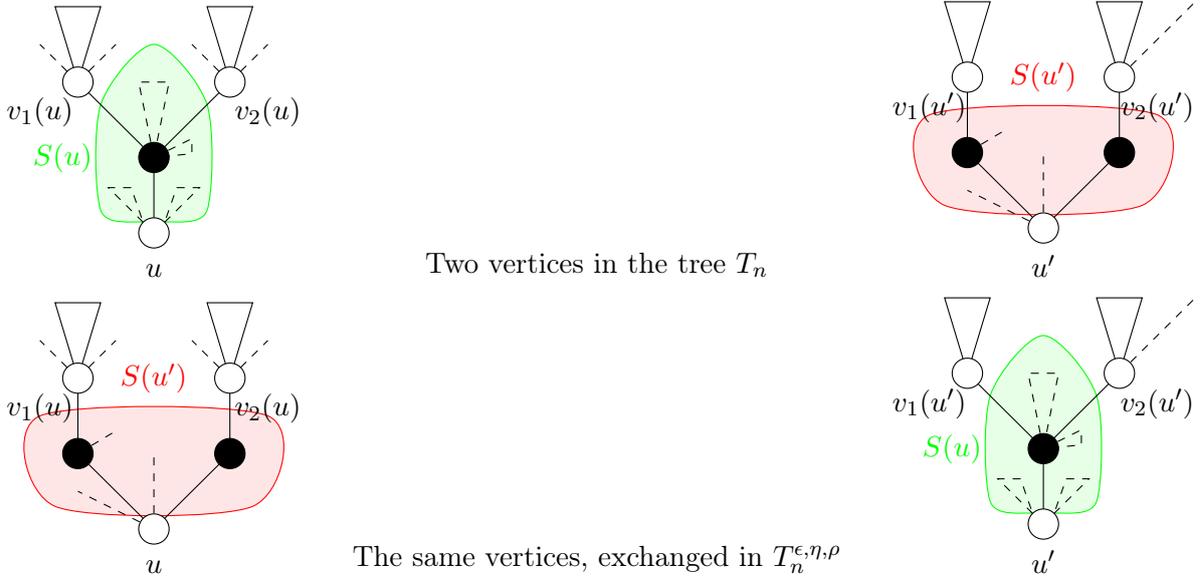

The idea is now to apply the transformation of Definition \ref{def:epsilonshuffling} to the tree $\cT_n$. It appears that one can choose the parameters $\eta$ and $\rho_n$ (depending on $n$) carefully, so that the colored lamination-valued process associated to $\cT_n^{\epsilon, \eta, \rho_n}$ converges in distribution towards $((\bL_c^{(2)})_{c \geq 0}, \bL_\infty^{(2), p_\nu})$ for some $p_\nu \in [0,1]$.

\begin{lem}
\label{lem:propertiesrandom}
Fix $\epsilon>0$ and set $\eta = \epsilon/6$. The following holds:
\begin{itemize}
\item[(i)] For all $n \geq 1$, for all $\rho >0$ such that $\rho<\eta$, conditionally to the fact that $\cT_n$ belongs to $\kBT_n^{\epsilon,\eta}$, $\cT_n^{\epsilon,\eta,\rho} \overset{(d)}{=} \cT_n$.
\item[(ii)] With high probability, $\cT_n$ belongs to $\kBT_n^{\epsilon,\eta}$.
\item[(iii)] 
Recall that $\mu$ is defined as the probability measure such that $\cT_n^\circ$ is a $\mu$-GW conditioned to have $n$ vertices. Define $K_\epsilon(\cT^\circ_n)$ as the (random) number of white $\epsilon n$-branching points in $\cT^\circ_n$, and label them $U_1, \ldots, U_{K_\epsilon(\cT^\circ_n)}$ in lexicographical order. Assume that $\cT_n$ belongs to $\kBT_n^{\epsilon, \eta}$. Then, for any $\epsilon'>0$, one can find $\rho>0$ such that, as $n \rightarrow \infty$, uniformly in $1 \leq j \leq \epsilon^{-1}$, uniformly in $k_1, \ldots, k_j \geq 1$:
\begin{align*}
\left| \P \left( \bigcup_{i=1}^{j} \left\{ k_{U_i}(\cT^{\circ,\epsilon,\eta,\rho}_n)=k_i \right\} \Big| K_\epsilon(\cT^\circ_n) = j \right)-  (\sigma_\mu^2)^{-j} \prod_{i=1}^{j} \mu_k k(k-1) \right| \leq \epsilon' + o(1),
\end{align*}
the $o(1)$ depending only on $n$.
\end{itemize}
\end{lem}

Let us see how it implies Theorem \ref{thm:convergenceblackwhite} (iii). First, by Lemma \ref{lem:propertiesrandom} (ii) and (iii), for any $M>0$ one can choose $\rho_M>0$ such that, for $n$ large enough, uniformly for $j \leq \epsilon^{-1}$, uniformly for any $k_1, \ldots, k_j \in \Z_+$:
\begin{align*}
\left| \P \left( \bigcup_{i=1}^j \left\{ k_{U_i}\left(\cT_n^{\circ, \epsilon, \epsilon/6,\rho_M}\right)=k_i \right\} \Big| K_\epsilon=j \right) - \left( \sigma_\mu^2 \right)^{-j} \prod_{i=1}^j \mu_{k_i} k_i(k_i-1) \right| < M^{-1}.
\end{align*}
On the other hand, at $\rho>0$ fixed, Lemma \ref{lem:bp} implies that $C_\rho(\cT_n) \overset{\P}{\rightarrow} 0$ in probability, as $n \rightarrow \infty$. Therefore, by diagonal extraction, one can find a sequence of parameters $(M_n)_{n \geq 1}$ such that the tree $\tilde{\cT}_n \coloneqq \cT_n^{\epsilon, \epsilon/6, \rho_{M_n}}$ satisfies the following conditions (using Lemma \ref{lem:Crho} to get (H2)):
\begin{itemize}
\item[(H1)] For all $n \geq 0$, $\tilde{\cT}_n \overset{(d)}{=} \cT_n$.
\item[(H2)] In probability,
\begin{align*}
\underset{0 \leq c \leq \log n}{\sup} d_H \left( \bL_{c\tilde{B}_n}^\bullet(\cT_n), \bL_{c\tilde{B}_n}^\bullet(\tilde{\cT}_n) \right) \overset{\P}{\rightarrow} 0.
\end{align*}
\item[(H3)] Uniformly for any $j \leq \epsilon^{-1}$, uniformly for any $k_1, \ldots, k_j \in \Z_+$
\begin{align*}
\P \left( \bigcup_{i=1}^j \left\{ k_{U_i}\left(\tilde{\cT}_n^\circ\right)=k_i \right\} \Big| K_\epsilon\left(\tilde{\cT}_n^\circ\right)=j \right) \underset{n \rightarrow \infty}{\rightarrow} \left( \sigma_\mu^2 \right)^{-j} \prod_{i=1}^j \mu_{k_i} k_i(k_i-1),
\end{align*}
where we recall that $U_i$ denotes the $i$-th $\epsilon n$-branching point of $\tilde{\cT}^\circ_n$.
\end{itemize}
Properties (H2) and (H3) mean in particular that the joint degree distribution of the $\epsilon n$-branching points in $\tilde{\cT}^\circ_n$ is asymptotically independent of the shape of the tree. We can now use this transformation, and notably (H3), to compute the value of the parameter $p_\nu$. To this end, we use the fact that, the number of white grandchildren of an $\epsilon n$-branching point $u$ being given equal to $k \geq 2$, the event that $v_1(u), v_2(u)$ have the same black parent is independent of the rest of the tree. Thus, by (H3), all faces that correspond to $\epsilon n$-branching points of $\cT_n$ in the limiting lamination are colored black in an i.i.d. way, with probability $p_\nu \in [0,1]$ given by the following proposition:

\begin{prop}
\label{prop:computation}
If $\nu$ has finite variance, then $p_\nu$ has the form:
\begin{align*}
p_\nu = \frac{\sigma_\nu^2}{\sigma_\nu^2+1}
\end{align*}
\end{prop}
Roughly speaking, to get this expression, we split according to the number $k$ of white children of a white branching point in $\cT^\circ_n$, thus computing the conditional probability given $k$ that such a white branching point codes a black face in $\bL_\infty^\bullet(\cT_n)$.

\begin{rk}
As said in Section \ref{sec:intro}, in the case $w = \delta^j$ for some $j \geq 2$, $\nu \coloneqq \frac{j-2}{j-1} \delta^0 + \frac{1}{j-1} \delta^j$, and this formula simplifies in $p_\nu=\frac{j-2}{j-1}$.
\end{rk}

\begin{proof}[Proof of Proposition \ref{prop:computation}]

According to (H3), $p_\nu$ is the limit as $n \rightarrow \infty$ of the sequence $(p_\nu^{(n)})_{n \geq 1}$, where:
\begin{align*}
p^{(n)}_\nu = \left(\sigma_\mu^2\right)^{-1} \sum_{k = 2}^{n-1} \mu_k k(k-1) \P \left( E(\emptyset) \Big| k_{\emptyset}(\cT^\circ_n) = k \right),
\end{align*}
where $E(u)$ is the event that $v_1(u), v_2(u)$ have the same black parent. Indeed, notice that, conditionally to having $k$ white grandchildren, the number $j$ of black children of a vertex is independent of the rest of the tree. Recall from \eqref{eq:defmubmun} the definition of $\mub$ and $\mun$, which are two probability measures satisfying $\cT_n \overset{(d)}{=} \cT_n^{(\mub, \, \mun)}$ for all $n \geq 1$. By construction of the tree, conditioning by $k_{\emptyset}(\cT^\circ_n) = k$ is the same as conditioning by $k_{\emptyset}(\cT^{\circ,(\mub, \, \mun)}) = k$. Hence,
\begin{equation}
\label{eq:pnuuu}
p_\nu = \left(\sigma_\mu^2\right)^{-1} \sum_{k = 2}^{\infty} \mu_k k(k-1) \P \left( E(\emptyset) \Big| k_{\emptyset}(\cT^\circ) = k \right),
\end{equation}
where we write $\cT$ instead of $\cT^{(\mub, \,\mun)}$ by convenience.

Finally, $j$ and $k$ being fixed, what is left to compute is the probability that the two grandchildren of $\emptyset$ with the largest subtrees rooted on them have the same black parent.

In order to compute $\P(E(\emptyset) | k_\emptyset(\cT^\circ)=k)$, remark that there are $k(k-1)$ possibilities for the locations of $v_1(\emptyset)$ and $v_2(\emptyset)$. Assuming that $u$ has $j$ black children, who respectively have $a_1,\ldots, a_j$ white children, the number of possible locations for $(v_1(\emptyset), v_2(\emptyset))$ such that they have the same black parent is $\sum_{i=1}^{j} a_i(a_i-1)$. More precisely, at $k$ fixed:

\begin{align*}
\P \left( E(\emptyset) \big| k_\emptyset(\cT_n^\circ)=k \right) &= \sum_{j=1}^k \P \left( E(\emptyset) \big| k_\emptyset(\cT)=j, k_\emptyset(\cT^\circ)=k \right)\\
&= \sum_{j=1}^k \sum_{\substack{a_1 + \cdots + a_j = k\\ a_1, \ldots, a_j \geq 1}} \P\left( G(a_1, \ldots, a_j) \big| k_\emptyset(\cT^\circ)=k \right) \P\left(E(\emptyset) \big| G(a_1, \ldots, a_j) \right)
\end{align*}
where $G(a_1, \ldots, a_j)$ is the event that $\emptyset$ has $j$ black children, who respectively have $a_1, \ldots, a_j$ white children. Thus, one just computes:
\begin{align*}
\P\left(E(\emptyset) \big| G(a_1, \ldots, a_j) \right) = \frac{1}{k(k-1)} \sum_{i=1}^j a_i(a_i-1)
\end{align*}
and
\begin{align*}
\P\left( G(a_1, \ldots, a_j) \big| k_\emptyset(\cT^\circ)=k \right) = \left(\P \left( k_\emptyset(\cT^\circ)=k \right) \right)^{-1} \mub_j \prod_{i=1}^j \mun_{a_i} = \frac{1}{\mu_k} \mub_j \prod_{i=1}^j \mun_{a_i},
\end{align*}
Hence, one gets for \eqref{eq:pnuuu}:
\begin{align*}
p_\nu &= \frac{1}{\sigma^2_\mu} \, \sum_{k = 2}^{\infty} \sum_{j=1}^k \, \mub_j \, \sum_{\substack{a_1,\ldots, a_j \geq 1 \\ \sum a_i=k}} \, \prod_{i=1}^j \mun_{a_i} \, \left( \sum_{i=1}^j a_i(a_i-1) \right)\\
&= \frac{1}{\sigma^2_\mu} \, \sum_{j = 1}^{\infty} \mub_j \, \sum_{a_1,\ldots, a_j \geq 1} \, \prod_{i=1}^j \mun_{a_i} \, \left( \sum_{i=1}^j a_i(a_i-1) \right) = \frac{1}{\sigma^2_\mu} \, \sum_{j = 1}^{\infty} \mub_j \, \E \left[ \sum_{i=1}^j X_i(X_i-1) \right],
\end{align*}
where the $X_i$'s are i.i.d. random variables of law $\mun$. By definition of $\mun$ and independence of the $X_i$'s, the expectation on the right-hand side is equal to $j \, (1-\nu_0)^{-1} \, \sigma_\nu^2$ since $\nu$ is critical. Thus, checking from \eqref{eq:variance} that $\sigma_\mu^2 = \sigma_\nu^2+1$, one gets:
\begin{align*}
p_\nu = \frac{1}{1-\nu_0} \, \frac{\sigma_\nu^2}{\sigma_\nu^2+1} \, \sum_{j=1}^\infty j \, \mub_j \, = \, \frac{\sigma_\nu^2}{\sigma_\nu^2+1}
\end{align*}
by \eqref{eq:meanmub}.
\end{proof}

We finally need to prove the technical lemma \ref{lem:propertiesrandom}.

\paragraph*{Proof of Lemma \ref{lem:propertiesrandom} (i) and (ii)} The image of any bi-type tree $T_n$ by the transformation of Definition \ref{def:epsilonshuffling} has the same weight as $T_n$, which implies (i).
In order to prove (ii), just observe that, if no subtree of the white reduced tree $\cT^\circ_n$ has size between $\epsilon n/6$ and $\epsilon n$, then its contour function attains at least twice the same local minimum, at two times at which it visits the same white vertex. More precisely, there exists $t_1 < t_2 < t_3 < t_4 \in [0,1]$ such that $C_{2nt_1}(T^\circ_n) = C_{2nt_2}(\cT^\circ_n) = C_{2nt_3}(\cT^\circ_n) = C_{2nt_4}(\cT^\circ_n)$, $C_{2ns}(\cT^\circ_n) \geq C_{2nt_1}(\cT^\circ_n)$ for all $s \in [t_1,t_4]$ and $t_2-t_1, t_3-t_2, t_4-t_3$ are all larger than $\epsilon/6$. The white vertex $u$ visited at these four times satisfies $|\theta_u(\cT^\circ_n)|\geq \epsilon n$, but for any of its children $v$, $|\theta_u(\cT^\circ_n)| < \epsilon n$ (such a vertex $u$ necessarily exists). But with high probability as $n \rightarrow \infty$ this does not occur. Indeed, by Theorem \ref{thm:cvcontour}, $C(\cT^\circ_n)$ converges after renormalization towards the Brownian excursion, whose local minima are almost surely unique.

\paragraph*{Proof of Lemma \ref{lem:propertiesrandom} (iii)}

The third part of this lemma focuses on the distribution of the degree of branching points in the white reduced tree $\cT_n^\circ$. Our main tool is therefore the following proposition, which computes the asymptotic distribution of the number of children of a branching point, in a large monotype size-conditioned tree. Recall that, a distribution $\mu$ being fixed, $\bT$ denotes a $\mu$-GW tree and, for any $n \geq 1$, $\bT_n$ denotes a $\mu$-GW tree conditioned to have $n$ vertices.

\begin{lem}
\label{lem:lienaveckesten}
Fix $\epsilon>0$, and let $\mu$ be a critical distribution in the domain with finite variance $\sigma_\mu^2$. Then:
\begin{itemize}
\item[(i)] For any $n \geq 1$, any $i \geq 2 \epsilon n+1$, any $k \geq 2$,
\begin{align*}
&\P \left( |\bT|=i, \emptyset \in E_{\epsilon n}(\bT), k_\emptyset(\bT)=k \right) = \\
& \qquad \qquad \qquad\mu_k \binom{k}{2} \sum_{q=0}^{i-1-2\epsilon n} \P( |\cF_{k-2}|= q) \sum_{t=\epsilon n}^{i-1-q-\epsilon n} \P(|\bT|=t) \P(|\bT|=i-1-q-t),
\end{align*}
where $\cF_j$ is a forest of $j$ i.i.d. $\mu$-Galton-Watson trees.

\item[(ii)]
Let $U$ be a uniform vertex in $\bT_n$. Then, for any $k \geq 2$:
\begin{align*}
\P \left( k_U(\bT_n) = k \Big| U \in E_{\epsilon n}(\bT_n) \right) \underset{n \rightarrow\infty}{\rightarrow} \mu_k \, k(k-1) \, (\sigma_\mu^2)^{-1},
\end{align*}
where we recall that $E_{\epsilon n}(T)$ denotes the number of $\epsilon n$-branching points in a tree $T$.

\item[(iii)]
Let $K_\epsilon(\bT_n)$ be the (random) number of $\epsilon n$-branching points in $\bT_n$, and denote them by $U_1, \ldots, U_{K_\epsilon(\bT_n)}$ in lexicographical order. Then, for all $j \geq 0$, all $k_1, \ldots, k_j \in \Z_+$:
\begin{align*}
\P \left(\bigcup_{i=1}^j \left\{ k_{U_i}(\bT_n) = k_i \right\} \big|  K_\epsilon(\bT_n) = j \right) \underset{n \rightarrow \infty}{\rightarrow} \left( \sigma_\mu^2\right)^{-j}\prod_{i=1}^j \mu_{k_i} k_i(k_i-1).
\end{align*}
\end{itemize}
\end{lem}

The proof of this lemma is postponed to the end of this section. Let us see how it implies Lemma \ref{lem:propertiesrandom} (iii).
To this end, $\epsilon$ and $\eta \coloneqq \epsilon/6$ being fixed, we adapt the parameter $\rho$ in order to control the distribution of the degrees of the $\epsilon n$-branching points. Let $u$ be the vertex chosen in the transformation of Definition \ref{def:epsilonshuffling}. By Lemma \ref{lem:lienaveckesten} (iii), we know how the degrees of the $\rho n$-branching points in $\theta_u(\cT^\circ_n)$ behave. Therefore, by choosing $\rho$ so that the number of $\rho n$-branching points in $\theta_u(\cT^\circ_n)$ is much larger than the number of $\epsilon n$-branching points in $\cT_n$ with high probability, we can control the degree distribution of these $\epsilon n$-branching points of $\tilde{\cT}^\circ_n$, after the shuffling operation.

Specifically, for any $q \geq 1$, we define $\rho_q$ such that, uniformly for $\ell \in (\eta n, \epsilon n)$ satisfying $Z_{\ell,\mu_*,w}>0$, with probability larger than $1-1/q$ there are at least $q$ $\rho_q n$-branching points in $\cT_{\ell}^\circ$. The existence of such a $\rho_q$, for $q \geq 1$, is a consequence of the convergence of Theorem \ref{thm:cvcontour}, and the fact that the set of local minima of the normalized Brownian excursion is dense in $[0,1]$ with probability $1$.

Now for any $\epsilon'>0$, Lemma \ref{lem:lienaveckesten} (iii) ensures that one can choose $q>0$ such that with high probability, uniformly in $k$, for $n$ large enough,

\begin{align*}
\left|N_q^{(k)}\left(\theta_u\left(\cT^\circ_n\right)\right) - \mu_k k(k-1) \, \left(\sigma_\mu^2\right)^{-1} N_q\left(\theta_u\left(\cT^\circ_n\right)\right)\right| \leq \epsilon' N_q\left(\theta_u\left(\cT^\circ_n\right)\right).
\end{align*}
Here, $N_q(T)$ denotes the number of $\rho_q n$-branching points in $T$ and $N_q^{(k)}(T)$ denotes the number of $\rho_q n$-branching points who have $k$ black children. In other terms, the proportion of $\rho_q n$-branching points in $\theta_u(\cT_n)$ that have $k$ children is asymptotically proportional to $\mu_k k(k-1)$. As there are at most $\epsilon^{-1}$ $\epsilon n$-branching points in $\cT_n$, Lemma \ref{lem:propertiesrandom} (iii) follows.

Let us finish with the proof of Lemma \ref{lem:lienaveckesten}.

\begin{proof}[Proof of Lemma \ref{lem:lienaveckesten} (i)]
To prove (i), remark that, for any $k \geq 2$, any $i \geq 2\epsilon n+1$,
\begin{align*}
\P \left( |\bT|=i, \emptyset \in E_{\epsilon n}(\bT), k_\emptyset(\bT)=k \right) &= \P \left( k_\emptyset(\bT)=k \right) \P \left( |\bT|=i, \emptyset \in E_{\epsilon n}(\bT) \Big| k_\emptyset(\bT)=k \right)
\end{align*}
The right-hand side can be estimated thanks to Lemma \ref{lem:bp}, through the formula:
\begin{align*}
\P \left( |\bT|=i, \emptyset \in E_{\epsilon n}(\bT) \Big| k_\emptyset(\bT)=k \right) = \sum_{1 \leq a < b \leq k} \P\left( |\bT|=i, B_{\epsilon, a, b} \Big| k_\emptyset(\bT)=k  \right)
\end{align*}
where $B_{\epsilon,a,b}$ is the event that the subtrees rooted in the $a$-th and $b$-th children of $\emptyset$ have size $\geq \epsilon n$. Thus, since $\P(k_\emptyset(\bT)=k)=\mu_k$:
\begin{align*}
\P \left( |\bT|=i, \emptyset \in E_{\epsilon n}(\bT), k_\emptyset(\bT)=k \right) &= \mu_k \binom{k}{2} \sum_{\substack{t_1 \geq \epsilon n \\ t_2 \geq \epsilon n \\ t_1+t_2 \leq i-1}} \P(|\bT|=t_1) \P(|\bT|=t_2) \P\left( |\cF_{k-2}|= i-1-t_1-t_2 \right),
\end{align*}
where $\cF_j$ is a forest of $j$ i.i.d. $\mu$-GW trees.

Separating according to the value $q \coloneqq i-1-t_1-t_2$, the right-hand side is equal to
\begin{align*}
\mu_k \binom{k}{2} \sum_{q=0}^{i-1-2\epsilon n} \P( |\cF_{k-2}|= q) \sum_{t=\epsilon n}^{i-1-q-\epsilon n} \P(|\bT|=t) \P(|\bT|=i-1-q-t).
\end{align*}
\end{proof}

\begin{proof}[Proof of Lemma \ref{lem:lienaveckesten} (ii)]
Let $U$ be a uniform vertex of $\bT_n$. Then: 
\begin{align*}
\P \left(U \in E_{\epsilon n} (\bT_n), k_{U}(\bT_n)=k \right) &= \frac{1}{n} \sum_{j \geq 1} \E \left[ \sum_{\substack{u \in \bT_n \\ |u|=j}} \mathds{1}_{u \in E_{\epsilon n}(\bT_n)} \mathds{1}_{k_u(\bT_n)=k} \right]\\
&= \frac{1}{n} \P \left( |\bT|=n \right)^{-1} \sum_{j \geq 1} \sum_{i=2\epsilon n}^n \E \left[ \sum_{\substack{u \in \bT \\ |u|=j}} F_{n-i}(Cut_u(\bT)) G_{i,k}(\theta_u(\bT)) \right],
\end{align*}
where $F_{n-i}(T) = \mathds{1}_{|T|=n-i}$ and $G_{i,k}(T)= \mathds{1}_{|T|=i} \mathds{1}_{\emptyset \in E_{\epsilon n}(T)} \mathds{1}_{k_\emptyset(T)=k}$. Here, for $T$ a tree and $u$ a vertex of $T$, $Cut_u(T)$ denotes the tree $T\backslash \theta_u(T)$, obtained by cutting $T$ at the level of $u$ (not keeping $u$). 

In order to investigate this quantity, let us now define $\bT^*$, the so-called local limit of the conditioned Galton-Watson trees $\bT_n$: $\bT^*$ is a random variable taking its values in the set of infinite trees, and satisfies, for all $r \geq 1$,
\begin{align*}
B_r(\bT_n) \overset{(d)}{\rightarrow} B_r(\bT^*),
\end{align*}
where, a tree $T$ (finite or infinite) being given, $B_r(T)$ denotes the ball of radius $r$ around the root of $T$ for the graph distance - that is, all edges have length $1$. The structure of this tree $\bT^*$, called Kesten's tree, is known: it has a unique infinite branch, on which independent nonconditioned $\mu$-GW trees are planted. See Fig. \ref{fig:kestentree} for an illustration, and \cite{Kes86} for more background. Information on the large tree $\bT_n$ can therefore be deduced from the properties on $\bT^*$. In particular, by estimations \textit{à la} Lyons-Pemantle-Peres (see \cite[Section $3$]{Duq08}), we obtain that, for any $j \geq 0$,
\begin{align*}
\E \left[ \sum_{\substack{u \in \bT \\ |u|=j}} F_{n-i}(Cut_u(\bT)) G_{i,k}(\theta_u(\bT)) \right] = \E \left[ F_{n-i}\left(Cut_{U^*_j}(\bT^*) \right) \right] \E \left[ G_{i,k}(\bT) \right].
\end{align*} 
where $U_j^*$ denotes the unique vertex of the infinite branch of $\bT^*$ at height $j$.
We can now use Lemma \ref{lem:lienaveckesten} (i) to obtain an expression of $\E[G_{i,k}(\bT)]$:
\begin{equation}
\label{eq:equationcompliquee}
\sum_{j \geq 0} \sum_{i=2\epsilon n}^n \E \left[ \sum_{\substack{u \in \bT \\ |u|=j}} F_{n-i}(Cut_u(\bT)) G_{i,k}(\theta_u(\bT)) \right] = \, \mu_k \binom{k}{2} \sum_{i=2\epsilon n}^n A_i \sum_{q=0}^{i-1-2\epsilon n} \P( |\cF_{k-2}|= q)  B_{i,q},
\end{equation}
where, for $i \in \llbracket 2\epsilon n, n \rrbracket, A_i = \sum_{j \geq 0} \P(|Cut_{U^*_j}(\bT^*)|=n-i)$ and, for $q \in \llbracket 0,i-1-2\epsilon n \rrbracket$, 
\begin{align*}
B_{i,q}= \sum_{t=\epsilon n}^{i-1-q-\epsilon n} \P(|\bT|=t) \P(|\bT|=i-1-q-t).
\end{align*}

\begin{figure}[!h]
\center
\caption{Kesten's infinite tree $\cT^*$. On the infinite branch (in the middle), independent $\mu$-GW are planted.}
\label{fig:kestentree}
\begin{tikzpicture}[scale=1, every node/.style={scale=0.7}]
\draw (0,0) -- (0,5);
\draw[dashed] (0,5) -- (0,6);
\foreach \x in {0,...,5}
\draw[fill] (0,\x) circle (.05);
\draw (0,0) -- (1,1);
\draw (0,2) -- (-1,3) (0,2) -- (-2,3);
\draw (0,3) -- (-1,4) (2,4) -- (0,3) -- (1,4);
\draw (0,4) -- (1,5);
\draw[fill=gray!20] (1,1.35) circle (.35) node{$GW$}; 
\draw[fill=gray!20] (-2,3.35) circle (.35) node{$GW$}; 
\draw[fill=gray!20] (-1,3.35) circle (.35) node{$GW$}; 
\draw[fill=gray!20] (-1,4.35) circle (.35) node{$GW$}; 
\draw[fill=gray!20] (1,4.35) circle (.35) node{$GW$}; 
\draw[fill=gray!20] (2,4.35) circle (.35) node{$GW$}; 
\draw[fill=gray!20] (1,5.35) circle (.35) node{$GW$}; 
\end{tikzpicture}
\end{figure}
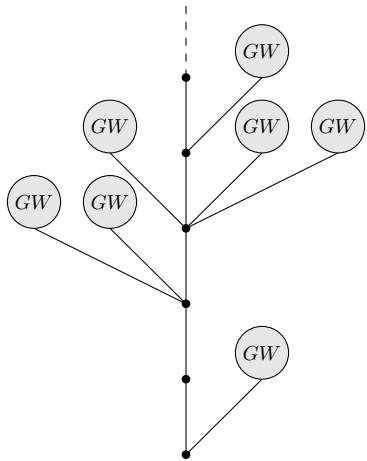

Set, for $n, k \in \Z_+$,
\begin{align*}
R^{(n)}_k \coloneqq \sum_{i=2\epsilon n}^n A_i \sum_{q=0}^{i-1-2\epsilon n} \P( |\cF_{k-2}|= q)  B_{i,q}.
\end{align*} 
In order to prove Lemma \ref{lem:lienaveckesten} (ii), we show two things: 
\begin{itemize}
\item[(a)] there is no loss of mass as $n$ grows, in the sense that, for any $\gamma>0$, there exists $K \in \Z_+$ such that, for any $n$ large enough,
\begin{align*}
\sum_{k > K} \mu_k \binom{k}{2} R^{(n)}_k \leq \gamma\sum_{k \leq K} \mu_k \binom{k}{2} R^{(n)}_k.
\end{align*}
In other words, the degree of a uniform $\epsilon n$-branching point in $\bT_n$ is tight;
\item[(b)] uniformly for $k_1,k_2$ on a compact subset of $\Z_+$, as $n \rightarrow \infty$:
\begin{align*}
R^{(n)}_{k_1} \sim R^{(n)}_{k_2}.
\end{align*}
\end{itemize}

By (a),(b) and \eqref{eq:equationcompliquee}, we conclude that, for all $k \geq 2$, $\P \left( k_U(\bT_n) = k \, \Big| \, U \in E_{\epsilon n}(\bT_n) \right)$ is asymptotically proportional to $\mu_k \binom{k}{2}$. This implies Lemma \ref{lem:lienaveckesten} (ii).

We finish with the proofs of (a) and (b). Let us first prove (a). By the local limit theorem \ref{thm:llt}, as $\mu$ has finite variance, there exists two constants $C>c>0$ depending only on $\epsilon$ and $\mu$ such that, for any $i \in \llbracket 2\epsilon n, n \rrbracket$, any $0 \leq q \leq i-1-2\epsilon n$, 

\begin{equation}
\label{eq:biq}
c (i-1-2\epsilon n - q) n^{-3} \leq B_{i,q} \leq C (i-1-2\epsilon n - q) n^{-3}.
\end{equation}
Thus, for any $k \in \Z_+$, 
\begin{equation}
\label{eq:bornesuprkn}
R^{(n)}_k \leq C n^{-3} \sum_{i=2\epsilon n}^n A_i \, n \, \P( | \cF_{k-2} | \leq n) \leq C n^{-2} \sum_{i=2\epsilon n}^n A_i.
\end{equation}

Now, take $\eta = 1/2+\epsilon \in (2\epsilon,1)$. We claim that 
\begin{equation}
\label{eq:subadd}
\sum_{i=2\epsilon n}^n A_i \leq 3 \sum_{i=\eta n}^n A_i,
\end{equation}
which we prove later. Then, for any $i \geq \eta n$, by \eqref{eq:biq}, 
\begin{align*}
\sum_{q=0}^{i-1-2\epsilon n} \P( |\cF_{k-2}|= q)  B_{i,q} &\geq \sum_{q=0}^{\eta n-2\epsilon n} \P( |\cF_{k-2}|= q)  B_{i,q} \geq c n^{-3} \, (\eta - 2 \epsilon) \, n \, \P\left( |\cF_{k-2}| \leq (\eta-2\epsilon) n \right).\\
\end{align*}
At $k$ fixed, for $n$ large enough, this quantity is larger than $c \, n^{-2} \, (\eta/2-\epsilon)$, and by \eqref{eq:subadd} $$R_k^{(n)} \geq \sum_{i=\eta n}^n A_i \sum_{q=0}^{i-1-2\epsilon n} \P( |\cF_{k-2}|= q)  B_{i,q} \geq \frac{c}{3} (\eta/2-\epsilon) n^{-2} \sum_{i=2\epsilon n}^n A_i.$$

Using \eqref{eq:bornesuprkn} and the fact that $\sum_{k>K} \mu_k \binom{k}{2} \rightarrow 0$ as $K \rightarrow \infty$, this implies (a) and ensures the tightness of the degree of a uniform $\epsilon n$-branching point.

The only thing left to prove is that, indeed, $\sum_{i=2\epsilon n}^n A_i \leq 2 \sum_{i=\eta n}^n A_i$. To this end, remark that, by definition, for any $\delta \in (0,1)$,
\begin{align*}
\sum_{i=\lfloor \delta n \rfloor}^n A_i& = \sum_{i=\lfloor \delta n \rfloor}^n \sum_{j \geq 0} \P\left(|Cut_{U^*_j}(\bT^*)| = n-i \right) = \sum_{j \geq 0} \P\left( |Cut_{U^*_j}(\bT^*)| \leq n-\lfloor \delta n \rfloor \right)\\
&= \E\left[ \sup \left\{ j \geq 0, |Cut_{U^*_j}(\bT^*)| \leq n-\lfloor \delta n \rfloor \right\} \right] = \E\left[ \inf \left\{ j \geq 1, |Cut_{U^*_j}(\bT^*)| > n-\lfloor \delta n \rfloor \right\} \right] - 1.
\end{align*}
Now, by definition of the tree $\bT^*$, for any $j \geq 1$, $|Cut_{U^*_j}(\bT^*)|$ is the sum of $j$ i.i.d. random variables. In particular, the sequence $(u_r)_{r \geq 0}$ defined as 
\begin{align*}
u_r = \E\left[ \inf \left\{ j \geq 1, |Cut_{U^*_j}(\bT^*)| > r \right\} \right]
\end{align*}
is clearly subadditive, in the sense that, for all $r_1,r_2 \geq 0$, $u_{r_1+r_2} \leq u_{r_1} + u_{r_2}$. On the other hand this sequence is increasing and goes to $+\infty$. This proves \eqref{eq:subadd}, since $n-2\epsilon n = 2 (n-\eta n)$.

\bigskip

To prove (b), just remark that, at $k$ fixed, the mass of $R^{(n)}_k$ is asymptotically concentrated on small values of $q$. Indeed, by \eqref{eq:biq}, uniformly for $i \geq 2\epsilon n+1+\log n$,
\begin{align*}
\sum_{q=\log n}^{i-1-2\epsilon n} \P \left( |\cF_{k-2}| = q \right) B_{i,q} \leq C \, n^{-3} (i-1-2\epsilon n - \log n) \P \left( |\cF_{k-2}| \geq \log n \right) = o\left( \sum_{q=0}^{\log n} \P \left( |\cF_{k-2}| = q \right) B_{i,q} \right), 
\end{align*}
since, by \eqref{eq:biq} again,
\begin{align*}
\sum_{q=0}^{\log n} \P ( |\cF_{k-2}| = q ) B_{i,q} \geq c \, n^{-3} \, (i-1-2\epsilon n-\log n) \P (|\cF_{k-2}| \leq \log n) \geq \frac{c}{2} \, n^{-3} \, (i-1-2\epsilon n-\log n)
\end{align*}
for $n$ large enough. 
Now, by Theorem \ref{thm:llt}, for any $i \geq 2 \epsilon n+1$, there exists a constant $\tilde{C}_i$ such that, as $n \rightarrow \infty$, uniformly for $q \leq \log n$, $B_{i,q} \sim \tilde{C}_i n^{-3}$. Thus, for any $k \geq 2$ fixed, 
\begin{align*}
R_k^{(n)} \sim n^{-3} \sum_{i=2\epsilon n}^n A_i \tilde{C}_i \sum_{q=0}^{\log n} \P (|\cF_{k-2}| = q) \sim n^{-3} \sum_{i=2\epsilon n}^n A_i \tilde{C}_i.
\end{align*}
In particular, this implies (b).
\end{proof}

\begin{proof}[Proof of Lemma \ref{lem:lienaveckesten} (iii)]
In order to check (iii), one only needs to see that, conditionally to its size, a subtree of $\bT_n$ is independent of the rest of the tree. Therefore, taking $u$ an $\epsilon n$-branching point of $\bT_n$, the subtrees $\theta_{v_1(u)}(\bT_n)$ and $\theta_{v_2(u)}(\bT_n)$, conditionally to their sizes (which are larger than $\epsilon n$ by definition) are independent of the rest of the tree. Hence, using repeatedly Lemma \ref{lem:lienaveckesten} (ii) on these subtrees, one obtains (iii).
\end{proof}

\section{A bijection between minimal factorizations and a set of bi-type trees}
\label{sec:minfac}

In this section, we first discuss some properties of minimal factorizations of the $n$-cycle, and specify a way to code them by bi-type trees with $n$ white vertices. In a second time, we use this bijection to code a $w$-minimal factorization of the $n$-cycle $f_n^w$ by a random labelled BTSG $T(f_n^w)$, with some constraints on the labels of its black vertices. This BTSG with constraints is of particular interest, as we prove that the process $(S_u(f_n^w))_{u \in [0,\infty]}$ is asymptotically close to the black process of this tree:

\begin{thm}
\label{thm:mainthmsection4}
Let $\alpha \in (1,2]$. Let $w$ be a sequence of $\alpha$-stable type, $\nu$ its critical equivalent and $(\tilde{B}_n)_{n \geq 1}$ satisfying \eqref{eq:tbn} for $\nu$. Then, if $\alpha<2$ or if $\nu$ has finite variance, the face configuration process obtained from $f_n^w$ is close in distribution to the process associated to a $(\mu_*,w)$-BTSG with uniformly labelled black vertices. More precisely, there exists a coupling of $f_n^w$ and $\cT_n^{(\mu_*,w)}$ such that, in probability:
\begin{align*}
d_{Sk}\left(\left( S_{c \tilde{B}_n}(f_n^w) \right)_{c \in [0,\infty]}, \left( \bL_{c \tilde{B}_n}^\bullet \left(\cT_n^{(\mu_*,w)} \right) \right)_{c \in [0,\infty]} \right) \overset{\P}{\underset{n \rightarrow \infty}{\rightarrow}} 0,
\end{align*}
where $d_{Sk}$ denotes the Skorokhod distance on $\D([0,\infty],\bCL(\oD))$.
\end{thm}

This theorem, which we prove later in the section, directly implies Theorem \ref{thm:mainresult}:

\begin{proof}[Proof of Theorem \ref{thm:mainresult}]
The proof of the main result in this paper, Theorem \ref{thm:mainresult}, is just a consequence of Theorem \ref{thm:cvbitypelam} and Theorem \ref{thm:mainthmsection4}.
\end{proof}

The principal tool in the proof of Theorem \ref{thm:mainthmsection4} is a operation that we perform on the white vertices of $T(f_n^w)$, which consists in shuffling its black children in two different ways, in order to lift the constraints on this labelling. The aim is to obtain at the end a tree distributed as $\cT_n^{(\mu_*,w)}$ (that is, its black vertices are uniformly labelled), whose black process is close in probability to the one of $T(f_n^w)$. See Section \ref{ssec:shuffling} for details.

\subsection{Coding a minimal factorization by a colored lamination}
\label{ssec:minfactree}

In a first time, our aim is to prove Theorem \ref{thm:nkbijection} by showing an explicit bijective way to code a factorization of $\kM_n$ by a tree of $\kU_n$. We do it in two steps, first coding a minimal factorization by a colored lamination of $\oD$ and then coding it by a bi-type tree. 

It is to note that our bijection is close to the one presented by Du and Liu \cite{DL15}, who investigate minimal factorizations of a given cycle. Notably, what they call $S-[d]$ bipartite graphs is exactly what we call the "dual tree" of the factorization. In their paper, Du and Liu use this bipartite graph as a tool to show a bijection between minimal factorizations and a new family of trees which they call multi-noded rooted trees; we prefer studying the bipartite graph (or bi-type tree in our case) directly, as its structure allows to use the machinery of random trees and seems more adapted in our setting.

The first step consists in adapting the bijection introduced by Goulden and Yong \cite{GY02} to code minimal factorizations into transpositions by monotype trees with labelled vertices. In our broader framework, we code general minimal factorizations by bi-type trees with $n$ white vertices and labelled black vertices. We check in a first time that we can code a minimal factorization of the $n$-cycle by a colored lamination, as explained in Section \ref{ssec:codingminfac}. For this, we need to be able to define the face associated to a cycle appearing in the factorization.

For $n \geq 1$, we say that a cycle $\tau \in \kC_n$ is increasing if it can be written as $(e_1 \, e_2 \, \cdots \, e_{\ell(\tau)})$, where $e_1 < e_2 < \ldots < e_{\ell(\tau)}$.

\begin{prop}
\label{prop:increasingcycles}
Let $n \geq 1$. Then any cycle appearing in a minimal factorization of the $n$-cycle is increasing.
\end{prop}

To prove Proposition \ref{prop:increasingcycles}, we make use of what we call the \textit{transposition slicing} of a minimal factorization, which is roughly speaking a decomposition into transpositions of the factorization:

\begin{defi}
Let $n,k \geq 1$ and $f \coloneqq (\tau_1,\ldots,\tau_k) \in \kM_n^{(k)}$. We define from $f$ a factorization $\tilde{f}$ of the $n$-cycle into transpositions as follows: for $1 \leq i \leq k$, let us write the cycle $\tau_i$ as $(d^{(i)}_1 \ldots d^{(i)}_{\ell(\tau_i)})$, where $d^{(i)}_1$ is the minimum of the support of $\tau_i$. Now remark that $\tau_i$ can be written as the product of $(\ell(\tau_i)-1)$ transpositions: $(d_1^{(i)} d_2^{(i)}) \, (d_1^{(i)} d_3^{(i)}) \cdots (d_1^{(i)} d_{\ell(\tau_i)}^{(i)})$. By replacing all cycles of $f$ by their decomposition into transpositions, we obtain a factorization $\tilde{f}$ of the $n$-cycle into transpositions, which we call the \textit{transposition slicing} of $f$.
\end{defi}

See Fig. \ref{fig:transpositionslicing} for an example.
It is clear that, if $f$ is a minimal factorization of the $n$-cycle, then its transposition slicing $\tilde{f}$ is made of $n-1$ transpositions, and hence is a minimal factorization of the $n$-cycle into transpositions. This allows to translate results on $\tilde{f}$ (mostly taken from \cite{GY02}) into results on $f$.

\begin{figure}
\caption{The labelled colored laminations $S_{lab}(f)$ and $S_{lab}(\tilde{f})$, for $f \coloneqq (5678)(23)(125)(45)$. Constructing the second one from the first one just consists in triangulating each black face, starting from the smallest of its vertices.}
\label{fig:transpositionslicing}
\center
\begin{tabular}{c c c}
\begin{tikzpicture}[scale=.8,rotate=-45]
\draw[red] (0,0) circle (3);

\foreach \k in {1,...,8}
{\draw ({3.25*cos(360*(\k-1)/8)},{-3.25*sin(360*(\k-1)/8)}) node{\k};}

\draw[red,fill=black] ({3*cos(360*0/8)},{-3*sin(360*0/8)}) -- ({3*cos(360*4/8)},{-3*sin(360*4/8)}) -- ({3*cos(360/8)},{-3*sin(360/8)}) -- cycle;

\draw[red] ({3*cos(360*4/8)},{-3*sin(360*4/8)}) -- node[black]{4} ({3*cos(360*3/8)},{-3*sin(360*3/8)});

\draw[red] ({3*cos(360*1/8)},{-3*sin(360*1/8)}) -- node[black]{2} ({3*cos(360*2/8)},{-3*sin(360*2/8)});

\draw[red,fill=black] ({3*cos(360*7/8)},{-3*sin(360*7/8)}) -- ({3*cos(360*4/8)},{-3*sin(360*4/8)}) -- ({3*cos(360*5/8)},{-3*sin(360*5/8)}) -- ({3*cos(360*6/8)},{-3*sin(360*6/8)}) -- cycle ;

\draw[fill=black] (-.7,1.7) circle (.3) node[white]{1};

\draw[fill=black] (1.2,-.7) circle (.3) node[white]{3};

\draw[fill=black] (1.1,-2.5) circle (.3) node[white]{2};

\draw[fill=black] (-2.5,-1) circle (.3) node[white]{4};
\end{tikzpicture}
&
\begin{tikzpicture}
\draw[white] (0,0) -- (1,0);
\end{tikzpicture}
&
\begin{tikzpicture}[scale=.8,rotate=-45]
\draw[red] (0,0) circle (3);

\foreach \k in {1,...,8}
{\draw ({3.25*cos(360*(\k-1)/8)},{-3.25*sin(360*(\k-1)/8)}) node{\k};}

\draw[red] ({3*cos(360*4/8)},{-3*sin(360*4/8)}) -- node[circle,white,fill=black]{6} ({3*cos(360*0/8)},{-3*sin(360*0/8)}) -- node[circle,white,fill=black]{5} ({3*cos(360/8)},{-3*sin(360/8)});

\draw[red] ({3*cos(360*4/8)},{-3*sin(360*4/8)}) -- node[circle,white,fill=black]{7} ({3*cos(360*3/8)},{-3*sin(360*3/8)});

\draw[red] ({3*cos(360*1/8)},{-3*sin(360*1/8)}) -- node[circle,white,fill=black]{4} ({3*cos(360*2/8)},{-3*sin(360*2/8)});

\draw[red] ({3*cos(360*7/8)},{-3*sin(360*7/8)}) -- node[circle,white,fill=black]{3} ({3*cos(360*4/8)},{-3*sin(360*4/8)}) -- node[circle,white,fill=black]{1} ({3*cos(360*5/8)},{-3*sin(360*5/8)});

\draw[red] ({3*cos(360*6/8)},{-3*sin(360*6/8)}) -- node[circle,white,fill=black]{2} ({3*cos(360*4/8)},{-3*sin(360*4/8)});
\end{tikzpicture}
\end{tabular}
\end{figure}
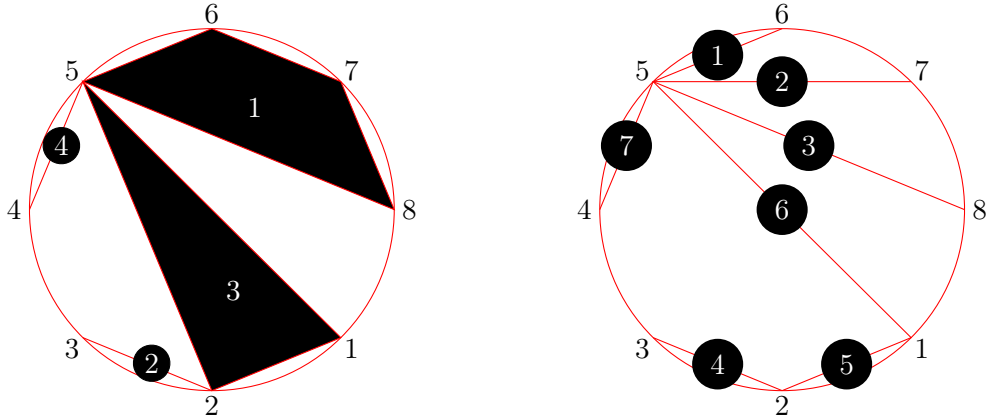

\begin{proof}[Proof of Proposition \ref{prop:increasingcycles}]
Let $n,k \geq 1$ and $f \coloneqq (\tau_1, \ldots, \tau_k) \in \kM_n^{(k)}$. We first construct a lamination, denoted by $S_{lab}(\tilde{f})$, by giving labels to the faces of $S(\tilde{f})$ (which are in fact all chords): draw for the $i$-th transposition of $\tilde{f}$, say $(a_i \, b_i)$, the chord $[e^{-2i\pi a_i/n}, e^{-2i\pi b_i/n}]$, and label it $i$. This provides a  lamination in which all chords are labelled. Furthermore, \cite[Theorem $2.2$ (iii)]{GY02} states that chords in $S_{lab}(\tilde{f})$ are labelled in increasing clockwise order around each vertex $e^{-2i j \pi/n}, 1 \leq j \leq n$ (remark that it is sorted in decreasing clockwise order in \cite{GY02}: indeed their coding is slightly different from ours, since they label the $n$-th roots of unity decreasingly clockwise. Nonetheless our result is an easy consequence of theirs).

Thus, for any $i \leq k$, around the vertex $e^{-2i\pi d_1^{(i)}/n} \in \oD$, the chords of $S_{lab}(\tilde{f})$ are labelled in clockwise increasing order (see an example on Fig. \ref{fig:transpositionslicing}, right). This implies that $d_1^{(i)} < d_2^{(i)} < \ldots < d_{\ell(i)}^{(i)}$ and the result follows.
\end{proof}

We can therefore code a factorization $f$ by a colored lamination with labelled faces, by labelling the black faces of $S(f)$ from $1$ to $k$ (where $k$ denotes the number of cycles that appear in $f$) in the order in which they appear. See Fig. \ref{fig:transpositionslicing}, left for an example. We denote this labelled colored lamination by $S_{lab}(f)$.

\begin{prop}
\label{prop:properties}
Let $n,k \geq 1$ and $f \in \kM_n^{(k)}$. Then $S_{lab}(f)$ satisfies the following properties:
\begin{itemize}
\item[$P_1$:] It has $k$ black faces and $n$ white faces (with the convention that chords corresponding to a cycle of length $2$ are considered to be black faces).
\item[$P_2$:] A black (resp. white) face has only white (resp. black) neighbouring faces (with the same convention).
\item[$P_3$] The set of black faces of $S_{lab}(f)$ obeys a noncrossing tree-like structure. Specifically, the chords only meet at their endpoints and form a connected graph; in addition, there is no cycle of chords of length $\geq 1$ containing at most one edge of each black face.
\item[$P_4$:] Around each $n$-th root of unity, black faces are labelled in increasing clockwise order.
\end{itemize}
\end{prop}

These properties can be easily deduced from the results of \cite[Section $2$]{GY02}, which straightforwardly imply that they are satisfied by $S_{lab}(\tilde{f})$. In particular, the chords of $S_{lab}(\tilde{f})$ form a tree and are labelled in increasing clockwise order around each vertex (see Fig. \ref{fig:transpositionslicing}, right, for an example).

\begin{proof}
Let us first check $P_1$. It is clear by definition that $S_{lab}(f)$ has $k$ black faces. Now remark that each white face contains exactly one arc of the form $\wideparen{(e^{-2i\pi a/n},e^{-2i\pi(a+1)/n})}$ (where $a \in \Z)$ in its boundary, since the chords of $S_{lab}(\tilde{f})$ form a tree (see \cite[Theorem $2.2$ (i)]{GY02}). As there are exactly $n$ such arcs, $P_1$ is satisfied.

To prove $P_2$, notice that two white faces cannot be neighbours, as they need a chord to separate them, which belongs to a black face. In addition, one can check that two black faces cannot have a chord in common in their boundaries; otherwise either the same transposition would appear twice in $\tilde{f}$, or there would be a cycle of chords in $S_{lab}(\tilde{f})$. None of these configurations can happen, which proves $P_2$.

$P_3$ follows from a similar argument, again using the fact that the chords of $S_{lab}(\tilde{f})$ form a tree.

To prove $P_4$, let $a$ be an $n$-th root of unity and $F,F'$ two consecutive black faces around $a$ in clockwise order. Then there exist two chords $c$ (resp. $c'$) in their respective boundaries in $S_{lab}(\tilde{f})$ having $a$ as an endpoint, and corresponding to a transposition that appears in the cycle of $f$ coded by $F$ (resp. $F'$). By \cite[Theorem $2.2$]{GY02}, the labels of $c$ and $c'$ are sorted in increasing clockwise order around $a$. By definition of $\tilde{f}$, so are the labels of $F$ and $F'$.
\end{proof}

One can check in addition that, if a labelled colored lamination satisfies these four properties, then it also satisfies:

{\itshape
\begin{itemize}
\item[$P_5$:] Let $F$ be a white face of $S_{lab}(f)$. By $P_4$ , $F$ has exactly one arc in its boundary, of the form $\wideparen{(e^{-2i\pi a/n},e^{-2i\pi(a+1)/n})}$ for some $a \in \Z$. Then, the labels of its neighbouring black faces are sorted in decreasing clockwise order around $F$, starting from this unique arc.
\end{itemize}}

For $n,k \geq 1$, we now define $\kK_n^{(k)}$ the set of labelled colored lamination satisfying properties $P_1$ to $P_4$. In addition, we set $\kK_n = \cup_{1 \leq k \leq n-1} \kK_n^{(k)}$. Then the following holds:

\begin{thm}
\label{thm:bij1}
Let $n,k \geq 1$. The map 
$$\Phi_n^{(k)}: \left\{
\begin{array}{l}
 \kM_n^{(k)} \rightarrow \kK_n^{(k)} \\
  f \mapsto S_{lab}(f)
\end{array}
\right.$$
is a bijection.
\end{thm}

As a corollary, the map $\Phi_n: \kM_n \rightarrow \kK_n, f \mapsto S_{lab}(f)$ is also a bijection.

\begin{proof}
Let $f$ be an element of $\kM_n^{(k)}$. By Proposition \ref{prop:properties}, $S_{lab}(f)$ is an element of $\kK_n^{(k)}$, and therefore $\Phi_n^{(k)}$ is well defined. It is also clearly an injection. Let us now take $L \in \kK_n^{(k)}$. We prove that there exists a minimal factorization $f \in \kM_n^{(k)}$ such that $L=S_{lab}(f)$. To this end, for $i \in \llbracket 1,k \rrbracket$, denote by $F_i$ the black face of $L$ labelled $i$, and denote by $\ell(i)$ the number of chords in its boundary. These chords connect $\exp(-2i\pi a_1/n), \ldots, \exp(-2i\pi a_{\ell(i)}/n)$ so that $1 \leq a_1<a_2<\cdots<a_{\ell(i)} \leq n$. Let $c_i \coloneqq (a_1 \, a_2 \, \cdots \, a_{\ell(i)})$, and consider the product $\sigma \coloneqq c_1 c_2 \cdots c_k$.
By $P_4$ and $P_5$, it is clear that, for all $j \in \llbracket 1,n \rrbracket$, $\sigma(j)=j+1 \mod n$. Thus, $\sigma$ is the $n$-cycle. In addition, $f \coloneqq (c_1, c_2, \ldots, c_k)$ is an element of $\kM_n^{(k)}$, which satisfies $L = S_{lab}(f)$. The result follows.
\end{proof}

\subsection{Coding a minimal factorization by a bi-type tree}

We now construct from $S_{lab}(f)$ a tree $T(f)$. We then prove that, for $n,k \geq 1$, the map $$\Psi_n^{(k)}: f \rightarrow T(f)$$is a bijection from $\kM_n^{(k)}$ to the set of bi-type trees $\kU_n^{(k)}$. To this end, we rely on Theorem \ref{thm:bij1}, proving in fact that the mapping $\Psi^{(k)}_n \circ (\Phi_n^{(k)})^{-1}$ is bijective. As a corollary, $$\Psi_n: \kM_n \rightarrow \kU_n, f \rightarrow T(f)$$is also a bijection.

Let $n,k \geq 1$ and take $f := (\tau_1, ..., \tau_k) \in \kM_n^{(k)}$ a minimal factorization of the $n$-cycle. To $f$, we associate the graph $T(f)$, constructed as the \textit{dual graph} of $S_{lab}(f)$: black vertices correspond to black faces of $S_{lab}(f)$, while white vertices correspond to its white faces. Specifically, put a white vertex in each white face of $S_{lab}(f)$, and a black vertex in each of its black faces (including faces of perimeter $2$, which correspond to transpositions in $f$). Now, draw an edge between two vertices whenever the boundaries of the corresponding faces share a chord. Finally, root this graph at the white vertex corresponding to the white face whose boundary contains the arc $\wideparen{1,e^{-2i\pi/n}}$, and give to each black vertex of $T(f)$ the label of the corresponding face in $S_{lab}(f)$. See an example on Fig. \ref{fig:bijection}. 

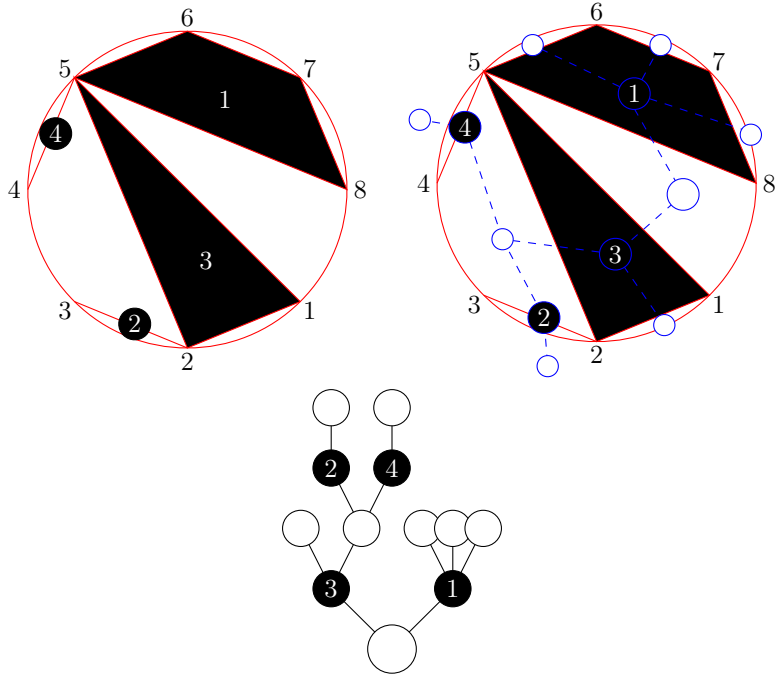
\begin{figure}
\center
\caption{An application of the bijection $\Psi_8$ to the minimal factorization $f \coloneqq (5678)(23)(125)(45) \in \kM_8$. Top-left: the colored lamination $S_{lab}(f)$. Top-right: the same lamination, with its dual tree $T(f)$ drawn in blue. The larger white vertex is its root. Bottom: the dual tree $T(f)$.}
\label{fig:bijection}
\begin{tabular}{c c}
\begin{tikzpicture}[scale=.7, every node/.style={scale=.9}, rotate=-45]
\draw[red] (0,0) circle (3);

\foreach \k in {1,...,8}
{\draw ({3.25*cos(360*(\k-1)/8)},{-3.25*sin(360*(\k-1)/8)}) node{\k};}

\draw[red,fill=black] ({3*cos(360*0/8)},{-3*sin(360*0/8)}) -- ({3*cos(360*4/8)},{-3*sin(360*4/8)}) -- ({3*cos(360/8)},{-3*sin(360/8)}) -- cycle;

\draw[red] ({3*cos(360*4/8)},{-3*sin(360*4/8)}) -- node[black]{4} ({3*cos(360*3/8)},{-3*sin(360*3/8)});

\draw[red] ({3*cos(360*1/8)},{-3*sin(360*1/8)}) -- node[black]{2} ({3*cos(360*2/8)},{-3*sin(360*2/8)});

\draw[red,fill=black] ({3*cos(360*7/8)},{-3*sin(360*7/8)}) -- ({3*cos(360*4/8)},{-3*sin(360*4/8)}) -- ({3*cos(360*5/8)},{-3*sin(360*5/8)}) -- ({3*cos(360*6/8)},{-3*sin(360*6/8)}) -- cycle ;

\draw[fill=black] (-.7,1.7) circle (.3) node[white]{1};

\draw[fill=black] (1.2,-.7) circle (.3) node[white]{3};

\draw[fill=black] (1.1,-2.5) circle (.3) node[white]{2};

\draw[fill=black] (-2.5,-1) circle (.3) node[white]{4};
\end{tikzpicture}
&
\begin{tikzpicture}[scale=.7, every node/.style={scale=.9}, rotate=-45]
\draw[red] (0,0) circle (3);

\foreach \k in {1,...,8}
{\draw ({3.25*cos(360*(\k-1)/8)},{-3.25*sin(360*(\k-1)/8)}) node{\k};}

\draw[red,fill=black] ({3*cos(360*0/8)},{-3*sin(360*0/8)}) -- ({3*cos(360*4/8)},{-3*sin(360*4/8)}) -- ({3*cos(360/8)},{-3*sin(360/8)}) -- cycle;

\draw[red] ({3*cos(360*4/8)},{-3*sin(360*4/8)}) -- node[black]{4} ({3*cos(360*3/8)},{-3*sin(360*3/8)});

\draw[red] ({3*cos(360*1/8)},{-3*sin(360*1/8)}) -- node[black]{2} ({3*cos(360*2/8)},{-3*sin(360*2/8)});

\draw[red,fill=black] ({3*cos(360*7/8)},{-3*sin(360*7/8)}) -- ({3*cos(360*4/8)},{-3*sin(360*4/8)}) -- ({3*cos(360*5/8)},{-3*sin(360*5/8)}) -- ({3*cos(360*6/8)},{-3*sin(360*6/8)}) -- cycle ;

\draw[blue,dashed]  (-.7,1.7) -- (-2.7,1) (-.7,1.7) -- (-1,2.7) (-.7,1.7) -- (1.5,2.8) (-.5,1.5) -- (1.3,1) -- (1.2,-.7) (1.2,-.7) -- (-.5,-2) -- (1.1,-2.5) -- (1.8,-3.1) (-.5,-2) -- (-2.5,-1) -- (-3.2,-1.5);

\draw[blue,dashed] (2.8,-1) -- (1.2,-.7);

\draw[blue,fill=black] (-.7,1.7) circle (.3) node[white]{1};

\draw[blue,fill=black] (1.2,-.7) circle (.3) node[white]{3};

\draw[blue,fill=black] (1.1,-2.5) circle (.3) node[white]{2};

\draw[blue,fill=black] (-2.5,-1) circle (.3) node[white]{4};

\draw[blue,fill=white] (1.3,1) circle (.3);

\draw[blue,fill=white] (-.5,-2) circle (.2);

\draw[blue,fill=white] (-2.7,1) circle (.2);

\draw[blue,fill=white] (1.4,2.7) circle (.2);

\draw[blue,fill=white] (-1,2.7) circle (.2);

\draw[blue,fill=white] (1.8,-3.1) circle (.2);

\draw[blue,fill=white] (-3.2,-1.5) circle (.2);

\draw[blue,fill=white] (2.8,-1) circle (.2);

\end{tikzpicture}
\\
\multicolumn{2}{c}{
\begin{tikzpicture}[scale=.8, every node/.style={scale=.9}]
\draw (-1,4) -- (-1,3) -- (-.5,2) -- (-1,1) -- (0,0) -- (1,1) -- (.5,2);
\draw (-.5,2) -- (0,3) -- (0,4);
\draw (-1.5,2) -- (-1,1);
\draw (1,2) -- (1,1) -- (1.5,2);
\draw[fill=white] (0,0) circle (.4);
\draw[fill=black] (-1,1) circle (.3) node[white]{3};
\draw[fill=black] (1,1) circle (.3) node[white]{1};
\draw[fill=white] (-1.5,2) circle (.3);
\draw[fill=white] (-.5,2) circle (.3);
\draw[fill=white] (.5,2) circle (.3);
\draw[fill=white] (1,2) circle (.3);
\draw[fill=white] (1.5,2) circle (.3);
\draw[fill=black] (-1,3) circle (.3) node[white]{2};
\draw[fill=black] (0,3) circle (.3) node[white]{4};
\draw[fill=white] (0,4) circle (.3);
\draw[fill=white] (-1,4) circle (.3);
\end{tikzpicture}
}
\end{tabular}
\end{figure}

\begin{lem}
\label{lem:lembij}
Let $n,k \geq 1$ and $f \in \kM_n^{(k)}$. Then $T(f) \in \kU_n^{(k)}$.
\end{lem}

\begin{proof}
Let us check all properties of $\kU_n^{(k)}$ one by one. First, $T(f)$ is clearly connected by construction. Moreover, since each chord splits the unit disk into two disjoint connected components, $T(f)$ is necessarily a tree.
By the property $P_1$, $T(f)$ has exactly $k$ black vertices and $n$ white vertices, and by $P_2$ all neighbours of a white vertex are black and conversely. The root of $T(f)$ is white by construction, and it is therefore a bi-type tree. In addition, a leaf of $T(f)$ has only one neighbour, and hence necessarily corresponds to a face that has an arc in its boundary. In particular this face is white, and thus all leaves of $T(f)$ are white. Finally, by $P_5$, the labels of the neighbours of each white face are sorted in decreasing clockwise order and the labels of the children of the root are decreasing from left to right. In conclusion, $T(f) \in \kU_n^{(k)}$.
\end{proof}

Remark notably that the degree of a black vertex in $T(f)$ corresponds to the length of the corresponding cycle in $f$. This mapping is a bijection, as stated in the following theorem.

\begin{thm}
\label{thm:nbijection}
For any $n,k \geq 1$, the map
\begin{align*}
\Psi_n^{(k)} : \, & \kM_n^{(k)} \rightarrow \kU_n^{(k)} \\
& f \mapsto T(f)
\end{align*}
is a bijection.
\end{thm}

Notice that, by Lemma \ref{lem:lembij}, $\Psi_n^{(k)}$ is well-defined from $\kM_n^{(k)}$ to $\kU_n^{(k)}$.

\begin{proof}[Proof of Theorem \ref{thm:nbijection}]
We rely here on \cite[Section $3$]{GY02}, where the Goulden-Yong bijection and its inverse are constructed. Let us construct as well the inverse of the map $\Psi_n^{(k)}$. Fixing a tree $T \in \kU_n^{(k)}$, we shall construct a lamination $L \coloneqq S_{lab}(f)$ associated to a minimal factorization $f$, such that $T=T(f)$. 
To this end, we define a way of exploring white vertices of the tree $T$, which we call its white exploration process. This process induces a way of labelling the white vertices, in the order in which they are explored. The white exploration process is defined the following way: we start from the root which receives label $1$, and explore the subtrees rooted in its white grandchildren from left to right. The rule is that, in order to explore a subtree of $T$ rooted in a white vertex $b$ whose black parent has label $a$, one first explores the subtrees rooted in a black child of $b$ with label $<a$ (if there are some, from left to right), then visits the vertex $b$, and finally explores the subtrees rooted in a black child of $b$ with label $>a$ if there are some, from left to right, starting from the leftmost of these subtrees. Exploring a subtree rooted in a black vertex juste consists in exploring the subtrees rooted in its white children, from left to right. An example is given on Fig. \ref{fig:treelamination}, top-right.

Let us now construct a colored lamination $L$ whose dual tree is exactly $T$. The idea (which is the main interest of this white exploration process) is that the white vertex labelled $k$ shall correspond to the white face whose boundary contains the arc $(\wideparen{e^{-2i(k-1)\pi/n},e^{-2ik\pi/n}})$ (so that the root corresponds to the arc $(\wideparen{1,e^{-2i\pi/n}})$). The colored lamination $L$ is constructed by drawing the faces that correspond to black vertices of $T$, and giving them the label of the associated vertices. See Fig. \ref{fig:treelamination} for an example. There is a unique way of drawing such a colored lamination. To see this, remark that there is only one way to draw the face corresponding to a black vertex whose children are all leaves, and that this drawing does not depend on the label of the white parent of this black vertex. Thus, there is only one way to draw all these faces from the leaves to the root, which gives $L$. Furthermore, $L$ belongs to $\kK_n^{(k)}$ by construction. Thus, by Theorem \ref{thm:bij1}, there exists $f \in \kM_n$ such that $L=S_{lab}(f)$. Hence, $f$ satisfies $T=T(f)$, and $\Psi_n^{(k)}$ is a bijection.
\end{proof}

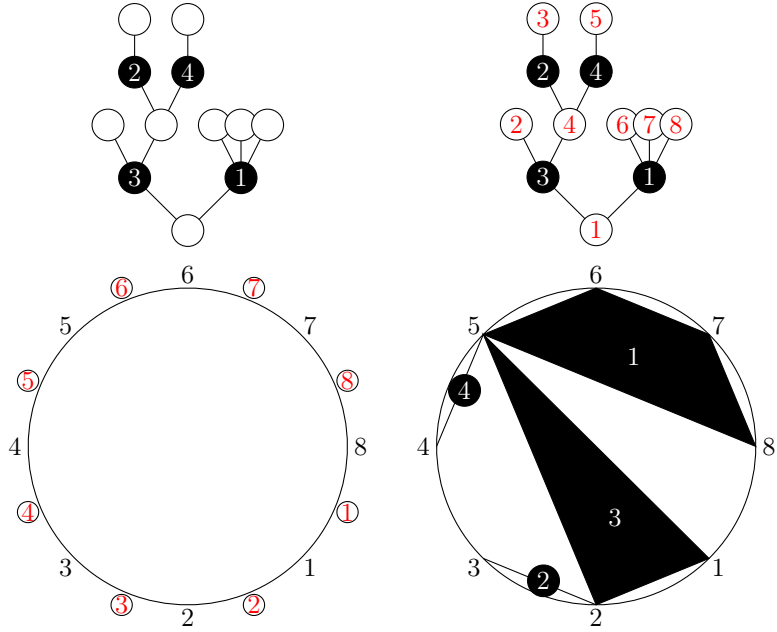
\begin{figure}
\center
\caption{An example of the inverse bijection $\Phi_8 \circ (\Psi_8)^{-1}$. Top-left: a tree $T \in \kU_8$. Top-right: the tree $T$ with labels on its white vertices, following the white exploration process. Bottom-left: the locations of the arcs corresponding to these white vertices, on the circle. Bottom-right: the associated colored lamination. We recover from this: $\Psi_8^{-1}(T)=(5678)(23)(125)(45)$.}
\label{fig:treelamination}
\begin{tabular}{c c}
\begin{tikzpicture}[scale=.7, every node/.style={scale=.9}]
\draw (-1,4) -- (-1,3) -- (-.5,2) -- (-1,1) -- (0,0) -- (1,1) -- (.5,2);
\draw (-.5,2) -- (0,3) -- (0,4);
\draw (-1.5,2) -- (-1,1);
\draw (1,2) -- (1,1) -- (1.5,2);
\draw[fill=white] (0,0) circle (.3);
\draw[fill=black] (-1,1) circle (.3) node[white]{3};
\draw[fill=black] (1,1) circle (.3) node[white]{1};
\draw[fill=white] (-1.5,2) circle (.3);
\draw[fill=white] (-.5,2) circle (.3);
\draw[fill=white] (.5,2) circle (.3);
\draw[fill=white] (1,2) circle (.3);
\draw[fill=white] (1.5,2) circle (.3);
\draw[fill=black] (-1,3) circle (.3) node[white]{2};
\draw[fill=black] (0,3) circle (.3) node[white]{4};
\draw[fill=white] (0,4) circle (.3);
\draw[fill=white] (-1,4) circle (.3);
\end{tikzpicture}
&
\begin{tikzpicture}[scale=.7, every node/.style={scale=.9}]
\draw (-1,4) -- (-1,3) -- (-.5,2) -- (-1,1) -- (0,0) -- (1,1) -- (.5,2);
\draw (-.5,2) -- (0,3) -- (0,4);
\draw (-1.5,2) -- (-1,1);
\draw (1,2) -- (1,1) -- (1.5,2);
\draw[fill=white] (0,0) circle (.3) node[red]{1};
\draw[fill=black] (-1,1) circle (.3) node[white]{3};
\draw[fill=black] (1,1) circle (.3) node[white]{1};
\draw[fill=white] (-1.5,2) circle (.3) node[red]{2};
\draw[fill=white] (-.5,2) circle (.3) node[red]{4};
\draw[fill=white] (.5,2) circle (.3) node[red]{6};
\draw[fill=white] (1,2) circle (.3) node[red]{7};
\draw[fill=white] (1.5,2) circle (.3) node[red]{8};
\draw[fill=black] (-1,3) circle (.3) node[white]{2};
\draw[fill=black] (0,3) circle (.3) node[white]{4};
\draw[fill=white] (0,4) circle (.3) node[red]{5};
\draw[fill=white] (-1,4) circle (.3) node[red]{3};
\end{tikzpicture}
\\
\begin{tikzpicture}[scale=.7, every node/.style={scale=.9}, rotate=-45]
\draw (0,0) circle (3);
\foreach \k in {1,...,8}
{\draw ({3.25*cos((360*(\k-2)+180)/8)},{-3.25*sin((360*(\k-2)+180)/8)}) circle (.2) node[red]{\k};}
\foreach \k in {1,...,8}
{\draw ({3.25*cos(360*(\k-1)/8)},{-3.25*sin(360*(\k-1)/8)}) node{\k};}
\end{tikzpicture}
&
\begin{tikzpicture}[scale=.7, every node/.style={scale=.9}, rotate=-45]
\draw (0,0) circle (3);
\foreach \k in {1,...,8}
{\draw ({3.25*cos(360*(\k-1)/8)},{-3.25*sin(360*(\k-1)/8)}) node{\k};}
\draw[fill=black] ({3*cos(360*0/8)},{-3*sin(360*0/8)}) -- ({3*cos(360*4/8)},{-3*sin(360*4/8)}) -- ({3*cos(360/8)},{-3*sin(360/8)}) -- cycle;
\draw ({3*cos(360*4/8)},{-3*sin(360*4/8)}) -- node[black]{4} ({3*cos(360*3/8)},{-3*sin(360*3/8)});
\draw ({3*cos(360*1/8)},{-3*sin(360*1/8)}) -- node[black]{2} ({3*cos(360*2/8)},{-3*sin(360*2/8)});
\draw[fill=black] ({3*cos(360*7/8)},{-3*sin(360*7/8)}) -- ({3*cos(360*4/8)},{-3*sin(360*4/8)}) -- ({3*cos(360*5/8)},{-3*sin(360*5/8)}) -- ({3*cos(360*6/8)},{-3*sin(360*6/8)}) -- cycle ;
\draw[fill=black] (-.7,1.7) circle (.3) node[white]{1};
\draw[fill=black] (1.2,-.7) circle (.3) node[white]{3};
\draw[fill=black] (1.1,-2.5) circle (.3) node[white]{2};
\draw[fill=black] (-2.5,-1) circle (.3) node[white]{4};
\end{tikzpicture}
\end{tabular}
\end{figure}

\subsection{Image of a random weighted minimal factorization}
\label{ssec:image}

We now investigate random weighted minimal factorizations of the $n$-cycle. Take $(w_i)_{i \geq 1}$ a weight sequence, and remember that $f_n^w$ is a minimal factorization of the $n$-cycle chosen proportionally to its weight: $\P (f_n^w=f) \propto \prod\limits_{i=1}^{k(f)} w_{\ell(\tau_i)-1}$, where $k(f)$ is the number of cycles in $f$. Then, it appears that the random tree $T(f_n^w)$ (which is the image of $f_n^w$ by $\Psi_n$) is a BTSG. In what follows, as in the previous section, $\mu_*$ denotes the Poisson distribution of parameter $1$.

\begin{thm}
\label{thm:btsg}
Let $w$ be a weight sequence. Then the plane tree $T(f_n^w)$, forgetting about the labels, has the law of the unlabelled version of $\cT_n^{(\mu_*,w)}$. In addition, this plane tree being fixed, the labelling of its black vertices is uniform among all labellings from $1$ to their total number $N^\bullet(T)$, satisfying the condition that the labels of all neighbours of a given white vertex are clockwise decreasing, and that the labels of the children of the root are decreasing from left to right.
\end{thm}

Notice that Lemma \ref{lem:intronbcycles} is an immediate corollary of Theorem \ref{thm:btsg} and Lemma \ref{lem:nbofcycles}: the number of cycles in a typical $w$-factorization of the $n$-cycle is of order $(1-\nu_0) \, n$. To see this, just remark that the number of cycles in a minimal factorization $F$ is exactly the number of black vertices in the tree $T(F)$. Theorem \ref{thm:btsg} also implies Proposition \ref{prop:largestcycle}: 

\begin{proof}[Proof of Proposition \ref{prop:largestcycle}]

It is clear, by the abovementioned bijection, that the maximum length of a cycle in a minimal factorization $F$ is the maximum degree of a black vertex in $T(F)$. Thus, by Theorem \ref{thm:btsg}, we need to study the maximum degree of a black vertex in the conditioned BTSG $\cT_n^{(\mu_*,w)}$. Let $\nu$ be as usual the critical equivalent of $w$. If $\nu$ has finite variance, then by the convergence of Theorem \ref{thm:cvcontour} the maximum degree of a white vertex in $\cT_n^{\circ,(\mu_*,w)}$ (that is, the maximum number of grandchildren of a white vertex in $\cT_n^{(\mu_*,w)}$) is $o(\sqrt{n})$ with high probability. Thus, the maximum degree of a black vertex in $\cT_n^{(\mu_*,w)}$ is necessarily $o(\sqrt{n})$ as well.

If $\alpha<2$, then for any $\epsilon>0$, again by the convergence of Theorem \ref{thm:cvcontour}, there exists $C>0$ such that, with probabiltiy larger than $1-\epsilon$, the maximum degree of a white vertex in $\cT_n^{\circ,(\mu_*,w)}$ is less than $C \tilde{B}_n$. Thus the maximum degree of a black vertex is also less than $C \tilde{B}_n$ with probability larger than $1-\epsilon$.
To prove the lower bound on this quantity, take $\epsilon>0$ and $\eta$ such that, with probability larger than $1-\epsilon$, there exists a white vertex $u$ in $\cT_n^{(\mu_*,w)}$ with at least $\eta\tilde{B}_n$ white grandchildren. Such an $\eta$ exists by Lemma \ref{lem:grossommets} (i). Then, by Lemma \ref{lem:pnu<2}, one can choose $\delta>0$ such that, with high probability, all white grandchildren of $u$ except at most $\tilde{B}_n n^{-\delta}$ of them have the same black parent $b(u)$. Hence, for $n$ large enough, with probability larger than $1-2\epsilon$, $b(u)$ has at least $\eta\tilde{B}_n/2$ children. The result follows. 
\end{proof}

\begin{proof}[Proof of Theorem \ref{thm:btsg}]
Take $T$ a bi-type tree with $n$ white vertices, whose leaves are all white. Then, the number of labellings of the black vertices that are clockwise decreasing around each white vertex is exactly 
$$n! \left(\prod\limits_{x \in T, \, x \,\text{ white}} k_x(T)!\right)^{-1}.$$ 
Furthermore, by Theorem \ref{thm:nbijection}, given such a labelling of $T$, exactly one minimal factorization is coded by the tree $T$ labelled this way, and this factorization has weight $ \prod\limits_{y \in T, y \,\text{ black}} \wn_{k_y(T)}$. Finally, 
$$\P \left( T(f_n^w) = T \right) \propto \prod\limits_{x \in T, x \,\text{ white}} \frac{1}{k_x(T)!} \prod\limits_{y \in T, y \,\text{ black}} w_{k_y(T)}.$$ The result follows.
\end{proof}

\paragraph*{Equivalence of weighted minimal factorizations}

An other way to understand Lemma \ref{lem:equivalence}, in the light of the bijection $\Psi_n$, is to remark that the weight sequence $w$ is not uniquely defined by the distribution of $f_n^w$. The following lemma characterizes the families of weight sequences that give birth to the same random minimal factorization:

\begin{lem}[Equivalent sequences]
Let $w$ be a weight sequence and $s >0$. Define $w^{(s)}$ the weight sequence verifying, for any $i \geq 1$, $w^{(s)}_i = w_i s^i$. Then, for any $n \geq 1$, $f_n^{w^{(s)}}$ has the same distribution as $f_n^w$.
\end{lem}

\begin{proof}
Take $n, k \geq 1$. For any $f \coloneqq (\tau_1,\ldots,\tau_k) \in \kM_n$, we have:
\begin{align*}
W_{w^{(s)}}(f) &= \prod_{i=1}^k w_{\ell(\tau_i)-1} s^{\ell(\tau_i)-1} = s^{\sum_{i=1}^k \left(\ell(\tau_i)-1 \right)} \times \prod_{i=1}^k w_{\ell(\tau_i)-1} = s^{n-1} W_w(f)
\end{align*}
by definition of $\kM_n$. Recalling that we have defined $Y_{n,v} = \sum_{f \in \kM_n} W_v(f)$ for any weight sequence $v$, this implies that $Y_{n,w^{(s)}} = s^{n-1} Y_{n,w}$ and therefore that, for $f \in \kM_n$,
\begin{align*}
\P\left( f_n^{w^{(s)}} = f \right) = \frac{W_{w^{(s)}}(f)}{Y_{n,w^{(s)}}} = \frac{W_w(f)}{Y_{n,w}} = \P\left( f_n^w = f \right).
\end{align*}
The result follows.
\end{proof}

Remember that we say that two weight sequences $w,w'$ are equivalent if, for all $n \geq 1$, $\cT_n^{(\mu_*,w)} \overset{(d)}{=} \cT_n^{(\mu_*,w')}$. One can check, the same way as in Lemma \ref{lem:monotypeequiv}, that the sequences $w^{(s)}, s>0$ are the only sequences equivalent to $w$. Indeed, recalling the notation of Lemma \ref{lem:equivalence}, the white weight sequence $w^\circ$ is the same (equal to $\mu_*$) in both trees $\cT_n^{(\mu_*,w)}$ and $\cT_n^{(\mu_*,w')}$, and thus the parameters $p$ and $q$ shall be equal to $1$. Since we impose the condition $qr=1$, the only parameter that is allowed to vary is $s$. It is therefore natural to obtain a family indexed by only one parameter $s$.

\subsection{Relation between the colored lamination-valued processes and the tree coding a minimal factorization}

In order to prove Theorem \ref{thm:mainthmsection4}, we start by proving that, when $w$ is of stable type (for $\alpha<2$ or for $\nu$ with finite variance), the colored lamination-valued constructed from $f_n^w$ is close with high probability to the black process of the associated tree $T(f_n^w)$. 

\begin{thm}
\label{thm:relationprocesstree}
Let $\alpha \in (1,2]$, $w$ be a factorization of $\alpha$-stable type and $\nu$ be its critical equivalent. Then, if $\alpha<2$ or if $\nu$ has finite variance, in $\D([0,+\infty],\bCL(\oD))$, in probability, as $n \rightarrow \infty$,

\begin{align*}
d_{Sk}\left( \left(S_{u}(f_n^w) \right)_{u \in [0,\infty]}, \left(\bL_{u}^\bullet \left(T(f_n^w)\right) \right)_{u \in [0,\infty]} \right) \overset{\P}{\rightarrow} 0,
\end{align*}
where $d_{Sk}$ denotes the Skorokhod distance on $\D([0,\infty],\bCL(\oD))$.
\end{thm}

To prove it, for $g: \Z_+ \rightarrow \R_+$, denote by $Z_n^g$ the set of minimal factorizations $f$ of the $n$-cycle satisfying two conditions: (i) $H(T(f)) \leq g(n)$;
(ii) there exists a constant $A > 0$ such that, for any white vertex $u$ of $T(f)$, taking the notation of the proof of Lemma \ref{lem:proportions}, for any $1 \leq i \leq 3$, we have 
$$\big| |G_i(u, T^\circ(f))| - A |G_i(u,T(f))| \big| \leq g(n).$$ 

For $f$ a factorization of the $n$-cycle into $k$ cycles for $1 \leq j \leq k$, denote by $F_j$ the face of $S_{lab}(f)$ labelled $j$, and by $u_j$ the black vertex of $T(f)$ labelled $j$. Recall that $F_{u_j}(T(f))$ denotes the face coding $u_j$ in the black process of $T(f)$. Then the following holds:

\begin{lem}
\label{lem:facesproches}
Let $g: \Z_+ \rightarrow \R_+$. Then there exists a constant $C>0$ such that, uniformly in $j$, as $n \rightarrow \infty$, uniformly for $f \in Z_n^g$,
\begin{align*}
d_H\left( F_j, F_{u_j}\left(T(f) \right) \right) \leq C g(n)/n,
\end{align*}
\end{lem}

This straightforwardly implies Theorem \ref{thm:relationprocesstree}: 

\begin{proof}[Proof of Theorem \ref{thm:relationprocesstree}]
By Lemma \ref{lem:proportions} and Theorem \ref{thm:cvcontour}, there exists $g: \Z_+ \rightarrow \R_+$ such that $g(n)=o(n)$ and $f_n^w \in Z_n^g$ with high probability. Thus, with high probability, jointly for all $j \leq N^\bullet(T(f_n^w))$, $d_H(F_j, F_{u_j}(T(f_n^w)) \rightarrow 0$ as $n \rightarrow \infty$. This implies Theorem \ref{thm:relationprocesstree}.
\end{proof} 

\begin{proof}[Proof of Lemma \ref{lem:facesproches}]
This proof is a straight adaptation of \cite[Lemma $4.4$]{The19}, which investigates the case of a minimal factorization into transpositions. Let $f \coloneqq (\tau_1, \ldots, \tau_k) \in \kM_n^{(k)}$, and fix $1 \leq j \leq k$. Denote by $\ell_j$ the length of the cycle $\tau_j$, and write $\tau_j$ as $(a_1 \cdots a_{\ell_j})$, with $1 \leq a_1 < \cdots < a_{\ell_j} \leq n$. By definition, the face $F_j$ connects the points $e^{-2i \pi a_1/n}, \ldots, e^{-2i \pi a_{\ell_j}/n} \in \bS^1$.
The lengths of the arcs delimited by $1$ and these $\ell_j$ points are therefore, in clockwise order, $2\pi a_1/n, 2\pi(a_2-a_1)/n, \ldots, 2\pi(a_{\ell_j} - a_{\ell_j-1})/n, 2\pi (n-a_{\ell_j})/n$.

Now, let us consider the vertex $u_j$. It induces a partition of the set of vertices of $T(f)$ into $\ell_j+1$ subsets: the set $S_1$ of vertices visited by the contour function before the first visit of $u_j$, the set $S_2$ of vertices visited between the first and the second visit of $u_j$, etc. up to $S_{\ell_j+1}$, the set of vertices visited for the first time after the last visit of $u_j$. Let us denote by $N^\circ(S_i)$ the number of vertices of $S_i$ that are white, and remark that the interval between two consecutive visits of $u_j$ exactly corresponds to the exploration of a subtree rooted in a white child of $u_j$. By the second point in definition of $Z_n^g$, it is clear that $|N^\circ(S_1)-na_1|\leq g(n), |N^\circ(S_2)-n(a_2-a_1)| \leq g(n), \ldots, |N^\circ(S_{\ell_j})-n(a_{\ell_j}-a_{\ell_j-1})| \leq g(n), |N^\circ(S_{\ell_j+1})-(n-a_{\ell_j})|\leq g(n)$.

In order to control the locations of the associated faces in the unit disk, we follow the proof of \cite[Lemma $4.4$]{The19}: remark that,  for all $i$, the white vertices of $S_i$ exactly correspond to white faces of $S_{lab}(f)$ whose boundary contains an arc between $e^{-2i\pi a_{i-1}/n}$ and $e^{-2i\pi a_{i}/n}$, except for ancestors of $u_j$ which may correspond to arcs either between $\exp(-2i\pi a_{\ell_j}/n)$ and $1$, or between $1$ and $e^{-2i\pi a_1/n}$. By the first point in the definition of $Z_n^g$, $u_j$ has at most $g(n)$ ancestors. 
This implies that $F_j$ and $F_{u_j}(T(f))$ are at distance less than $4 \pi g(n)/n$, jointly for all $j \leq N^\bullet(T(f_n^w))$.
\end{proof}

\subsection{A shuffling operation}
\label{ssec:shuffling}

By Theorem \ref{thm:relationprocesstree}, in order to prove Theorem \ref{thm:mainthmsection4}, we now only need to study the process $\left(\bL_u^\bullet \left(T(f_n^w)\right) \right)_{u \in [0,\infty]}$. The main obstacle in this study is the constraint on the labelling of black vertices in $T(f_n^w)$ (recall that the labels are clockwise decreasing around each white vertex, and decreasing from left to right around the root). To get rid of this constraint, we define a shuffling operation on the vertices of a bi-type tree, adapted from \cite[Section $4.4$]{The19}.

\begin{defi}
\label{def}
Fix $n,k \geq 1$. Let $T$ be a plane bi-type tree with $n$ white vertices and $k$ black vertices labelled from $1$ to $k$, and let $K \in \Z_+$. We define the shuffled tree $T^{(K)}$ as follows: starting from the root of $T$, we perform one of the following two operations on each white vertex of $T$. For consistency, we put the constraint that the operation shall be performed on a white vertex before being performed on its grandchildren.
\begin{itemize}
\item Operation $1$: for a white vertex such that the labels of its black children are all $> K$, we uniformly shuffle these labels (without touching the corresponding subtrees). See Fig. \ref{fig:shuffling} (a).
\item Operation $2$: for a white vertex such that at least one of its black children has a label $\leq K$, we uniformly shuffle these labelled children and keep the subtrees on top of each of them. See Fig. \ref{fig:shuffling} (b).
\end{itemize}
\end{defi}

\begin{figure}[!h]
\center
\caption{Examples of the shuffling operation. The operation is different in both cases, since in the second case the vertex labelled $9$ has a child with label $4 \leq K$.}
\label{fig:shuffling}
\begin{tabular}{c c c}
\begin{tikzpicture}
\draw[white] (0,0) -- (1,0);
\end{tikzpicture}
&
\begin{tikzpicture}[scale=.7, every node/.style={scale=.7}]
\draw (-2,3) -- (0,1.5) -- (-1,3);
\draw (0,3) -- (0,1.5) -- (1,3);
\draw (0,1.5) -- (2,3);
\draw[fill=black] (-2,3.3) circle (.3) node[white]{7};
\draw[fill=black] (-1,3.3) circle (.3) node[white]{4};
\draw[fill=black] (0,3.3) circle (.3) node[white]{28};
\draw[fill=black] (2,3.3) circle (.3) node[white]{12};
\draw[fill=black] (1,3.3) circle (.3) node[white]{16};
\draw (0,1.2) circle (.3);
\draw (-2,3.6) -- (-2.3,5) -- (-1.7,5) -- cycle;
\draw (-1,3.6) -- (-1.3,4.5) -- (-.7,4.5) -- cycle;
\draw (0,3.6) -- (-.3,5.5) -- (.3,5.5) -- cycle;
\draw (1,3.6) -- (.7,4) -- (1.3,4) -- cycle;
\draw (2,3.6) -- (1.7,6) -- (2.3,6) -- cycle;
\draw (-2,5.5) node{$T_1$};
\draw (-1,5) node{$T_2$};
\draw (0,6) node{$T_3$};
\draw (1,4.5) node{$T_4$};
\draw (2,6.5) node{$T_5$};
\end{tikzpicture}
&
\begin{tikzpicture}[scale=.7, every node/.style={scale=.7}]
\draw (-2,3) -- (0,1.5) -- (-1,3);
\draw (0,3) -- (0,1.5) -- (1,3);
\draw (0,1.5) -- (2,3);
\draw[fill=black] (-2,3.3) circle (.3) node[white]{28};
\draw[fill=black] (-1,3.3) circle (.3) node[white]{4};
\draw[fill=black] (0,3.3) circle (.3) node[white]{16};
\draw[fill=black] (2,3.3) circle (.3) node[white]{7};
\draw[fill=black] (1,3.3) circle (.3) node[white]{12};
\draw (0,1.2) circle (.3);
\draw (-2,3.6) -- (-2.3,5) -- (-1.7,5) -- cycle;
\draw (-1,3.6) -- (-1.3,4.5) -- (-.7,4.5) -- cycle;
\draw (0,3.6) -- (-.3,5.5) -- (.3,5.5) -- cycle;
\draw (1,3.6) -- (.7,4) -- (1.3,4) -- cycle;
\draw (2,3.6) -- (1.7,6) -- (2.3,6) -- cycle;
\draw (-2,5.5) node{$T_1$};
\draw (-1,5) node{$T_2$};
\draw (0,6) node{$T_3$};
\draw (1,4.5) node{$T_4$};
\draw (2,6.5) node{$T_5$};
\end{tikzpicture}
\\
\multicolumn{3}{c}{(a) Shuffling of a labelled plane tree when $K=3$: Operation $1$ is performed.}\\

\begin{tikzpicture}
\draw[white] (0,0) -- (1,0);
\end{tikzpicture}
&
\begin{tikzpicture}[scale=.7, every node/.style={scale=.7}]
\draw (-2,3) -- (0,1.5) -- (-1,3);
\draw (0,3) -- (0,1.5) -- (1,3);
\draw (0,1.5) -- (2,3);
\draw[fill=black] (-2,3.3) circle (.3) node[white]{7};
\draw[fill=black] (-1,3.3) circle (.3) node[white]{4};
\draw[fill=black] (0,3.3) circle (.3) node[white]{28};
\draw[fill=black] (2,3.3) circle (.3) node[white]{12};
\draw[fill=black] (1,3.3) circle (.3) node[white]{16};
\draw (0,1.2) circle (.3);
\draw (-2,3.6) -- (-2.3,5) -- (-1.7,5) -- cycle;
\draw (-1,3.6) -- (-1.3,4.5) -- (-.7,4.5) -- cycle;
\draw (0,3.6) -- (-.3,5.5) -- (.3,5.5) -- cycle;
\draw (1,3.6) -- (.7,4) -- (1.3,4) -- cycle;
\draw (2,3.6) -- (1.7,6) -- (2.3,6) -- cycle;
\draw (-2,5.5) node{$T_1$};
\draw (-1,5) node{$T_2$};
\draw (0,6) node{$T_3$};
\draw (1,4.5) node{$T_4$};
\draw (2,6.5) node{$T_5$};
\end{tikzpicture}
&
\begin{tikzpicture}[scale=.7, every node/.style={scale=.7}]
\draw (-2,3) -- (0,1.5) -- (-1,3);
\draw (0,3) -- (0,1.5) -- (1,3);
\draw (0,1.5) -- (2,3);
\draw[fill=black] (-2,3.3) circle (.3) node[white]{12};
\draw[fill=black] (-1,3.3) circle (.3) node[white]{28};
\draw[fill=black] (0,3.3) circle (.3) node[white]{4};
\draw[fill=black] (2,3.3) circle (.3) node[white]{16};
\draw[fill=black] (1,3.3) circle (.3) node[white]{7};
\draw (0,1.2) circle (.3);
\draw (1,3.6) -- (.7,5) -- (1.3,5) -- cycle;
\draw (0,3.6) -- (-.3,4.5) -- (.3,4.5) -- cycle;
\draw (-1,3.6) -- (-1.3,5.5) -- (-.7,5.5) -- cycle;
\draw (2,3.6) -- (2.3,4) -- (1.7,4) -- cycle;
\draw (-2,3.6) -- (-1.7,6) -- (-2.3,6) -- cycle;
\draw (1,5.3) node{$T_1$};
\draw (0,4.8) node{$T_2$};
\draw (-1,5.8) node{$T_3$};
\draw (2,4.3) node{$T_4$};
\draw (-2,6.3) node{$T_5$};
\end{tikzpicture}
\\
\multicolumn{3}{c}{(b) Shuffling of the same tree when $K=5$: Operation $2$ is performed.}
\end{tabular}
\end{figure}

The main interest of this shuffling operation is that, for any $K$, $T^{(K)}(f_n^w)$ has the law of $\cT_n$, which we recall is defined as the $(\mu_*, w)$-BTSG tree conditioned to have $n$ vertices, whose black vertices are labelled uniformly at random from $1$ to $N^\bullet(\cT_n)$. Furthermore, for a well-chosen sequence $(K_n)_{n \geq 1}$ of values of $K$ (depending on $n$), the black processes associated to $T(f_n^w)$ and $T^{(K_n)}(f_n^w)$ are asymptotically close. The choice of this sequence is important. Indeed, if $K_n=0$, we uniformly shuffle the labels of the children of each white vertex, and in particular the labels of large faces may be given to small ones, which completely changes the structure of the colored process. On the other hand, if $K_n=n$, then the subtrees on top of the children of branching points may be swapped, so that the structure of the underlying tree is changed.

\begin{lem}
\label{lem:chordconfigtree}
\begin{itemize}
\item[(i)] For any $K_n \geq 0$, any weight sequence $w$, the black vertices of the tree $T^{(K_n)}(f_n^w)$ are labelled uniformly at random:
\begin{align*}
T^{(K_n)}(f_n^w) \overset{(d)}{=} \cT_n.
\end{align*}

\item[(ii)] If $\alpha<2$ or if the critical equivalent $\nu$ of $w$ has finite variance, there exists a sequence $(K_n)_{n \geq 1}$ such that the black process of the initial tree is close in probability to the black process of the tree with shuffled vertices:
\begin{align*}
d_{Sk}\left(\left( \bL_{c \tilde{B}_n}^\bullet \left(T(f_n^w)\right) \right)_{c \in [0,\infty]}, \left( \bL_{c \tilde{B}_n}^\bullet \left(T^{(K_n)}(f_n^w)\right) \right)_{c \in [0,\infty]} \right) \underset{n \rightarrow \infty}{\overset{\P}{\rightarrow}} 0,
\end{align*}
where $\tilde{B}_n$ satisfies \eqref{eq:tbn} for $\nu$ and we recall that $d_{Sk}$ denotes the Skorokhod distance on $\D([0,\infty],\bCL(\oD))$.
\end{itemize}
\end{lem}

Remark that Theorem \ref{thm:mainthmsection4} is an easy corollary of this lemma.
In addition, notice that Lemma \ref{lem:chordconfigtree} (i) is immediate. Indeed, Operations $1$ and $2$ do not change the law of the underlying unlabelled tree, while the labelling of the black vertices after the shuffling operations is uniform.

The proof of Lemma \ref{lem:chordconfigtree} (ii) is the object of the next subsection. 

\subsection{Proof of the technical lemma \ref{lem:chordconfigtree} (ii)}
\label{ssec:proof2}

In order to prove Lemma \ref{lem:chordconfigtree} (ii), we need to quantify the distance between the locations of faces with the same label in both labelled colored laminations $\bL_\infty^\bullet(T(f_n^w))$ and $\bL_\infty^\bullet(T^{(K_n)}(f_n^w))$. For any $i \in \llbracket 1,N^\bullet(T(f_n^w)) \rrbracket$, denote by $F_i$ (resp. $F'_i$) the face corresponding to the vertex labelled $i$ in $T(f_n^w)$ (resp. $T^{(K_n)}(f_n^w)$). 

One can remark that, for $i \leq K_n$, if a vertex has label $i \leq K_n$, then Operation $2$ is performed on its parent and hence it keeps its subtree on top of it, whose size determines the lengths of the chords of the face $F_i$. Therefore, for all $i$, the lengths of the chords in the boundaries of $F_i$ and $F'_i$ are the same (which means that $F_i$ and $F'_i$ are the same up to rotation). In particular, if the boundary of $F_i$ contains no chord of length larger than $\epsilon$, so does the boundary of $F'_i$, and conversely. Thus, we only have to focus on the locations of the faces that have large chords in their boundary, which correspond to vertices of the tree that are the root of a large subtree.

The idea is the following: when one shuffles vertices following Definition \ref{def:epsilonshuffling}, the location of a face $F_i$ associated to a vertex $u$ with given label $i\leq K_n$ is impacted only by the subset of ancestors of $u$ in $T(f_n^w)$ on which Operation $2$ is performed; indeed, performing Operation $1$ on a vertex that is not $u$ does not change the underlying unlabelled tree. We first investigate the maximum possible displacement of $F_i$ induced by Operation $2$ on ancestors of $u$ that are not $\delta n$-nodes (Lemma \ref{lem:MPDCMPD}), at $\delta>0$ fixed. Then, we show that the way Operation $2$ is performed on ancestors that are not $\delta n$-nodes does not affect much the colored lamination-valued process. In the finite variance case, it is possible to choose $K_n$ in such a way that, with high probability, Operation $2$ is never performed on any $\delta n$-node. When $\alpha<2$, we take $K_n=n$ so that Operation $2$ is performed on all vertices, and we prove that this still does not affect much the colored lamination.

Fix $\epsilon >0$. Let us fix $\delta \in (0,\epsilon)$, and define, for $T$ a monotype tree with $n$ vertices and $u \in T$ an $\epsilon n$-node, the $\delta$-maximum possible displacement of $u$ as:

\begin{align*}
MPD_\delta(u,T) = \frac{1}{n} \sum_{v \in A_u^\delta(T)} \sum_{\substack{w \in K_v(T) \\ w \notin A_u(T)}} |\theta_w(T)|,
\end{align*}
where $A_u^\delta(T)$ denotes the set of ancestors of $u$ in $T$ that are not $\delta n$-nodes. Recall that $K_v(T)$ denotes the set of children of $v$ in $T$, and $A_u(T)$ the set of ancestors of $u$ in $T$. The quantity $MPD_\delta(u,T)$ takes into account the sizes of the subtrees rooted in children of ancestors of $u$ that are not $\delta n$-nodes. See Fig. \ref{fig:MPDdelta} for an example.

\begin{lem}
\label{lem:MPDCMPD}
Almost surely, jointly with the convergence of Theorem \ref{thm:cvcontour}:
\begin{equation*}
\lim_{n \rightarrow \infty} \sup_{\substack{z \in T^\circ(f_n^w) \\ \left|\theta_z\left(T^\circ(f_n^w)\right)\right|\geq \epsilon n}} MPD_\delta \left(z,T^\circ(f_n^w)\right) \underset{\delta \downarrow 0}{\longrightarrow} 0 .
\end{equation*}
\end{lem}
Note that here, $\epsilon$ is fixed while $\delta<\epsilon$ goes to $0$.
Roughly speaking, whether Operation $2$ is performed or not on vertices that are not $\delta n$-nodes does not impact much the associated colored lamination-valued process, as the locations of large faces do not change much.

\begin{figure}[!h]
\caption{A representation of the quantity $MPD_\delta(u,T)$, in a given tree $T$, for some vertex $u \in T$. The hatched part is $\theta_u(T)$. $MPD_\delta(u,T)$ is the sum of the sizes of the three plain subtrees on the left of the ancestral line of $u$, divided by $|T|$. Indeed, $a_1$ and $a_2$ are elements of $A_u^\delta(T)$. The dashed subtrees are not counted in $MPD_\delta(u,T)$, because they are rooted in  children of $\delta n$-nodes (namely, $\emptyset$ and $a_3$).} 
\label{fig:MPDdelta}
\center
\begin{tikzpicture}
\draw (0,0) -- (0,4);
\draw[pattern=north west lines] plot [smooth cycle] coordinates  {(0,4) (-.5,5) (.5,5)};
\draw[fill] (0,0) circle (.03);
\draw[fill] (0,1) circle (.025);
\draw[fill] (0,2) circle (.025);
\draw[fill] (0,3) circle (.025);
\draw[fill] (0,4) circle (.025);
\draw (.5,1) -- (0,0) -- (1,1);
\draw (0,2) -- (-1,3);
\draw (0,3) -- (1,4) (-1,4) -- (0,3) -- (-.5,4);
\draw (-.5,2) -- (0,1) -- (-1,2);
\draw[fill=blue!20] plot [smooth cycle] coordinates  {(-1,3) (-1.2,3.3) (-.8,3.3)};
\draw[dashed,fill=red!20] plot [smooth cycle] coordinates  {(.5,1) (.8,1.5) (.2,1.5)};
\draw[dashed,fill=red!20] plot [smooth cycle] coordinates  {(1,1) (.9,1.5) (1.2,1.5)};
\draw[dashed,fill=red!20] plot [smooth cycle] coordinates  {(-1,4) (-1.2,4.3) (-.8,4.3)};
\draw[dashed,fill=red!20] plot [smooth cycle] coordinates  {(-.5,4) (-.3,4.3) (-.7,4.3)};
\draw[dashed,fill=red!20] plot [smooth cycle] coordinates  {(1,4) (1.3,4.4) (.7,4.4)};
\draw[fill=blue!20] plot [smooth cycle] coordinates  {(-.5,2) (-.6,2.1) (-.4,2.1)};
\draw[fill=blue!20] plot [smooth cycle] coordinates  {(-1,2) (-1.1,2.1) (-.9,2.1)};
\draw (.2,3.8) node{$u$};
\draw (.3,2.8) node{$a_3$};
\draw (.3,1.8) node{$a_2$};
\draw (-.2,.8) node{$a_1$};
\draw (-.2,-.2) node{$\emptyset$};
\end{tikzpicture}
\end{figure}
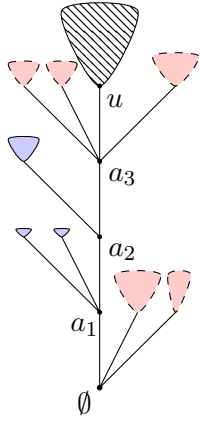

Let us immediately check that this indeed implies Lemma \ref{lem:chordconfigtree} (ii). We prove it in two different ways, depending whether $\alpha<2$ or $\nu$ has finite variance.

\begin{proof}[Proof of Lemma \ref{lem:chordconfigtree} (ii) when $\alpha<2$.]

In this case, let us take $K_n=n$ for all $n \geq 1$, which corresponds to the worst case in which Operation $2$ is performed on each white vertex of $T(f_n^w)$. We prove that doing this does not change much the underlying lamination, and Lemma \ref{lem:chordconfigtree} (ii) therefore follows by Lemma \ref{lem:MPDCMPD}. Fix $\delta>0$. The idea is that, roughly speaking, almost all grandchildren of any white $\delta n$-node $u$ of $T^\circ(f_n^w)$ have the same black parent, so that there is only one black child of $u$ that is the root of a big subtree. Thus, if this child always keeps its subtree on top of it, the colored lamination does not change much.

To state things properly, fix $q>0$ and take $u \in T^\circ(f_n^w)$ a white $\delta n$-node. Then, by Lemma \ref{lem:grossommets} (ii), there exists $\eta>0$ such that, with probability larger that $1-q$, $u$ has at least $\eta \tilde{B}_n$ white grandchildren in $T(f_n^w)$. On this event, by Lemma \ref{lem:pnu<2}, with high probability, there exists $r>0$ such that all of its white grandchildren, except at most $\tilde{B}_n n^{-r}$ of them, have the same black parent $b_u$. This implies that, with high probability, the sum of the sizes of the subtrees rooted in one of these (at most) $\tilde{B}_n n^{-r}$ grandchildren is $o(n)$. Indeed, for any $K \geq 1$, with high probability the $K$ white grandchildren of $u$ with the largest subtrees on top of them are all children of $b_u$. Thus, with high probability, shuffling the siblings of $b_u$ in any way does only change the location of the large face corresponding to $b_u$ by a distance $o(n)$. Since there is only a finite number of $\delta n$-nodes in $T^\circ(f_n^w)$, by letting $q$ go to $0$, the result follows: uniformly for $1 \leq i \leq N^\bullet(T(f_n^w))$, using Lemma \ref{lem:MPDCMPD}:
\begin{align*}
d_H\left(F_i, F'_i \right)=o(n).
\end{align*}
\end{proof}

\begin{proof}[Proof of Lemma \ref{lem:chordconfigtree} (ii) when $\nu$ has finite variance.]

In this case, we exactly follow the proof of \cite[Lemma $4.6$]{The19}.
To begin with, let us explain how to choose the sequence $(K_n)_{n \geq 1}$. Remark that, by the convergence of the Lukasiewicz path of $T^\circ(f_n^w)$ renormalized by a factor $\sqrt{n}$ (Theorem \ref{thm:cvcontour}) towards the normalized Brownian excursion which is almost surely continuous, with high probability the maximum degree of a vertex in the tree is $o(\sqrt{n})$. Furthermore, there are at most $\delta^{-1}$ $\delta n$-nodes in the tree, which proves that the number $N_{\delta n}(T(f_n^w))$ of children of $\delta n$-nodes is $o(\sqrt{n})$ with high probability. Thus, one can choose $(K_n^{(\delta)})_{n \geq 1}$ such that $K_n^{(\delta)} \gg \sqrt{n}$ and $N_{\delta n}(T(f_n^w)) \times K_n^{(\delta)} = o(n)$. Thus, by diagonal extraction, one can choose $(K_n)_{n \geq 1}$ such that $K_n \gg \sqrt{n}$ and, for any $\delta>0$, $N_{\delta n}(T(f_n^w)) \times K_n = o(n)$. Let us take such a sequence. We prove Lemma \ref{lem:chordconfigtree} (ii) in two steps. On one hand, $K_n$ is small enough, so that:
\begin{equation}
\label{eq:lessthann}
d_{Sk}\left( \left(\bL^\bullet_{c \tilde{B}_n}(T(f_n^w))\right)_{c \leq K_n/\tilde{B}_n}, \left(\bL^\bullet_{c \tilde{B}_n}(T^{(K_n)}(f_n^w))\right)_{c \leq K_n/\tilde{B}_n}\right) \overset{\P}{\rightarrow} 0.
\end{equation} 

On the other hand, $K_n$ is large enough, so that: 
\begin{equation}
\label{eq:morethansqrtn}
d_{Sk}\left( \left(\bL^\bullet_{c\tilde{B}_n}(T(f_n^w))\right)_{c \in [K_n/\tilde{B}_n,\infty] }, \left(\bL^\bullet_{c\tilde{B}_n}(T^{(K_n)}(f_n^w))\right)_{c \in [K_n/\tilde{B}_n,\infty]}\right) \overset{\P}{\rightarrow} 0.
\end{equation}

To show \eqref{eq:lessthann}, we prove that, for any $\delta>0$ fixed, for this choice of $(K_n)_{n \geq 1}$, with high probability Operation $2$ is not performed on any $\delta n$-node. For this, let $p_n$ be the probability that there exists an $\delta n$-node of $T^\circ(f_n^w)$ with a child of label $\leq K_n$. Then, conditionally to the values of $N^\bullet(T(f_n^w))$ and $N_{\delta n}(T(f_n^w))$,
\begin{align*}
p_n = 1-\frac{\binom{N^\bullet(T(f_n^w))-N_{\delta n}(T(f_n^w))}{K_n}}{\binom{N^\bullet(T(f_n^w))}{K_n}} \leq 1-\left( 1-\frac{N_{\delta n}(T(f_n^w))}{N^\bullet(T(f_n^w)) - K_n} \right)^{K_n}.
\end{align*}
Now, by Lemma \ref{lem:nbofcycles} and Lemma \ref{lem:chordconfigtree} (i), with high probability, $N^{\bullet}(T(f_n^w)) \overset{(d)}{=} N^\bullet(\cT_n) \geq (1-\nu_0) \, n/2$. Thus, with high probability,
\begin{align*}
p_n \leq 1-\left( 1 - \frac{N_{\delta n}(T(f_n^w))}{(1-\nu_0) \, n/2-K_n} \right)^{K_n} \sim \frac{2 K_n N_{\delta n}(T(f_n^w))}{(1-\nu_0) n}
\end{align*}
which converges in probability to $0$. Hence, with high probability Operation $2$ is not performed on any $\delta n$-node. The faces appearing in the colored lamination-valued processes until time $K_n$ therefore do not code $\delta n$-nodes, and \eqref{eq:lessthann} follows by Lemma \ref{lem:MPDCMPD}.

To prove \eqref{eq:morethansqrtn}, remark that, by Theorem \ref{thm:convergenceblackwhite} (i), the red part of $(\bL_{K_n}^\bullet(T^{(K_n)}(f_n^w))_{n \geq 1}$ converges in distribution towards the Brownian triangulation, which is maximum in the set of laminations of the disk. In addition, by \eqref{eq:lessthann}, $d_H(\bL^\bullet_{K_n}(T(f_n^w)), \bL^\bullet_{K_n}(T^{(K_n)}(f_n^w))) \rightarrow 0$ with high probability as $n \rightarrow \infty$, and both red parts converge to the same Brownian triangulation. Therefore, the only thing that we need to prove is that, for any $\epsilon>0$, faces corresponding to the same white $\epsilon n$-node in both white reduced trees have the same color. This is clear, since the fact that the two large subtrees rooted in white grandchildren of a given white vertex have the same black parent or not is not affected by Operation $1$ nor by Operation $2$. The result follows.
\end{proof}

We finally prove Lemma \ref{lem:MPDCMPD}:

\begin{proof}[Proof of Lemma \ref{lem:MPDCMPD}]
To study the asymptotic behaviour of $\sup_{u \in T^\circ(f_n^w)} MPD_\delta(u,T^\circ(f_n^w))$, we define the continuous analogue of this quantity on the stable tree $\cT^{(\alpha)}$. 
Recall that $H^{(\alpha)}$ is seen as the contour function of $\cT^{(\alpha)}$, in the following sense: there exists a coupling between $\cT^{(\alpha)}$ and $H^{(\alpha)}$ such that, if a particle explores the tree starting from its root from left to right as in the discrete case, so that the exploration ends at time $1$, then almost surely the distance between the particle and the root as time passes is coded by $H^{(\alpha)}$.
Therefore, for any $u \in \cT^{(\alpha)}$, one can define the subtree of $\cT^{(\alpha)}$ rooted at $u$, $\theta_u(\cT^{(\alpha)})$, as the set of points visited between the first and last visits of $u$ by $H^{(\alpha)}$. 

Mimicking the notations of Section \ref{ssec:deftree}, for any $x \in \cT^{(\alpha)}$, define $g(x)$ (resp. $d(x)$) the first (resp. last) time at which the vertex $x$ is visited by $H^{(\alpha)}$. Then we simply define the size of the subtree $\theta_u(\cT^{(\alpha)})$ as $|\theta_u(\cT^{(\alpha)})|=d(u)-g(u)$. In particular, if one denotes by $\emptyset$ the root of $\cT^{(\alpha)}$, $|\theta_\emptyset(\cT^{(\alpha)})|=1$.

Let us also define the analogue of $\delta n$-nodes in this continuous setting, which we will call $\delta$-nodes of $\cT^{(\alpha)}$: for $\delta>0$, we say that $u$ is a $\delta$-node of $\cT^{(\alpha)}$ if there exist $0 \leq a_1(u) < a_2(u) < a_3(u) \leq 1$ such that $H^{(\alpha)}$ visits $u$ at times $a_1(u), a_2(u), a_3(u)$ and in addition $a_3(u)-a_2(u) \geq \delta, a_2(u)-a_1(u) \geq \delta$.

Now, for $\delta>0$ and $u \in \cT^{(\alpha)}$, define $A^\delta(u,\cT^{(\alpha)})$ the set of ancestors of $u$ in $\cT^{(\alpha)}$ (i.e. elements of the tree that are visited before the first visit of $u$ or after the last visit of $u$ by $H^{(\alpha)}$) that are not $\delta$-nodes. Finally, set:

\begin{align*}
CMPD_\delta \left(u,\cT^{(\alpha)}\right) \coloneqq \sum_{v \in A^\delta(u,\cT^{(\alpha)})} \tilde{h}_{u,v}\left(\cT^{(\alpha)}\right)
\end{align*}
where, in words, we define $\tilde{h}_{u,v}$ as follows: removing the vertex $v$ from the tree splits it into several connected components. We sum the sizes of all of these components which do not contain the root of $\cT^{(\alpha)}$, nor the vertex $u$. Rigorously, one can define it as: 
\begin{align*}
\tilde{h}_{u,v} \coloneqq \left(d(v)-g(v)\right) - \left(d_u(v)-g_u(v)\right),
\end{align*}
where $g_u(v)=\sup\{s<g(u), H^{(\alpha)}_s=H^{(\alpha)}_{g(v)}\}$, $d_u(v)=\inf\{s>d(u), H^{(\alpha)}_s=H^{(\alpha)}_{g(v)}\}$ are the consecutive times at which $H^{(\alpha)}$ visits $v$ such that $u$ is visited inbetween, corresponding to the branch starting from $v$ that contains $u$.

Assume by Skorokhod theorem that the convergence of Theorem \ref{thm:cvcontour}, stating that the renormalized white reduced tree $T^\circ(f_n^w)$ converges to $\cT^{(\alpha)}$, holds almost surely. Assume that there exists $\eta>0$ and an increasing extraction $\phi:\N^* \rightarrow \N^*$ such that, for all $n \geq 1$, we can find a vertex $v_n \in T(f_{\phi(n)}^w)$ such that $|\theta_{v_n}(T(f_n^w))|\geq \epsilon n$, for which $MPD_{1/n}(v_n, T(f_{\phi(n)}^w)) \geq \eta$. Using the fact that $\cT^{(\alpha)}$ is compact, up to extraction, $v_n$ converges to some vertex $v_\infty \in \cT^{(\alpha)}$ satisfying $|\theta_{v_\infty}(T(f_n^w))|\geq \epsilon$. $v_\infty$ should in addition satisfy $CMPD_0(v_\infty, \cT^{(\alpha)}) \geq \eta$, which is impossible. The result follows. 

Notice that the condition that the subtree rooted in the vertex $v_n$ has size at least $\epsilon n$ is mandatory. Indeed, otherwise, it may happen that $v_n$ belongs to a 'small' branch of the tree, but converges to a point $v_\infty$ of $\cT^{(\alpha)}$ with a large subtree on top of it.
\end{proof}

\bibliographystyle{abbrv}
\bibliography{biblidual}

\begin{thebibliography}{10}

\bibitem{ald}
D.~Aldous.
\newblock Triangulating the circle, at random.
\newblock {\em The American Mathematical Monthly}, 101:223--233, 1994.

\bibitem{AP98}
D.~Aldous and J.~Pitman.
\newblock The standard additive coalescent.
\newblock {\em Annals of Probability}, pages 1703--1726, 1998.

\bibitem{AHRV07}
O.~Angel, A.~E. Holroyd, D.~Romik, and B.~Vir{\'a}g.
\newblock Random sorting networks.
\newblock {\em Advances in Mathematics}, 215(2):839--868, 2007.

\bibitem{Bet17}
J.~Bettinelli.
\newblock Convergence of uniform noncrossing partitions toward the brownian
  triangulation.
\newblock {\em arXiv preprint arXiv:1711.04872}, 2017.

\bibitem{Bia96}
P.~Biane.
\newblock Minimal factorizations of a cycle and central multiplicative
  functions on the infinite symmetric group.
\newblock {\em journal of combinatorial theory, Series A}, 76(2):197--212,
  1996.

\bibitem{BGT89}
N.~H. Bingham, C.~M. Goldie, and J.~L. Teugels.
\newblock {\em Regular variation}, volume~27.
\newblock Cambridge university press, 1989.

\bibitem{BDG04}
J.~Bouttier, P.~Di~Francesco, and E.~Guitter.
\newblock Planar maps as labeled mobiles.
\newblock {\em The Electronic Journal of Combinatorics}, 11(1):R69, 2004.

\bibitem{CK14}
N.~Curien and I.~Kortchemski.
\newblock Random non-crossing plane configurations: A conditioned galton-watson
  tree approach.
\newblock {\em Random Structures \& Algorithms}, 45(2):236--260, 2014.

\bibitem{Dau19}
D.~Dauvergne.
\newblock The archimedean limit of random sorting networks.
\newblock {\em arXiv preprint arXiv:1802.08934}, 2019.

\bibitem{Den59}
J.~D{\'e}nes.
\newblock The representation of a permutation as the product of a minimal
  number of transpositions and its connection with the theory of graphs.
\newblock {\em Publ. Math. Inst. Hungar. Acad. Sci}, 4:63--70, 1959.

\bibitem{DL15}
R.~R. Du and F.~Liu.
\newblock Factorizations of cycles and multi-noded rooted trees.
\newblock {\em Graphs and Combinatorics}, 31(3):551--575, 2015.

\bibitem{Duq03}
T.~Duquesne et~al.
\newblock A limit theorem for the contour process of condidtioned
  galton--watson trees.
\newblock {\em The Annals of Probability}, 31(2):996--1027, 2003.

\bibitem{Duq08}
T.~Duquesne et~al.
\newblock An elementary proof of {H}awkes's conjecture on {G}alton-{W}atson
  trees.
\newblock {\em Electronic Communications in Probability}, 14:151--164, 2009.

\bibitem{DLG02}
T.~Duquesne and J.-F. {Le Gall}.
\newblock {\em Random Trees, {L}évy Processes and Spatial Branching
  Processes}.
\newblock 2002.

\bibitem{Fel08}
W.~Feller.
\newblock {\em An introduction to probability theory and its applications},
  volume~2.
\newblock John Wiley \& Sons, 2008.

\bibitem{FK17}
V.~F\'eray and I.~Kortchemski.
\newblock The geometry of random minimal factorizations of a long cycle via
  biconditioned bitype random trees.
\newblock {\em Annales Henri Lebesgue}, 1:149--226, 2018.

\bibitem{FK18}
V.~F{\'e}ray and I.~Kortchemski.
\newblock Trajectories in random minimal transposition factorizations.
\newblock {\em arXiv preprint arXiv:1810.07586}, 2018.

\bibitem{GY02}
I.~Goulden and A.~Yong.
\newblock Tree-like properties of cycle factorizations.
\newblock {\em Journal of Combinatorial Theory, Series A}, 98(1):106--117,
  2002.

\bibitem{GP93}
I.~P. Goulden and S.~Pepper.
\newblock Labelled trees and factorizations of a cycle into transpositions.
\newblock {\em Discrete Mathematics}, 113(1-3):263--268, 1993.

\bibitem{IL71}
I.~A. Ibragimov and Y.~V. Linnik.
\newblock {\em Independent and stationary sequences of random variables}.
\newblock Wolters-Noordhoff Publishing, Groningen, 1971.
\newblock With a supplementary chapter by I. A. Ibragimov and V. V. Petrov,
  Translation from the Russian edited by J. F. C. Kingman.

\bibitem{Jan11}
S.~Janson.
\newblock Stable distributions.
\newblock {\em arXiv preprint arXiv:1112.0220}, 2011.

\bibitem{Jan12}
S.~Janson.
\newblock Simply generated trees, conditioned {G}alton--{W}atson trees, random
  allocations and condensation.
\newblock {\em Probability Surveys}, 9:103--252, 2012.

\bibitem{Kal02}
O.~Kallenberg.
\newblock {\em Foundations of {M}odern {P}robability}.
\newblock 2002.

\bibitem{Ken75}
D.~P. Kennedy.
\newblock The galton-watson process conditioned on the total progeny.
\newblock {\em Journal of Applied Probability}, 12(4):800--806, 1975.

\bibitem{Kes86}
H.~Kesten.
\newblock Subdiffusive behavior of random walk on a random cluster.
\newblock {\em Ann. Inst. H. Poincar\'e Probab. Statist.}, 22(4):425--487,
  1986.

\bibitem{Kor12}
I.~Kortchemski.
\newblock Invariance principles for galton--watson trees conditioned on the
  number of leaves.
\newblock {\em Stochastic Processes and Their Applications}, 122(9):3126--3172,
  2012.

\bibitem{Kor14}
I.~Kortchemski.
\newblock Random stable laminations of the disk.
\newblock {\em The Annals of Probability}, 42(2):725--759, 2014.

\bibitem{KM17}
I.~Kortchemski and C.~Marzouk.
\newblock Simply generated non-crossing partitions.
\newblock {\em Combinatorics, Probability and Computing}, 26(4):560--592, 2017.

\bibitem{LGM10}
J.-F. Le~Gall and G.~Miermont.
\newblock On the scaling limit of random planar maps with large faces.
\newblock In {\em XVIth International Congress On Mathematical Physics: (With
  DVD-ROM)}, pages 470--474. World Scientific, 2010.

\bibitem{LGP08}
J.-F. Le~Gall and F.~Paulin.
\newblock Scaling limits of bipartite planar maps are homeomorphic to the
  2-sphere.
\newblock {\em Geometric and Functional Analysis}, 18(3):893--918, 2008.

\bibitem{MaMo03}
J.-F. Marckert and A.~Mokkadem.
\newblock The depth first processes of {G}alton-{W}atson trees converge to the
  same {B}rownian excursion.
\newblock {\em The Annals of probability}, 31:1655--1678, 2003.

\bibitem{MM78}
A.~Meir and J.~W. Moon.
\newblock On the altitude of nodes in random trees.
\newblock {\em Canadian journal of Mathematics}, 30(5):997--1015, 1978.

\bibitem{Mos89}
P.~Moszkowski.
\newblock A solution to a problem of {D}{\'e}nes: a bijection between trees and
  factorizations of cyclic permutations.
\newblock {\em European Journal of Combinatorics}, 10(1):13--16, 1989.

\bibitem{Nev86}
J.~Neveu.
\newblock Arbres et processus de {G}alton-{W}atson.
\newblock {\em Ann. Inst. H. Poincar\'e Probab. Statist.}, 22(2):199--207,
  1986.

\bibitem{The18}
P.~Th{\'e}venin.
\newblock Vertices with fixed outdegrees in large {G}alton--{W}atson trees.
\newblock {\em arXiv preprint arXiv:1812.07365}, 2018.

\bibitem{The19}
P.~Th{\'e}venin.
\newblock A geometric representation of fragmentation processes on stable
  trees.
\newblock {\em arXiv preprint arXiv:1910.04508}, 2019.

\end{thebibliography}
\end{document}